\documentclass[pdflatex,sn-mathphys-num]{sn-jnl}% Math and Physical Sciences Numbered Reference Style
\usepackage{graphicx,verbatim,lineno,titletoc}
\usepackage{amssymb,mathrsfs}
\usepackage{amsmath}
\usepackage{adjustbox}
\usepackage{calc}
\usepackage{array}
\usepackage[font=small]{caption}
\usepackage[T1]{fontenc}
\usepackage[english]{babel}
\usepackage{appendix}
\usepackage{enumerate}
\usepackage[english]{babel}
\usepackage{multirow}
\usepackage{float}
\usepackage{enumitem}
\usepackage[numbers]{natbib}
\usepackage[all]{xy}
\usepackage{tikz}
\usetikzlibrary{shapes.multipart,patterns,arrows}
\usepackage{hyperref}
\hypersetup{colorlinks,linkcolor={red},citecolor={olive},urlcolor={red}}
\usepackage{mathtools}
\mathtoolsset{showonlyrefs}
\usepackage{algorithm}   % For the algorithm environment
\usepackage{algpseudocode} % For the algorithmic environment

\renewcommand{\descriptionlabel}[1]{%
	\hspace\labelsep \upshape\bfseries #1%
}

\newcolumntype{M}[1]{>{\centering\arraybackslash}m{#1}}
\newcolumntype{N}{@{}m{0pt}@{}}

\def\.{\hskip.06cm}

\newcommand{\vertiii}[1]{{\left\vert\kern-0.25ex\left\vert\kern-0.25ex\left\vert #1 
		\right\vert\kern-0.25ex\right\vert\kern-0.25ex\right\vert}}

\def\param{\boldsymbol{{\mathbf{\theta}}}}
\def\Param{\boldsymbol{\Theta}}

\newcommand{\htheta}{\hat{\theta}}
\newcommand{\hparam}{{\hat{\param}}}
\newcommand{\hParam}{{\widehat{\Param}}}

\newcommand{\PR}{\mathbb{P}}
\def\bH{\textbf{H}}

\def\param{\boldsymbol{\theta}}
\def\Param{\boldsymbol{\Theta}}

\def\<{\langle}
\def\>{\rangle}

\def\0{{\mathbf 0}}

\def\.{\hskip.06cm}

\def\dist{{\text {\rm dist} }}

\def\P{\mathbb{P}}

\def\R{\mathbb{R}}

\def\E{\mathbb{E}}
\def\P{\mathbb{P}}

\def\eps{\varepsilon}

\def\i{\textbf{i}}

\usepackage{graphicx}

\makeatletter
\newcommand*\rel@kern[1]{\kern#1\dimexpr\macc@kerna}
\newcommand*\widebar[1]{%
	\begingroup
	\def\mathaccent##1##2{%
		\rel@kern{0.8}%
		\overline{\rel@kern{-0.8}\macc@nucleus\rel@kern{0.2}}%
		\rel@kern{-0.2}%
	}%
	\macc@depth\@ne
	\let\math@bgroup\@empty \let\math@egroup\macc@set@skewchar
	\mathsurround\z@ \frozen@everymath{\mathgroup\macc@group\relax}%
	\macc@set@skewchar\relax
	\let\mathaccentV\macc@nested@a
	\macc@nested@a\relax111{#1}%
	\endgroup
}
\DeclareMathOperator*{\argmax}{arg\,max}

\theoremstyle{thmstyleone}%
\newtheorem{theorem}{Theorem}%  meant for continuous numbers
\newtheorem{prop}[theorem]{Proposition}% 
\newtheorem{lemma}[theorem]{Lemma}
\newtheorem{corollary}[theorem]{Corollary}

\newtheorem{claim}[theorem]{Claim}

\theoremstyle{thmstyletwo}%
\newtheorem{remark}{Remark}%

\theoremstyle{thmstylethree}%
\newtheorem{definition}{Definition}%
\newtheorem{assumption}[definition]{Assumption}%

\makeatletter
\def\blfootnote{\xdef\@thefnmark{}\@footnotetext}
\makeatother

\begin{document}

\title[Likelihood landscape of binary latent model on a tree]{Likelihood landscape of binary latent model on a tree%\blfootnote{A preliminary version of this work, containing the finite-sample landscape analysis and the statistical/computational guarantees, appeared in the Proceedings of the 42nd International Conference on Machine Learning (ICML 2025) \cite{CLR:25}.}
}


\author[1]{\fnm{David} \sur{Clancy Jr.}}\email{djclancyjr@gmail.com}

\author[1]{\fnm{Hanbaek} \sur{Lyu}}\email{hlyu@math.wisc.edu}

\author*[1]{\fnm{Sebastien} \sur{Roch}}\email{roch@math.wisc.edu}

\affil[1]{\orgdiv{Department of Mathematics}, \orgname{University of Wisconsin}, \orgaddress{\street{480 Lincoln Dr.}, \city{Madison}, \postcode{53706}, \state{WI}, \country{USA}}}


\abstract{ 

We investigate the optimization landscape of maximum likelihood estimation (MLE) for the Cavender-Farris-Neyman (CFN) model, a two-state latent tree model fundamental to statistical phylogenetics and the ferromagnetic Ising model. Although the log-likelihood function is non-concave and may admit many critical points, simple coordinate maximization algorithms are remarkably effective in practice. We provide the first theoretical justification for this success. We prove that sufficiently deep inside the reconstruction regime, the population log-likelihood is strongly concave and smooth within a box around the true parameter, whose size is independent of tree topology and number of leaves. This fundamental result implies that the empirical landscape shares  these regularity properties  with high probability given polynomial sample complexity and also that coordinate maximization converges exponentially fast to an $O(1/\sqrt{m})$-consistent MLE. Our analysis centers on a novel decay property of the population Hessian: diagonal entries remain large while off-diagonal entries decay exponentially with graph distance. These results provide rigorous theoretical evidence for the efficacy of likelihood-based tree inference and suggest broader principles for latent variable models.

}




\maketitle

	\section{Introduction}
	\label{sec:Introduction}

    Maximum likelihood estimation (MLE) for estimating an unknown model parameter is a fundamental technique in statistics and machine learning. In this framework, one considers a parametric probabilistic model $(\mathbb{P}_{\param}; \param \in \Param \subset \mathbb{R}^d)$ and a dataset $x_1, \dots, x_m$ assumed to be i.i.d. observations from $\mathbb{P}_{\theta^*}$ for some unknown $\param^* \in \Param$. The goal is to find an estimate $\hparam_{\textup{MLE}} \in \Param$ such that
\begin{align}
\hat{\param}_{\textup{MLE}} &\in \arg \max_{\param \in \Theta} \,\, \left[ \ell(\param; x_1, \dots, x_m) := \frac{1}{m} \sum_{j=1}^m \log \mathbb{P}_{\param}(x_j) \right], \label{eq:generic_log_likelihood}
\end{align} 
thereby maximizing the likelihood of the observed data.

Classical theory \cite{Cramer.46, wald1949note} establishes that whenever the population landscape (i.e., the $m \to \infty$ limit)
\begin{align*}
    \mathbb{E}[\ell(\param)] = \mathbb{E}_{X \sim \mathbb{P}_{\param^*}} \log \mathbb{P}_{\param}(X),
\end{align*}
is concave and maximized at a unique point coinciding with the true parameter $\param^*$, then $\hat{\param}_{\textup{MLE}}$ serves as a consistent estimator of $\param^*$. Much of the rigorous theoretical foundation of MLE assumes that the properties of this population landscape extend to the empirical landscape $\ell(\param; x_1, \dots, x_m)$. This extension relies on the ability to approximate $\mathbb{E}[\ell(\param)]$ by $\ell(\param; x_1, \dots, x_m)$ such that optimizing the latter approximately recovers $\param^*$ \cite{LB.19}. However, many natural MLE problems—such as those arising from mixture models \cite{redner1984mixture, Murphy.12}, errors-in-variables regression \cite{vanhuffel1991total}, or blind deconvolution \cite{ahmed2014blind}—exhibit highly non-concave likelihood landscapes in both the population and the empirical levels, which significantly complicates the analysis of MLEs.

In this paper, we study a fundamental parameter estimation problem for two-state latent models on broadcasting trees known as \textit{branch-length estimation} \cite{guindon2003simple} under the Cavender-Farris-Neyman (CFN) model \cite{neyman1971molecular, farris1973probability, cavender1978taxonomy}. This model is widely used to study molecular evolution along phylogenetic trees. Broadly speaking, the goal is to estimate the flip probabilities in noisy channels along the edges of a broadcasting tree, given signals observed only at the tips (leaves) of the tree. While arising in phylogenetics, such models and their various inference problems have applications in theoretical computer science, signal processing, and statistical physics \cite{mossel:22}. 

The associated MLE problem is challenging. Marginalizing hidden variables makes the likelihood non-concave, and even for closely related latent-tree models the sharpest known landscape guarantees are population-level: for Gaussian latent tree models, \cite{dagan2022em} proves that the true parameter is the unique non-trivial stationary point of the population log-likelihood, where ``non-trivial'' excludes boundary edge correlations equal to 0 or 1. Thus, in that setting, spurious interior stationary points are absent. By contrast, empirical likelihoods for discrete phylogenetic models can have substantial algebraic complexity. A standard measure of this complexity is the maximum-likelihood degree (ML degree): for a generic data vector, it is the number of complex solutions to the likelihood critical equations, divided by the generic fiber cardinality of the model parametrization, i.e., the number of parameter values that generically map to the same model distribution. Garc\'ia Puente, Garrote-L\'opez, and Shehu \cite{GGMS.24} compute ML degrees for small group-based phylogenetic models, including the CFN model, and show that these counts can already be large; importantly, ML degree counts complex generic-data critical points, not real or biologically admissible stationary points, and it does not address the population landscape. The real likelihood landscape is nevertheless known to be nontrivial even on small trees: Steel \cite{Steel.94} gives a four-leaf CFN example with two global maximizers on the boundary of the parameter space, and related work exhibits multiple optima or local maxima for phylogenetic likelihoods \cite{rogers1999multiple, chor2000multiple}. These difficulties motivate our semi-global approach: rather than characterizing the entire non-concave landscape, we identify a dimension-independent neighborhood of the true parameter on which the population--and, with enough samples, empirical--likelihood is well conditioned.

Despite these theoretical hurdles, likelihood maximization has proven remarkably effective in practice. For instance, Guindon and Gascuel \cite{guindon2003simple, guindon2010new} developed PHYML, a coordinate-ascent algorithm that performs well empirically even with a few coordinate updates. Other widely used likelihood-based methods include RAxML \cite{stamatakis2014raxml} and IQ-TREE \cite{nguyen2015iq}. However, a rigorous theoretical explanation for the success of these methods in the discrete setting has remained elusive. Recent advances in statistical estimation theory have emphasized the importance of analyzing the geometric structure of likelihood landscapes to bridge this gap between practice and theory \cite{ma2017implicit, chen2019inference, chen19nonconvex}.

In this work, we establish a condition, regardless of the tree's size, under which the population likelihood landscape becomes smooth and well-behaved within a box centered at the true parameter. This fundamental likelihood landscape result, combined with our recent work on its finite-sample applications \cite{CLR:25}, shows that a simple coordinate maximization algorithm converges quickly and reliably to the correct estimate if initialized in this box. These results provide the first theoretical guarantees for the optimization methods commonly used in phylogenetic and machine learning analyses of discrete latent tree models. Importantly, our semi-global landscape analysis is distinct from standard local analyses, as it establishes regularity within a region whose size is independent of the problem dimensions.

    \begin{figure}[h!]
        \centering
        \includegraphics[width=0.7\linewidth]{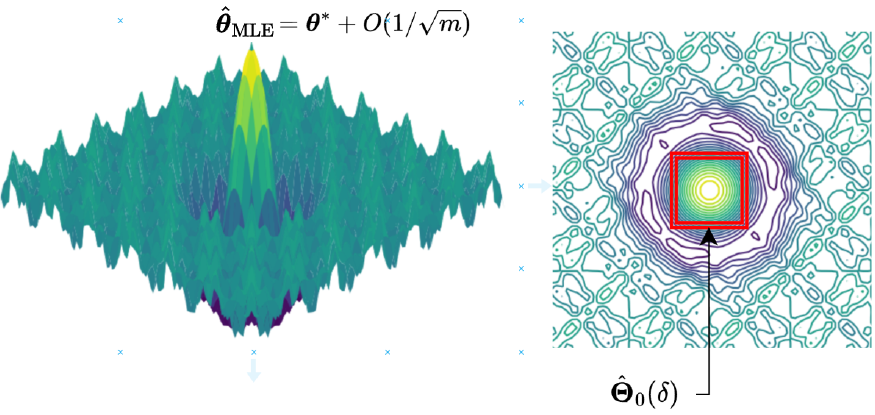}
        \caption{A cartoon depiction of non-concave 2D likelihood landscape (left) and its contour plot (right). Thm. \ref{thm:MLE_landscape_sample} and \ref{thm:MLE_estimation_error} asserts that the empirical likelihood landscape has a well-conditioned strongly concave landscape over ``universal box'' $\hParam_{0}(\delta)$ of size order 1 
        around the true parameter $\param^{*}$. Thm. \ref{thm:MLE_opt_alg} asserts that coordinate maximization initialized in the box converges to the MLE $\hparam_{\textup{MLE}}$ in $O(1)$ iterations. }
        \label{fig:enter-label}
    \end{figure}

    \subsection{The CFN model}
    
Let $T = (V,E)$ be an unrooted tree where all vertices have degree 1 or 3. The vertices with degree 1, called leaves and denoted by the set $L$, represent the observed variables, while the remaining vertices represent latent variables. Let $n = |L|$ denote the number of leaves. A counting argument shows that such a tree must have $2n-2$ total vertices. 

At each vertex $v\in V$, we have a binary state $\sigma_v\in\{\pm 1\}$ which starts from its stationary distribution and evolves along the branches, i.e., edges, of the tree according to a reversible Markov chain with unknown transition matrices. Specifically, there is  an unknown ``true'' parameter vector $\param^{*} = (\theta^*_e;e\in E)$ such that the transition matrices $(P^*_e;e\in E)$ are given by
	\begin{equation*}
		P^*_e = \begin{bmatrix}
			1-p^*_e & p^*_e\\
			p^*_e&1-p^*_e
		\end{bmatrix}=\begin{bmatrix}
			\frac{1+\theta^*_e}{2} & \frac{1-\theta^*_e}{2}\\
			\frac{1-\theta^*_e}{2} &\frac{1+\theta^*_e}{2}
		\end{bmatrix}.
	\end{equation*}
	Here, $p^*_e = \PR_{\param^{*}}(\sigma_u \neq \sigma_v)$ represents the transition probability along edge $e = \{u,v\}$, and the stationary distribution is uniform: $\PR_{\param^{*}}(\sigma_v = 1) = \PR_{\param^{*}} (\sigma_v = -1) = \frac{1}{2}$. The relationship between the transition probabilities and edge parameters is given by
\begin{equation}\label{eqn:peDef}
p^*_e = \frac{1-\theta^*_e}{2},\qquad \theta^*_e = 1-2p^*_e.
\end{equation} 

While $\theta_e$ can theoretically take values in $[-1,1]$, we focus on the regime $\param^{*}\in [0,1]^{E}$ where neighboring variables are positively correlated (i.e., the ``ferromagnetic regime''). This constraint appears naturally in applications, including in phylogenetics.

The central problem we study is the recovery of the true parameter vector $\param^{*}$ from repeated, independent observations of the states at the leaves of the tree. Formally, let $\sigma^{(1)},\dots,\sigma^{(m)}$ be independent samples from the model described above with true parameter $\PR_{\param^{*}}$. We only observe the leaf states, denoted by $\sigma^{(j)}|_{L} = (\sigma_v^{(j)}; v\in L)$ for $j=1,\dots,m$. Our goal is to estimate $\param^{*}$ from these partial observations.
This estimation problem arises naturally in statistical phylogenetics, where it is known as ``branch length estimation''. In that context, $p^*_e$ can be interpreted as a monotone function of the evolutionary distance, or branch length, along edge $e$ assuming a constant-rate mutation process\footnote{More precisely, $\theta^*_{e}=e^{-2l^*_{e}}$ and $p^*_{e}=\frac{1-\exp{(-2l^*_{e})}}{2}$ where $l^*_e$ represents the ``evolutionary distance'' (i.e., time scaled by mutation rate). We work directly with the $\theta^*_e$ parameterization throughout this paper.}.
See, e.g.,~\cite{felsenstein2004inferring,yang2014molecular} for background on statistical phylogenetics.

	\subsection{The maximum likelihood landscape}
	
	It has long been an open problem
	whether standard coordinate ascent algorithms for maximum likelihood can solve the branch length estimation problem. Namely, the log-likelihood of the leaf observations $\sigma^{(1)},\dots,\sigma^{(m)}$ under the model with parameter $\hparam$ is given by 
	\begin{align}\label{eq:def_log_likelihood}
	\ell(\hparam):=\ell(\hparam; \sigma^{(1)}|_{L},\dotsm, \sigma^{(m)}|_{L}) := \frac{1}{m} \sum_{j=1}^m  \log \PR_{\hparam}(\sigma_v = \sigma^{(j)}_v , \,\,\forall v\in L).
	\end{align} 
	We then seek to find the maximum likelihood estimator (MLE) as 
	\begin{align}\label{eq:MLE_def}
		\hparam_{\textup{MLE}} \in \argmax_{\hparam\in [0,1]^{E}}  \,\,  \ell(\hparam; \sigma^{(1)}|_{L},\dotsm, \sigma^{(m)}|_{L}).
	\end{align}

    As with many latent variable or mixture models \cite{balakrishnan2017statistical,dagan2022em}, the objective function in \eqref{eq:MLE_def} is non-concave and may admit multiple critical points for generic data. See also~\cite{GGMS.24} and references therein. Steel~\cite{Steel.94} provided an explicit example, reproduced in Figure \ref{fig:fail} below, where there are multiple global maximizers to \eqref{eq:MLE_def}. The likelihood function $L(\hparam;\sigma^{(1)},\sigma^{(2)})$ for Figure \ref{fig:fail} has exactly two maximum values
	\begin{equation*}
		\param^1 = (0,1,0,1,1)\qquad\textup{and}\qquad \param^2 = (1,0,1,0,1),
	\end{equation*} 
	which are both on the boundary $\partial [0,1]^E$. 
	\begin{figure}[h!]
		\centering
		\includegraphics[width=0.45 \linewidth]{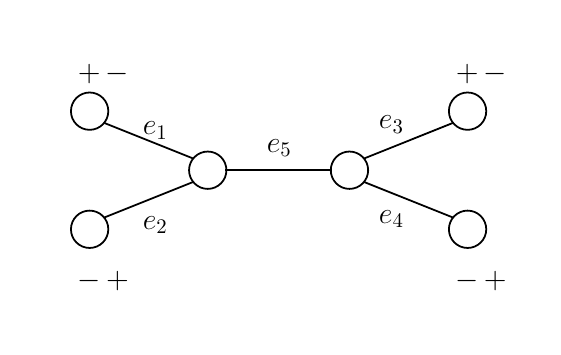}
		\caption{Steel's example. There are four leaves, the top two vertices have states $\sigma_v^{(1)} = +1$, $\sigma^{(2)}_v = -1$, and the bottom two vertices have $\sigma^{(1)}_u = -1$, $\sigma^{(2)}_u = +1$.}
		\label{fig:fail}
	\end{figure}
     In Figure \ref{fig:1} we provide a simple visualization of the likelihood landscape for the 4-leaf tree in Figure \ref{fig:fail}. 
      \begin{figure}[h!]
        \centering
        \includegraphics[width=0.6\linewidth]{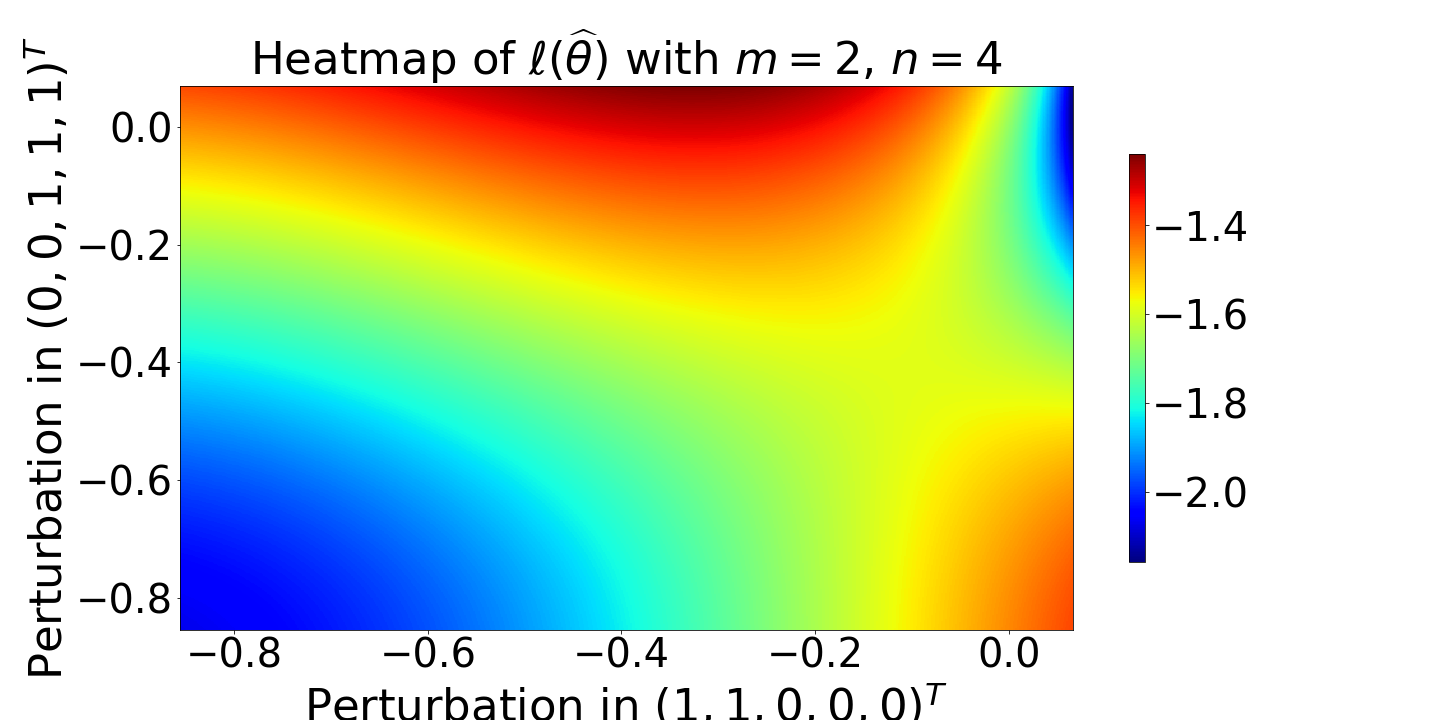}
        \caption{%\commHL{Nice! Can you make it landscape?}
        A 2D slice of a 5D 2-sample empirical log-likelihood $\ell(\hparam;\sigma^{(1)},\sigma^{(2)})$ containing $\param^*$ on a tree with $n=4$ leaves.}
        \label{fig:1}
    \end{figure}
See also~\cite{rogers1999multiple,chor2000multiple} for related results.

    \subsection{The one-dimensional likelihood landscape} 

    Several results have been obtained about the shape of the one-dimensional likelihood landscape, i.e., the coordinate-wise optimization problem. Fukami and Tateno \cite{FT.89} showed that, under a four-state model known in phylogenetics as F81~\cite{felsenstein1981}, when restricted to a single edge parameter $\htheta_{e}$, the log-likelihood $\ell$ is strictly concave and attains a unique maximizer. For more general discrete models, Dinh and Matsen~\cite{dinh_matsen_2017} provide conditions under which the one-dimensional likelihood is guaranteed to have at most one stationary point, a condition satisfied by the CFN model in particular. 

    Recently, with Sly \cite{clancy2025likelihood}, we showed that under the CFN model, the gradient of the population log-likelihood function is characterized by 
    the following approximation:
    \begin{equation}\label{eqn:gradient}
			\frac{\partial}{\partial \htheta_e} \E_{\param^{*}} \left[ \ell(\hparam;\sigma|_{L}) \right] = \frac{\theta^*_e-\htheta_e}{1-\htheta_e^{2} }  +O(\delta),
	\end{equation}
    where the parameter $\delta>0$ controls the scale of the transition probabilities along the edges. 
    From this result, it follows for instance that standard coordinate maximization for solving \eqref{eq:MLE_def} achieves an estimation of the true parameter within $L_{\infty}$-error of $O(\delta^{2})$ in a single round in the population limit. 
    A key tool developed in~\cite{clancy2025likelihood} is a sensitivity analysis of ancestral reconstruction, i.e., hidden state estimation (see Theorem~\ref{thm:Robust}). Roughly speaking, it says that the posterior mean (or ``magnetization''; see \eqref{eqn:def_magnetization}) of the state of a node (expressed as a spin value), remains a good estimate for the actual state there even under approximate parameters. This technical result is key to our analysis in this paper as well. 
    
    %\subsection{Our contribution: a second-order analysis of the population landscape}

\subsection{Contribution.} 

Our work contributes to this line of research by providing a fundamental population likelihood landscape result for the CFN model. 
    
    \begin{description}[itemsep=0.1cm, leftmargin=0.6cm]
        \item{1.} (\textit{Population likelihood landscape}; Thm. \ref{thm:DiagDominant}, Cor. \ref{cor:MLE_landscape}) There exists a box $\mathtt{B}\subseteq \Param$ with constant $L_{\infty}$-norm side lengths around the true parameter $\param^{*}$ such that the population likelihood $\E[\ell(\param)]$ is strongly concave and smooth on $\mathtt{B}$. %Furthermore, the $L_{\infty}$-diameter of $\mathtt{B}$ does not depend on the tree $T$.  
        (See the illustration in Fig. \ref{fig:enter-label}.)
        
        %\item{2.} The entries of the population Hessian $\E[\nabla^{2}\ell(\param)]$ vary in a Lipschitz manner with respect to the parameter.  
        %\item has a bound at the condition number at the true parameter $\param^*$.
        %\item The per-sample empirical Hessian  has uniformly bounded spectral norm almost surely. 
    \end{description}

In our recent work \cite{CLR:25}, we have already established that such a population  likelihood landscape result can be used to establish sample complexity and consistency of the MLE and also obtain computational guarantees of coordinate maximization for computing the MLE. We summarize these applications below. %which demonstrates the significance of our population likelihood landscape result. 

    \begin{description}[itemsep=0.1cm, leftmargin=0.6cm]

         \item{2.} (\textit{Empirical likelihood landscape}; Thm. \ref{thm:MLE_landscape_sample}) With enough samples $m$ (polynomial in the size of the tree in the balanced case), the empirical log-likelihood is strongly concave and smooth on the same universal box $\mathtt{B}$ with high probability.

        \item{3.} (\textit{Statistical estimation guarantee}; Thm. \ref{thm:MLE_estimation_error}) For any fixed problem size, the MLE is $O(1/\sqrt{m})$-consistent with the true parameter with arbitrarily large probability. 

        \item{4.} (\textit{Computational guarantee of coordinate maximization}; Thm. \ref{thm:MLE_opt_alg}) The iterates of the coordinate maximization algorithm %(Alg. \ref{algorithm:coord}) 
        converge exponentially fast to the confined MLE $\hparam_{\textup{MLE}}$ with a rate independent of the tree, provided a sufficiently close initial point.
    \end{description}

    The overarching goal in this work is to understand the structure of the optimization landscape of the maximum likelihood estimation problem in \eqref{eq:MLE_def}. In particular, we seek to understand the structure of the Hessian of the log-likelihood function in \eqref{eq:def_log_likelihood}. Our likelihood landscape results above show that, even though the population and empirical landscapes may be non-concave and contain numerous local (or even global) maximizers, it is very well conditioned inside the universal box $\mathtt{B}$ where there is a unique maximizer and all other local maximizers lie strictly outside this region. It is important that this region does not shrink as the sample size or the problem dimension grows large, which is typically the case in standard local landscape analysis, relying on positive definiteness of the Fisher information and continuity of the Hessian around the true parameter. Our analysis is semi-global in its nature. 
    
    By applying our semi-global likelihood landscape analysis,  we show that one can recover the true parameter $\param^*$ using standard likelihood maximization techniques and assess the incurred statistical and computational errors rigorously.

    \subsection{Organization}

    The remainder of this paper is organized as follows. In Section \ref{sec:main_results}, we formally state our main results, detail our assumptions, and provide high-level sketches of the analysis. In Section \ref{sec:preliminary_lemmas}, we present preliminary lemmas, including explicit representations for the gradient and Hessian of the log-likelihood function using the concept of ``magnetizations'' (Definition \ref{def:magnetization}). %Section \ref{sec:empirical} is devoted to proving our guarantees on the finite-sample likelihood landscape and the exponential convergence of coordinate maximization (Theorems \ref{thm:MLE_landscape_sample}, \ref{thm:MLE_estimation_error}, and \ref{thm:MLE_opt_alg}). 
    In Section \ref{sec:Hessian_off_diag_pf}, we provide the second-order analysis of the population landscape (Theorem \ref{thm:DiagDominant}) by bounding the diagonal entries of the expected Hessian (Section \ref{sec:diagonal_Hessian_pf}) and its off-diagonal entries. To handle the long-range correlations inherent in the model, we utilize a grouping strategy that decouples dependent random variables into independent blocks. Finally, Section 6 contains the detailed proofs for the technical lemmas underpinning the Hessian analysis.
    
    \section{Statement of main results}
\label{sec:main_results}

In this section, we formally characterize the geometry of the log-likelihood landscape in the neighborhood of the true parameter $\param^*$. Moving beyond projection-based analysis, we provide a multidimensional second-order treatment under the global assumption that edge transition probabilities scale as $\Theta(\delta)$ for a sufficiently small $\delta > 0$. 

Our primary result establishes that the population log-likelihood exhibits strong regularity within a ``universal box'' $\mathtt{B} \subset \Param$ of $L_\infty$-radius $\Theta(\delta)$ centered at $\param^*$. Specifically, as detailed in Theorem~\ref{thm:DiagDominant}, the expected landscape is $\Theta(\delta^{-1})$-strongly concave and $\Theta(\delta^{-1})$-smooth over $\mathtt{B}$. A key feature of this result is its independence from the specific tree topology $T$ and the leaf count $n$. This robust geometric structure ensures that the likelihood remains well-conditioned even as the problem dimension grows, providing the foundation for the finite-sample guarantees and the exponential convergence of coordinate maximization algorithms.

	We formally state our main results in this section. We also sketch the proofs. Detailed proofs follow in subsequent sections.
	We begin with some assumptions.
	
	\subsection{Assumptions} 
	Our analysis operates under the assumption that we are well within the reconstruction regime, that is, that mutation probabilities are sufficiently small that ancestral states can be reconstructed with better-than-random accuracy.
		Throughout the paper, constants are numbered by the equation in which they appear. We introduce the following restricted parameter spaces that depend on $\delta$.

    \begin{definition}[Restricted parameter spaces]\label{def:parameter_spaces}
        Let $ C_{\ref{eqn:pHatBounds}}> C_{\ref{eqn:pBounds}}>c_{\ref{eqn:pBounds}}> c_{\ref{eqn:pHatBounds}}>0$ be fixed constants. For a fixed $\delta>0$, define two subsets $\Param_{0}(\delta)\subseteq \hParam_{0}(\delta)\subset[-1,1]^{E}$  
        		by 
        		\begin{align}
        			\Param_{0}(\delta) &:=\left\{ (\theta^*_{e}=1-2p^*_{e}\,;\, e\in E) \,\bigg|\, c_{\ref{eqn:pBounds}} \delta \le p^*_e \le C_{\ref{eqn:pBounds}}\delta 
        			\textup{  $\forall$}e\in E \right\}=[1-2C_{\ref{eqn:pBounds}}\delta, 1-2c_{\ref{eqn:pBounds}}\delta]^{E}, \label{eqn:pBounds} \\
        			\hParam_{0}(\delta) &:=\left\{ (\htheta_{e}=1-2\hat{p}_{e}\,;\, e\in E) \,\bigg|\, c_{\ref{eqn:pHatBounds}}\delta \le \hat{p}_e \le C_{\ref{eqn:pHatBounds}} \delta
        			\textup{  $\forall$}e\in E\right\} = [1-2C_{\ref{eqn:pHatBounds}}\delta, 1-2c_{\ref{eqn:pHatBounds}}\delta]^{E}. \label{eqn:pHatBounds}
        		\end{align}
    \end{definition}

\noindent Our reconstruction bounds also require the following mild assumptions.
\begin{assumption}[Parameter regime]\label{assumption1} 
		Assume that $\param^{*}\in \Param_{0}(\delta)$ and $\hparam\in \hParam_{0}(\delta)$. Moreover, the constants $C_{\ref{eqn:pHatBounds}}>  C_{\ref{eqn:pBounds}}>c_{\ref{eqn:pBounds}}> c_{\ref{eqn:pHatBounds}}$ satisfy $C_{\ref{eqn:pHatBounds}}\ge 2c_{\ref{eqn:pHatBounds}}$.
	\end{assumption}
	\noindent We will frequently say that $\param^{*}$ or $\hparam$ satisfy \eqref{eqn:pBounds} and \eqref{eqn:pHatBounds}, respectively, which means that these parameters belong to the sets defined in these equations.

	 \subsection{Comments on notation}

 Before continuing, we introduce the following convention that is used throughout the paper. 
   
Given any non-negative function $f(\sigma,\hparam)$ depending on the state vector $\sigma = (\sigma_u;u\in T)$ and the estimator $\hparam\in \hParam_0$ satisfying Assumption \ref{assumption1}, we write $\E_{\param^{*}} [f(\sigma,\hparam)] = O(\delta^\alpha)$, $\Omega(\delta^\alpha)$ and $\Theta(\delta^\alpha)$ for some $\alpha\in \R$ to mean the following:
 \begin{align*}
   & \E_{\param^{*}}\left[f\left(\sigma,\hparam\right)\right]= O(\delta) \textup{  if } \exists K>0, \delta_0\in(0,1)\textup{ s.t. } \forall \delta\in(0,\delta_0], \,\, \sup_{(\param,\hparam )\in \Param_0\times\hParam_0} \E_{\param^{*}}\left[f\left(\sigma,\hparam\right)\right] \le K\delta^\alpha \\
   &\E_{\param^{*}}\left[f\left(\sigma,\hparam\right)\right] = \Omega (\delta^\alpha) \textup{  if } \exists K>0, \delta_0\in(0,1)\textup{ s.t. } \forall \delta\in(0,\delta_0],\,\, \inf_{(\param,\hparam )\in \Param_0\times\hParam_0} \E_{\param^{*}}\left[f\left(\sigma,\hparam\right)\right] \ge K\delta^\alpha \\
   & \E_{\param^{*}}\left[f\left(\sigma,\hparam\right)\right]= \Theta(\delta^\alpha) \textup{  if both }\E_{\param^{*}}\left[f\left(\sigma,\hparam\right)\right] = O(\delta^\alpha) \textup{ and } \E_{\param^{*}}\left[f\left(\sigma,\hparam\right)\right]= \Omega(\delta^\alpha)  
 \end{align*} The constants $K = K(c_{\ref{eqn:pBounds}},C_{\ref{eqn:pBounds}}, c_{\ref{eqn:pHatBounds}},C_{\ref{eqn:pHatBounds}})\in(0,\infty)$ and $\delta_0 = \delta_0(c_{\ref{eqn:pBounds}},C_{\ref{eqn:pBounds}}, c_{\ref{eqn:pHatBounds}},C_{\ref{eqn:pHatBounds}})$ depend on the constants appearing in Assumption \ref{assumption1}, but otherwise \textit{\textbf{independent of both the size and topology of the tree}} $T$. 
 
 We will similarly write $f(\sigma,\hparam) = O(\delta^\alpha)$ if there exists a constant $K = K(c_{\ref{eqn:pBounds}},C_{\ref{eqn:pBounds}}, c_{\ref{eqn:pHatBounds}},C_{\ref{eqn:pHatBounds}})\in(0,\infty)$ and $\delta_0 = \delta_0(c_{\ref{eqn:pBounds}},C_{\ref{eqn:pBounds}}, c_{\ref{eqn:pHatBounds}},C_{\ref{eqn:pHatBounds}})$
 \begin{equation*}
     \PR_{\param^{*}}\left(f(\sigma,\hparam) \le  K\delta^\alpha\right) = 1 \textup{ for all }(\param,\hparam)\in \Param_0\times\hParam_0 \textup{ and }\delta\in (0,\delta_0].
 \end{equation*}
 Here we remind the reader that $\Param_0, \hParam_0$ are sets that depend on $\delta$. We similarly, use $\Omega(\delta^\alpha)$ and $\Theta(\delta^\alpha)$.
 
We will also reserve the symbol $\kappa$ for a generic constant with the following property:
\begin{equation*}
    \exists \delta_0 = \delta_0(c_{\ref{eqn:pBounds}},C_{\ref{eqn:pBounds}}, c_{\ref{eqn:pHatBounds}},C_{\ref{eqn:pHatBounds}})\in (0,1)\textup{ s.t. } \kappa\in[\delta_0 , 1-\delta_0].
\end{equation*} The precise value of $\kappa$ may change from line to line.

  We will specify precise constants in the statements of results, but we will otherwise use the above convention for constants.

	\subsection{Population  likelihood landscape}

	The main  result in this paper, Theorem \ref{thm:DiagDominant}, concerns the structure of the expected Hessian of the log-likelihood function $\ell$ in \eqref{eq:MLE_def}. We establish that the diagonal entries of the expected Hessian are of order $O(\delta^{-1})$ while the off-diagonal entries are exponentially small in the shortest path distance between the two corresponding edges. For its statement, we let $\mathbf{H}(\hparam)\in \R^{|E|\times |E|}$ denote the Hessian of the expected log-likelihood, whose $(e,f)$ entry for $e,f\in E$ is defined as 
    \begin{align}\label{eq:def_pop_hessian}
     \mathbf{H}(\hparam)_{e,f} = \frac{\partial^2}{\partial \htheta_e \partial\htheta_f} \E_{\param^{*}}\left[\ell(\hparam;\sigma|_{L})\right]= \E_{\param^{*}}\left[\frac{\partial^2}{\partial \htheta_e \partial\htheta_f}\ell(\hparam;\sigma|_{L})\right].     \end{align}

	\begin{theorem}[Population log-likelihood landscape: Hessian]\label{thm:DiagDominant} 
		There exist constants $C_{\ref{eqn:HessianDiag}}, \widetilde{C}_{\ref{eqn:HessianDiag}}$, $C_{\ref{eqn:HessiangOffDiag}}$, $C_{\ref{eqn:HessiangOffDiag_Var}}$,  and $\delta_{\ref{eqn:HessianDiag}}$ that depend only on $c_{\ref{eqn:pBounds}},C_{\ref{eqn:pBounds}}, c_{\ref{eqn:pHatBounds}},C_{\ref{eqn:pHatBounds}}$ such that the following holds for any binary tree $T$, any $\delta<\delta_{\ref{eqn:HessianDiag}}$, and $\hparam\in \hParam_{0}(\delta)$: Assume \ref{assumption1} holds then for all edges $e,f$,
		\begin{description}
			\item[(i)] (\textit{Large diagonal entries in expectation}) 
			\begin{align} 
            - \frac{\widetilde{C}_{\ref{eqn:HessianDiag}}}{\delta}\le   
             \mathbf{H}(\hparam)_{e,e}
            \le- \frac{C_{\ref{eqn:HessianDiag}}}{\delta}. \label{eqn:HessianDiag}
			\end{align}
			
			\item[(ii)] (\textit{Small off-diagonal entries in expectation})  
			\begin{align}
                 \mathbf{H}(\hparam)_{e,f}
				\le (C_{\ref{eqn:HessiangOffDiag}}\delta)^{ \lfloor \frac{(\textup{dist}(e,f)-1)\lor 0}{4}\rfloor  }. \label{eqn:HessiangOffDiag}
			\end{align}

		\end{description}

	\end{theorem}
	\noindent The proof of Theorem \ref{thm:DiagDominant} is quite involved and proceeds in several steps (see Section \ref{sec:Hessian_off_diag_pf}). 
 
	The key consequence of Theorem \ref{thm:DiagDominant} (along with Gershgorin's circle theorem), and the main result of the paper, is that the population landscape is strongly concave and smooth under Assumption \ref{assumption1} and if $\delta$ is smaller than a universal constant. 
	
	\begin{corollary}[Population log-likelihood landscape: strong concavity and smoothness]
		\label{cor:MLE_landscape}
            There exists a constant $\delta_{\ref{eq:expected_Hessian_eval_range}}\in (0,1)$ depending only on $c_{\ref{eqn:pBounds}},C_{\ref{eqn:pBounds}},c_{\ref{eqn:pHatBounds}},C_{\ref{eqn:pHatBounds}}$ such that for all binary trees $T$, $\delta\le \delta_{\ref{eq:expected_Hessian_eval_range}}$ and $\hparam\in \hParam_{0}(\delta)$ and $\param^*\in \Param_0(\delta)$, 
		\begin{align}\label{eq:expected_Hessian_eval_range}
			-\frac{\widetilde{C}_{\ref{eqn:HessianDiag}}}{\delta} - 26  \le \lambda_{\min}(\mathbf{H}(\hparam)) \le \lambda_{\max}(\mathbf{H}(\hparam)) \le -\frac{C_{\ref{eqn:HessianDiag}}}{\delta} + 26,
		\end{align}
		where $\lambda_{\min}(\cdot)$ and $\lambda_{\max}(\cdot)$ denote the minimum and the maximum eigenvalues of a matrix. In particular, 
        in the population limit $m\rightarrow\infty$, the log-likelihood function $\ell$ in \eqref{eq:MLE_def} is
         $(\frac{C_{\ref{eqn:HessianDiag}}}{\delta} - 26)$ - strongly concave and $	(\frac{\widetilde{C}_{\ref{eqn:HessianDiag}}}{\delta}+ 26)$-smooth. In particular, the true parameter $\param^{*}$ is the unique maximizer of $\ell$ over $\hParam_{0}$. 
	\end{corollary}

    \subsection{Applications to finite-sample likelihood landscape and coordinate maximization}

    The population landscape results stated in the previous section have applications to finite-sample likelihood landscapes, consistency of maximum-likelihood estimators, and convergence of the coordinate maximization algorithm. These applications have recently been published in \cite{CLR:25},  which were obtained assuming the population landscape result in Theorem \ref{thm:DiagDominant} and Corollary \ref{cor:MLE_landscape}. We summarize the main results in \cite{CLR:25}. We emphasize that the results stated in this section are not sharp and we believe our techniques will lead to further improvements.

	First,  the result in Corollary \ref{cor:MLE_landscape} continues to hold for the empirical likelihood landscape as long as the number  $m$ of observed samples is large enough.  To establish such a result, we use a uniform version of matrix Bernstein's inequality \cite[Lem. 4.5]{CLR:25} to show that the Hessian of the empirical log-likelihood function is concentrated near its expectation \textit{uniformly} over the box $\hParam_{0}(\delta)$ with high probability. Then the assertion will follow from the population landscape result  in Corollary \ref{cor:MLE_landscape}. 
	
	\begin{theorem}[Finite-sample log-likelihood landscape: strong concavity and smoothness; Thm. 3.2 in \cite{CLR:25}]
		\label{thm:MLE_landscape_sample}
		Let $\delta\le \delta_{\ref{eq:expected_Hessian_eval_range}}$ and let  $\widehat{\mathbf{H}}$ denote the Hessian of the $m$-sample log-likelihood function $\ell$ in \eqref{eq:def_log_likelihood}.  
		Fix $\eps\in (0,1)$. Then there exists a constant $C_{\ref{eq:sample_complexity}}>0$ such that if 
    \begin{align}\label{eq:sample_complexity}
			%m\ge2 C_{\ref{eq:sample_complexity}}^{2} |E|^{3} \textup{diam}(T)\delta^{-2\textup{diam}(T)} \log  \left(  2C_{\ref{eq:sample_complexity}}|E|\eps^{-1}\delta^{-1} \right).
            %m \ge  2|E|^{2} C_{\ref{eq:sample_complexity}}^{2} \delta^{-2\textup{diam}(T)}  \left(  |E| \log  \left( 2(C_{\ref{eqn:pHatBounds}}-c_{\ref{eqn:pHatBounds}})\delta (1+4C_{\HL{???}}|E|^{2}\delta^{-3\textup{diam}(T)} )  \right) +  \log (4|E|\eps^{-1})   \right). 
            %m&\ge |E|^{2} (C_{\ref{eq:sample_complexity}}/\delta)^{\textup{diam}(T)+8} \, \log (2e/\eps),
            m&\ge (C_{\ref{eq:sample_complexity}}/\delta)^{\textup{diam}(T)+8} \, \log (\eps^{-1}),
		\end{align}
        then for any binary tree $T$, and $\param^*\in \Param_0(\delta)$,%\HL{specify $C_{\ref{eq:sample_complexity}}$ here?} 
\begin{align}\label{eq:finite_sample_ML_landscape_thm}
        \P_{\param^{*}} \Bigg(    -\frac{\widetilde{C}_{\ref{eq:expected_Hessian_eval_range}}}{\delta} - 27  \le \inf_{\param\in \hParam_{0}(\delta)} \lambda_{\min}(\widehat{\mathbf{H}}(\param)) \le \sup_{\param\in \hParam_{0}(\delta)} \lambda_{\max}(\widehat{\mathbf{H}}(\param)) \le -\frac{C_{\ref{eq:expected_Hessian_eval_range}}}{\delta} + 27 \Bigg) \ge 1-\eps.
        \end{align}		
	\end{theorem}

	The required sample complexity for Theorem \ref{thm:MLE_landscape_sample} grows exponentially in the diameter of the tree $T$, which can be as small as $O(\log n)$ when the tree is `well-balanced'. In the latter case, Theorem \ref{thm:MLE_landscape_sample} below establishes a polynomial sample complexity to obtain a strongly concave and smooth optimization landscape  for the MLE problem with high probability. %To the best of our knowledge, this is the first result in the literature establishing strong regularity properties of the maximum likelihood landscape for the branch length optimization problem with polynomial sample complexity. 
    However, we suspect this sample complexity can be improved to depend logarithmically in the number of leaves rather than exponentially in the diameter. We leave this important direction for future work.

    %To establish such a result, we prove and use a uniform version of matrix Bernstein's inequality (Lemma \ref{lem:unif_mx_bernstein}) to show that the Hessian of the empirical log-likelihood function is concentrated near its expectation \textit{uniformly} over the box $\hParam_{0}(\delta)$ with high probability. Then the assertion will follow from the population landscape result in Theorem \ref{thm:DiagDominant}. 

	%\subsection{Statistical and Computational Guarantees}

    Next,  denote a generic global maximizer of the empirical log-likelihood function $\ell(\cdot)$ (in \eqref{eq:def_log_likelihood}) over $\hParam_{0}(\delta)$ (which always exists) by $\hparam_{\textup{MLE}}$. A particular consequence of Theorem \ref{thm:MLE_landscape_sample} is that $\hparam_{\textup{MLE}}$ is uniquely determined with high probability and enough samples. Specifically, Theorem \ref{thm:MLE_estimation_error} below states that the MLE $\hparam_{\textup{MLE}}$ is a $1/\sqrt{m}$-consistent estimator of the true parameter $\param^{*}$ with high probability. 
    This consistency (up to a constant depending on $T$) does not depend on our choice of norm on $\R^E$ for any fixed tree. We write $\|\cdot\|$ for the $L^2$-norm of a vector.

	\begin{theorem}[Statistical estimation guarantee; Thm. 3.3 in \cite{CLR:25}]\label{thm:MLE_estimation_error}
		Assume the hypothesis of Theorem  \ref{thm:MLE_landscape_sample} holds. Let $E_{\ref{eq:finite_sample_ML_landscape_thm}}$ denote the event in \eqref{eq:finite_sample_ML_landscape_thm}. Fix $\eps\in (0,1)$ and denote $\rho:=\frac{C_{\ref{eq:expected_Hessian_eval_range}}}{\delta} - 27$ and 
         $  C_{\ref{eq:MLE_estimation_error}} :=  \frac{16 \widetilde{C}_{\ref{eq:expected_Hessian_eval_range}}}{C_{\ref{eq:expected_Hessian_eval_range}}}$.        
         Then we have \begin{align}\label{eq:MLE_estimation_error}
			\P_{\param^{*}}	& \left( 	E_{\ref{eq:finite_sample_ML_landscape_thm}} \cap \left\{	\lVert \param^{*} - \hparam_{\textup{MLE}}  \rVert \le C_{\ref{eq:MLE_estimation_error}} \sqrt{|E|/m} \log(|E|/\eps) \right\}  \right)   \ge 1-3\eps
		\end{align}
    provided that $m$ satisfies \eqref{eq:sample_complexity} and $ m \ge \frac{|E|^2}{4 C_{\ref{eq:expected_Hessian_eval_range}}^6c_{\ref{eqn:pHatBounds}}^6 \delta^3 \eps}$.
	\end{theorem}

	Now that we know the MLE  $\hparam_{\textup{MLE}}$ is close to the true parameter $\param^{*}$ with high probability, we turn our attention to how we can compute the  $\hparam_{\textup{MLE}}$ from the observed samples $\sigma^{(1)},\dots,\sigma^{(m)}$ restricted to the leaves. While the log-likelihood function $\ell$ in \eqref{eq:def_log_likelihood} is %highly 
    non-concave, it has the nice structure of being strictly %convex
    concave when restricted to a single branch length $\hat{\theta}_{e}$ for $e\in E$ (see Lem. \ref{lem:derivative}). Thus, it is natural to cycle through the branch lengths and optimize one at a time, maximizing the one-dimensional restricted likelihood function. This yields the following ``cyclic coordinate maximization'' algorithm for computing the MLE. Namely, given our estimate $\hparam_{k}=(\htheta_{k;e};e\in E)$ after $k$ iterations, the algorithm proceeds by optimizing for a single  branch length $\theta_{k;e}$ by %(approximately) 
 %solving 
 \begin{align}\label{eq:alg_high_level}
 \begin{cases}
 	&\htheta_{k+1;e} \leftarrow \argmax_{\htheta\in [-1,1]} \overline{f}_{k;e}(\htheta), \,\, \textup{where} \,\, \overline{f}_{k;e}(\htheta):=  \frac{1}{m} \sum_{i=1}^{m} \ell( \hparam_{k+1;1:e-1}, \htheta,\hparam_{k;e+1:|E|};\sigma^{(i)}), \\
    &\hparam_{k;i:j} := (\htheta_{k;i},\htheta_{k;i+1},\dotsm,\htheta_{k;j}) 
  \end{cases}
 \end{align}
 assuming that we label the edge set $E$ as integers from 1 through $|E|$. The one-dimensional objectives $\overline{f}_{k;e}(\htheta)$ in \eqref{eq:alg_high_level} are known to be strictly concave \cite{FT.89} and they have a unique maximizer in $(-1,1)$ at a unique critical point: 
 \begin{align}\label{eq:MLE_critical_pt}
 	\frac{\partial}{\partial\htheta_e} \overline{f}_{k;e}(\htheta) = 0. 
 \end{align} 
 The unique zero of the above critical-point equation can be found rapidly by using standard zero-finding algorithms (e.g., \cite{brent2013algorithms}). %This is what was originally proposed by \citet{guindon2003simple} for branch length optimization.
 %See Alg. \ref{algorithm:coord} in Appendix \ref{sec:Alg_appendix} for a detailed implementation of the algorithm. 
See e.g.~\cite{guindon2003simple} for a practical implementation of this type of algorithm.

 Despite the popularity and the empirical success of the coordinate maximization algorithm above, however, due to the %high
 non-concavity of the optimization landscape, there has been no guarantee about the %global 
 convergence of this algorithm to the maximizer of $\ell$ or the true parameter $\param^{*}$.  Theorem \ref{thm:MLE_opt_alg} below  establishes that the coordinate maximization algorithm above \eqref{eq:alg_high_level}  converges exponentially fast to the MLE  $\hparam_{\textup{MLE}}$, which is within $C(T,\eps,\delta) m^{-1/2}$ from the true parameter $\param^{*}$, provided the initial estimate $\hparam_{0}$ is within $O(\delta)$ from the true parameter $\param^{*}$ in $L_2$ norm.

	\begin{theorem}[Statistical and computational estimation guarantee; Thm. 3.4 in \cite{CLR:25}]\label{thm:MLE_opt_alg}
		Suppose the hypothesis of Theorem \ref{thm:MLE_estimation_error} holds. 
		Let $(\hparam_{k})_{k\ge 0}$ denote the sequence of estimated parameters generated by the coordinate maximization algorithm %(see Alg. \ref{algorithm:coord}) 
        with the initial estimate $\hparam_{0}$ satisfying 
		\begin{align}\label{eq:coordinate_max_initialization}
			\lVert \hparam_{0} - \param^{*} \rVert \le %\frac{\mu}{L} (C_{\ref{eq:param_set_box_gap}}/2)\delta  = 
			\frac{(C_{\ref{eq:expected_Hessian_eval_range}}-27\delta)  C_{\ref{eq:coordinate_max_initialization}}}{C_{\ref{eq:expected_Hessian_eval_range}}+\widetilde{C}_{\ref{eq:expected_Hessian_eval_range}}}\frac{\delta}{2},
		\end{align}
		where $C_{\ref{eq:coordinate_max_initialization}}:=(C_{\ref{eqn:pHatBounds}}-C_{\ref{eqn:pBounds}})\land (c_{\ref{eqn:pHatBounds}}-c_{\ref{eqn:pBounds}})>0$. 
		 Then with probability at least $1-3\eps$,  for all $k\ge 0$, 
		\begin{align}\label{eq:comp_stat_guarantee1}
			\lVert& \hparam_{\textup{MLE}} - \hparam_{k} \rVert^{2} \le \frac{\widetilde{C}_{\ref{eq:expected_Hessian_eval_range}} - 27\delta }{C_{\ref{eq:expected_Hessian_eval_range}}- 27\delta }  \left(1 - \frac{C_{\ref{eq:expected_Hessian_eval_range}}\delta^{-1} - 27}{\widetilde{C}_{\ref{eq:expected_Hessian_eval_range}}\delta^{-1} - 27} \right)^{k-1} \lVert \hparam_{\textup{MLE}}-\hparam_{0} \rVert^{2}.
		\end{align}
		In particular, 
\begin{align}\label{eq:comp_stat_guarantee2}
			\hspace{-0.4cm} \lVert \param^{*} - \hparam_{k} \rVert \le  \underbrace{  C_{\ref{eq:MLE_estimation_error}} \frac{|E|}{m} \log( \frac{|E|}{\eps}) }_{=\textup{statistical error}}  + \underbrace{ \sqrt{\frac{\widetilde{C}_{\ref{eq:expected_Hessian_eval_range}} - 27\delta }{C_{\ref{eq:expected_Hessian_eval_range}}- 27\delta } }  \left(1 - \frac{C_{\ref{eq:expected_Hessian_eval_range}}\delta^{-1} - 27}{\widetilde{C}_{\ref{eq:expected_Hessian_eval_range}}\delta^{-1} - 27} \right)^{(k-1)/2} \lVert \hparam_{\textup{MLE}}-\hparam_{0} \rVert }_{=\textup{computational error}}.
		\end{align}
	\end{theorem}

    It is important to note that the exponential rate of convergence of the coordinate maximization in Theorem \ref{thm:MLE_opt_alg}  is a universal constant that does not depend on the tree $T$ and also the %`ferromagnetism' 
    parameter $\delta$ (as long as it is less than some universal constant in Thm. \ref{thm:MLE_landscape_sample}). This means that the computational error for computing the MLE can be made to be less than a desired tolerance $\eps$ within $C\log \eps^{-1}$ iterations for some universal constant $C$. 
    This gives some theoretical support for the empirical fact that coordinate maximization algorithm performs well empirically, even with a small number of coordinate updates  \cite{guindon2003simple,guindon2010new}.  
    
    It would be of interest to show that the assumption that  $\|\hparam_0-\param^*\| = O(\delta)$ can be dropped; however, our proof needs the initial iterate to be sufficiently close to $\param^*$ in order to know that the empirical Hessian is smooth and strongly concave with high probability (Thm. \ref{thm:MLE_landscape_sample}) and that the subsequent iterates also lie in this ``good'' region.

    In Figure \ref{fig:opt_experiment}, we provide a numerical validation of Theorem \ref{thm:MLE_opt_alg} on a 20-node tree. The estimation error for the cyclic coordinate maximization \eqref{eq:alg_high_level} indeed decays exponentially toward a limiting value $\approx 10^{-3}$, which should correspond to the statistical error resulting from the discrepancy between the population and the empirical likelihood landscape.
    
    \begin{figure}
        \centering
        \includegraphics[width=.9\linewidth]{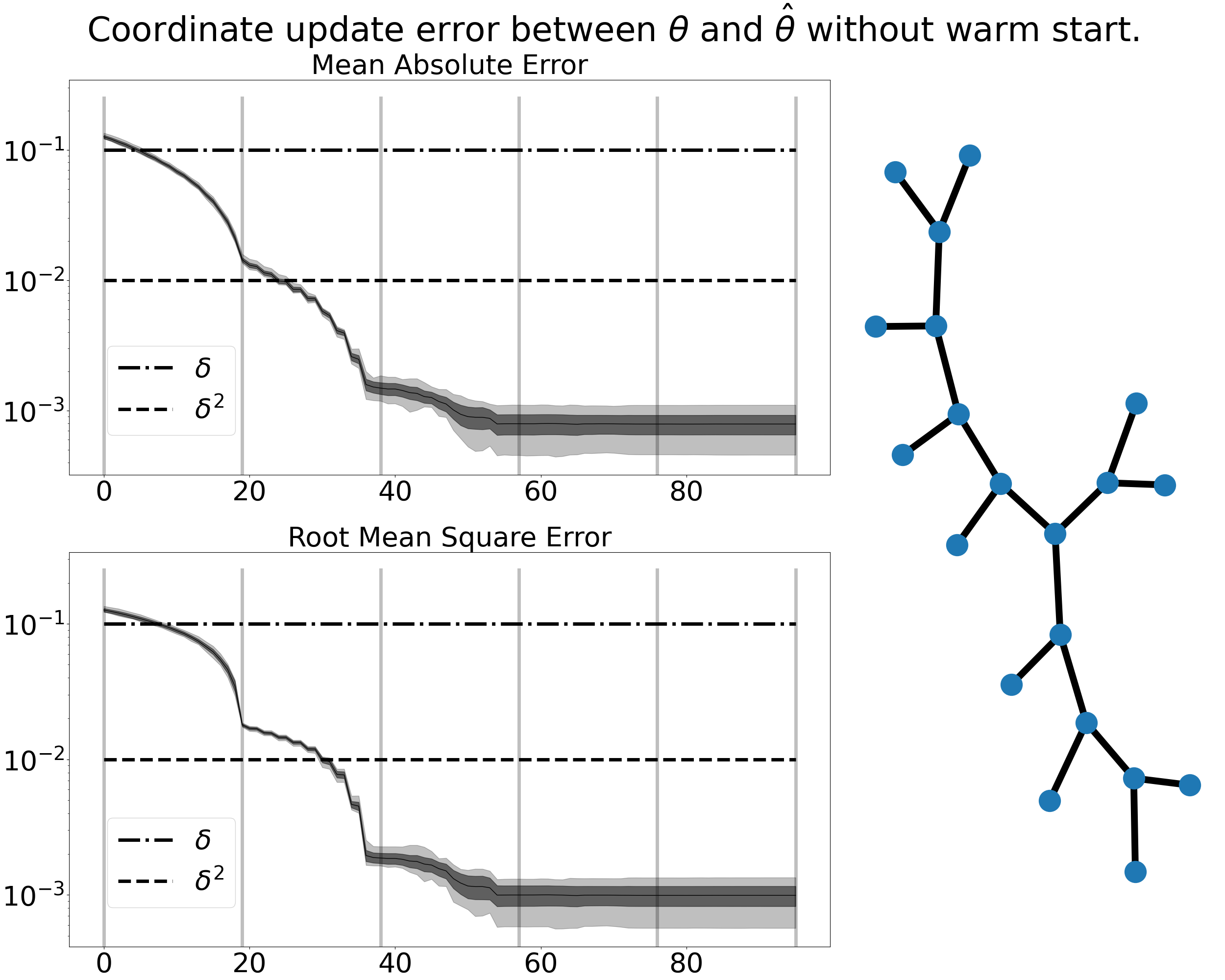}
        \vspace{0.2cm}
        \caption{Error between $\param^*$ and $\hparam$ computed via cyclic coordinate maximization. The underlying tree has 20 nodes and is depicted on the right. The algorithm is randomly initialized 25 times with random edge weights with $\param^*\sim \operatorname{Unif}( \Param_0(\delta))$ with $c_{\ref{eqn:pBounds}} = 1/4, C_{\ref{eqn:pBounds}}= 1/2$ and $\delta = 1/10$. Each tree had $m = 2\times 10^5$ spins at each leaf. Vertical gray lines are one complete round of coordinate updates. Solid black line is mean error over 25 samples, dark gray region is mean $\pm$ standard deviation, and light gray region is the range of observed errors. The parameter $\hparam$ is initialized with $\htheta_{0;e} = 0.8$ for all $e$.}
        \label{fig:opt_experiment}
    \end{figure}

    \subsection{Sketch of analysis for the population landscape}

    %\commHL{(HL: Texts in this section are mostly not new, but I've reorganized our proof sketches for the main theorems scattered in many sections (both in our AOS submission and the ICML paper) into a single section to clearly communicate the mathematical difficulties and how we address them.)}

    \subsubsection{Magnetization}
    
	We first discuss how to conveniently express the Hessian of the log-likelihood function \eqref{eq:def_log_likelihood} using an observable called the `magnetization'. 
	We begin with some important definitions.
	
	%\subsection{Definitions} \label{ssec:MagDefs}

	Fix two distinct nodes $u,v$ in $T$. We call a node $w$ a \textit{descendant} of $u$ with respect to node $v$ if the shortest path between $w$ and $v$ contains $u$. The \textit{descendant subtree at $u$} with respect to $v$ is the subtree $T_{u}$ rooted at $u$ consisting of all descendants of $u$ with respect to $v$. A subtree of $T$ rooted at $u$ is a \textit{descendant subtree of $u$} if it is a descendant subtree of $u$ with respect to some node $v$.

	The following notion of `magnetization' is central to the overall analysis in this work.  Roughly speaking, the magnetization  $Z_{u}$ of a node $u$ with respect to a descendant subtree $T_{u}$ rooted at $u$ is the `bias' on its spin after observing all spins at the leaves of the descendant subtree $T_{u}$. For instance, if all spins on the leaves of $T_{u}$ are $+$, then $u$ will be quite likely to have $+$ spin as well. The formal definition of magnetization is given below. 
	\begin{definition}[Magnetization]\label{def:magnetization}
		Let $T_{u}$ be a descendant 
		subtree of $T$ rooted at a node $u$. Let $L_{u}$ denote the set of all leaves in $T_{u}$. For a generic parameter $\hparam \in [-1,1]^{E(T_u)}$ and fixed spin configuration $\sigma_{L_u}\in \{\pm 1\}^{L_{u}}$
  on the leaves of $T_u$, define the magnetization at the root $u$ of $T_{u}$ under $\hparam$ as
		\begin{align}\label{eqn:def_magnetization}
   Z^{\hparam,T_u}_u (\sigma_{L_{u}}):=
			\PR_\hparam(\hat{\sigma}_u = +1\, |\,  \hat{\sigma}_{L_{u}} = \sigma_{L_{u}}) - \PR_\hparam(\hat{\sigma}_u=-1\, |\,  \hat{\sigma}_{L_{u}} = \sigma_{L_{u}}),
		\end{align}
		where $\hat{\sigma}$ is a random spin configuration on $T$  sampled from $\PR_{\hparam}$. 
		Furthermore, if $\sigma$ is a random spin configuration sampled from $\PR_{\param^{*}}$, we consider the random variable 
		\begin{align*}
  Z_{u}^{\hparam,T_{u}}:=Z_{u}^{\hparam, T_{u}}(\sigma_{L_{u}}). 
		\end{align*}
  We write this random variable simply as $Z_u$ when $\hparam,T_{u}$ are clear from the context. 
	\end{definition}
	
	If $T_{u}$ consists of a single node $u$, then $Z_{u}=X_{u}$ as we get to observe the spin at $u$. In general, $Z_{u}$ is a random variable determined by the spin configuration $X_{L_{u}}$ on the leaves of $T_{u}$ and takes values in $[-1,1]$. In fact, there is a recursive structure of magnetization, first established in Borgs, Chayes, Mossel, and Roch \cite{BCMR.06} which we now recall.

Suppose $T_{u}$ is a descendant subtree of a node $u$ and let $v,w$ be its two children in $T_{u}$. There are corresponding descendant subtrees rooted at these nodes with respect to $u$, which defines the magnetization at these nodes, say, $Z_{v}$ and $Z_{w}$.
Then \cite[Lemma 4 and 5]{BCMR.06} imply that, under $\PR_{\hparam}$,
\begin{equation}\label{eqn:recursionBorg}
	Z_u = \frac{\hat{\theta}_{v}Z_v + \hat{\theta}_w Z_{w}}{1+ \hat{\theta}_v\hat{\theta}_w Z_vZ_w}, 
\end{equation} 
where $\hat{\theta}_v = \hat{\theta}_{\{u,v\}}$ for the edge $\{u,v\}$ and similarly for $w$.
Throughout the paper, we reserve
\begin{align}\label{eq:q_def_recursion}
	q(x,y) := \frac{x+y}{1+xy},
\end{align}
so that \eqref{eqn:recursionBorg} reads as 
\begin{equation}\label{eqn:recursionBorg2}
	Z_u = q(\hat{\theta}_{v}Z_v, \hat{\theta}_w Z_{w}).
\end{equation}

	%\subsection{An expression for the Hessian}
	
	The magnetization $Z_{u}$ at a node $u$ depends implicitly on the choice of the descendant subtree $T_{u}$. As already observed in \cite{clancy2025likelihood}, magnetizations can be used to describe the gradient of the log-likelihood function in the mutation probabilities on the edges. To describe the Hessian, we introduce canonical decompositions of the tree $T$ with respect to either a single or two edges, which yields unambiguous choice of the descendant subtrees of all nodes. 
		
	\begin{figure}[h!]
		\centering
		\includegraphics[width=1\linewidth]{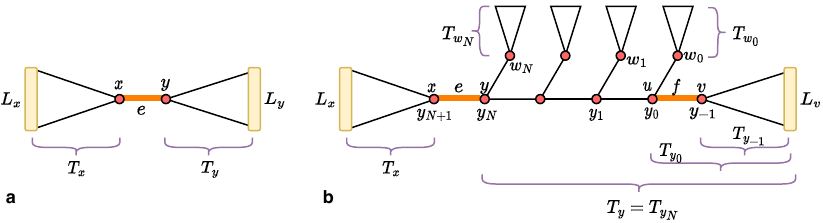}
		\caption{ Decomposition of the tree $T$ into subtrees with respect to  (\textbf{a}) a single edge $e=\{x,y\}$ and (\textbf{b}) two edges $e=\{y_{N}, y_{N+1}\}$ and $f=\{ y_{-1}, y_{0} \}$. For instance, the subtree $T_{u}$ in panel \textbf{a} is rooted at $u$ and $L_{u}$ denotes the set of all leaves in $T_{u}$; The subtree $T_{y_{N}}$ in panel \textbf{b} is rooted at $y_{N}$ and contains all nodes $y_{-1},\dots,y_{N}$ and $w_{0},\dots,w_{N}$ as well as the subtrees $T_{y_{-1}},T_{w_{0}},\dots,T_{w_{N}}$.} 
	\label{fig:TREEhess}
\end{figure}

Consider the log-likelihood function $\ell(\hparam; \sigma)$ in \eqref{eq:def_log_likelihood}
and two edges $e = \{x,y\}$ and $f=\{u,v\}\in E(T)$. Let $T_x$ and $T_y$ denote the subtrees rooted at $x$ and $y$ (resp.) obtained by removing $e$ from the edges of $T$ (see Figure \ref{fig:TREEhess}\textbf{a}). Suppose that $f\in E(T_y)$ and that $u$ is closer to $y$ than $v$, i.e., $d(u,y)< d(v,y)$. Enumerate the vertices on the path from $y$ to $u$ by $y = y_N, y_{N-1},\dots, y_1,y_0 = u$ and set $y_{-1} = v$ and $y_{N+1} = x$. Note that for each vertex $y_j$ with $j\in\{0,\dotsm, N\}$, the vertex $y_j$ has degree three and so has neighbors $\{y_{j-1},y_{j+1}, w_j\}$ for some other vertex $w_j$.
Accordingly, we have $T_{x}=T_{y_{N+1}}$ and $T_{y}=T_{y_{N}}$, and for every node $z$ in $T_{y}$, the descendant subtree $T_{z}$ is with respect to the root $x$. See Figure \ref{fig:TREEhess}\textbf{b} for illustration.

The following key lemma relates the derivatives of the log-likelihood and the magnetizations.
\begin{lemma}[Derivatives of the log-likelihood and magnetization]\label{lem:derivative}
	The following formulas hold.
	\begin{description}
		\item[(i)] (\textit{Gradient})     For edge $e = \{x,y\}$, we have 
		\begin{align}
			\frac{\partial}{\partial\htheta_e}\ell(\hparam;\sigma|_{L}) &= \frac{Z_x Z_y}{1+Z_x Z_y\htheta_e}, \label{eqn:derivative}
		\end{align}

		\item[(ii)](\textit{Hessian}) 
		For edges $e = \{x,y\}$ and $f = \{u,v\}$ with $\textup{dist}(e,f)=N$ as above, we have 
  \begin{align}\label{eqn:Z_derivative}
\frac{\partial}{\partial \htheta_f} Z_y =Z_v \prod_{j=1}^N \htheta_{\{y_j,y_{j-1}\}} \prod_{j=0}^N \frac{1-(\htheta_{\{y_j,w_j\}} Z_{w_j})^2}{\left(1+\htheta_{\{y_j,w_j\}}\htheta_{\{y_j,y_{j-1}\}} Z_{w_j}Z_{y_{j-1}}\right)^2}.
  \end{align}
  In particular,
		\begin{align}
			\frac{\partial^2}{\partial \htheta_e\partial \htheta_f} \ell(\hparam;\sigma|_{L})  &= \left( \frac{\htheta_e Z_x Z_v}{(1+\htheta_e Z_x Z_y)^2} \prod_{j=1}^N\htheta_{\{y_j,y_{j-1}\}} \right) \prod_{j=0}^{N} \frac{ \left(1- (\htheta_{\{y_j,w_j\}} Z_{w_j})^2 \right)}{\left(1+\htheta_{\{y_j,w_j\}} \htheta_{\{y_j,y_{j-1}\}} Z_{w_j} Z_{y_{j-1}}\right)^2}. \label{eqn:hessianTerms} 
		\end{align}

         \item[(iii)](\textit{Third-order derivatives})  If $\hparam\in \hParam_0(\delta)$, then for all edges $e_1,e_2,e_3$:
    \begin{align}\label{eqn:third_order_der_bd} 
        \left|\frac{\partial^3}{\partial\htheta_{e_1} \partial\htheta_{e_2} \partial \htheta_{e_3}} \ell(\hparam;\sigma)\right| \le     \frac{4\operatorname{diam}(T)}{(2c_{\ref{eqn:pHatBounds}}\delta)^{4\operatorname{diam}(T)+2}}. 
    \end{align}
	\end{description}
\end{lemma}

The expression for the Hessian in \eqref{eqn:hessianTerms} above is rather complicated. Looking at the denominators in \eqref{eqn:hessianTerms}, we see that each of them is (at worst) $\Omega(\delta^2)$, and as there are at most $\operatorname{diam}(T)$ many a na\"ive bound on the Hessian gives
\begin{align*}
   \left| \frac{\partial^2}{\partial\htheta_e\partial\htheta_f} \ell(\hparam;\sigma)\right| = O\left(\delta^{-2\operatorname{diam}(T)}\right);
\end{align*}
however, we can provide a much better bound. We state this as the following lemma.

\begin{lemma}\label{lem:uniformHessian}
    There exist constants $C_{\ref{eqn:uniformbound}}$, $\widetilde{C}_{\ref{eqn:uniformbound}}$ and $\delta_{\ref{eqn:uniformbound}}$ such that for all binary trees $T$ and $\delta\le \delta_{{\ref{eqn:uniformbound}}}$
    \begin{align}\label{eqn:uniformbound}
        \left| \frac{\partial^2}{\partial\htheta_e\partial\htheta_f} \ell(\hparam;\sigma)\right| \le C_{\ref{eqn:uniformbound}} \left(\frac{\widetilde{C}_{\ref{eqn:uniformbound}}}{\delta}\right)^{\operatorname{diam}(T)/2 + 4}.
    \end{align}
\end{lemma}

    The proofs of lemmas \ref{lem:derivative} and \ref{lem:uniformHessian} are relegated to Section \ref{sec:preliminary_lemmas}. 

    \subsubsection{Sketch of proof of Theorem \ref{thm:DiagDominant}}
    \label{sec:sketch_proof_offdiagonal}

    As we mentioned before, the proof of Theorem \ref{thm:DiagDominant} is the most challenging aspect of this work, and most of the difficulty lies in analyzing the off-diagonal entries of the Hessian of the log-likelihood function $\ell(\hparam;\sigma_{L})$ in \eqref{eq:def_log_likelihood}. 
    %We begin our analysis on the off-diagonal entries  
    %by observing the following. 
    The off-diagonal entries are written as the product of strongly correlated random variables with large variances and the length of the product is proportional to the shortest-path distance between the two edges indexing the off-diagonal entry. Controlling such large product is the main challenge in the analysis. 

    To be more precise,  first observe that, under \eqref{eqn:pHatBounds} in Assumption \ref{assumption1}, the expression \eqref{eqn:hessianTerms} in Lemma \ref{lem:derivative} for the off-diagonal entries in the Hessian  yields the following upper bound 
\begin{align}
\label{eqn:hessianBound1}  \bigg|&\frac{\partial^2}{\partial\htheta_e\partial\htheta_f} \ell(\hparam;\sigma|_{L})\bigg|\le  \frac{(1-2c_{\ref{eqn:pHatBounds}}\delta)^{N}}{(1+\htheta_e Z_x Z_{y})^2} \prod_{j=0}^{N} \frac{1-(\htheta_{\{y_{j},w_j\}} Z_{w_j})^2}{(1+\htheta_{\{y_{j},w_j\}} Z_{w_j}\htheta_{\{y_{j},y_{j-1}\}} Z_{y_{j-1}})^2},
\end{align}
recalling that $N=\textup{dist}(e,f)$ denotes the shortest-path distance between the ends of the edges $e$ and $f$ (see Figure \ref{fig:TREEhess}\textbf{b}).
We next introduce some notation to further simplify the bound in \eqref{eqn:hessianBound1}. 
For each $0\le j \le N$, the internal node $y_{i}$ in the path $\gamma$ has a unique neighbor, say $w_{j}$, that is not in $\gamma$.
Each $w_{j}$ is associated with a descendant subtree $T_{w_{j}}$ (see Fig. \ref{fig:TREEhess}\textbf{b}), which defines the magnetization $Z_{w_{j}}$ at $w_{j}$. We introduce the following notation
\begin{align}\label{eq:def_eta_xi}
	\hspace{-0.5cm} \htheta_j:=	\htheta_{\{y_j,y_{j+1}\}}, \quad \eta_j := \htheta_{\{y_j, w_j\}}Z_{w_j},\quad 	\xi_{j+1} := 
	\htheta_j q(\eta_j,\xi_j), \quad \xi_{0} := \hat{\theta}_{-1}Z_{y_{-1}}
\quad  \eta_{N+1}:= Z_{x},
\end{align} 
where $q(\cdot,\cdot)$ is the bivariate function introduced in \eqref{eq:q_def_recursion}. We will call the random variables $\xi_{j}$ and $\eta_{j}$ as \textit{signals}. 
According to the decomposition in Figure \ref{fig:TREEhess}, the information (in terms of magnetization) in the subtree $T_{y}=T_{y_{N}}$ flows toward the root $y=y_{N}$ (see Figure \ref{fig:Tree4Terms}). Such information flow can be understood as follows. First, $y_{0}$  sends signal $\xi_{1}$ to $y_{1}$, which also receives signal $\eta_{1}$ from $w_{1}$. Then by the recursion for magnetization \eqref{eqn:recursionBorg2}, this yields magnetization $Z_{y_{1}}$ as 
\begin{align}\nonumber 
	Z_{y_{1}} = q(\xi_{1},\eta_{1}) = \frac{\xi_{1}+\eta_{1}}{1+\xi_{1}\eta_{1}}. 
\end{align}
Then $y_{1}$ sends signal $\xi_{2}=\hat{\theta}_{1}Z_{y_{1}}$ to $y_{2}$, which is then combined with the signal $\eta_{2}$ from $w_{2}$ by the same recursion, and so on. This recursively defines the signals $\eta_{1},\dots,\eta_{N+1}$ and $\xi_{1},\dots,\xi_{N+1}$ as in Figure \ref{fig:Tree4Terms}. 
\begin{figure}[b!]
	\centering
	\includegraphics[width=0.8 \linewidth]{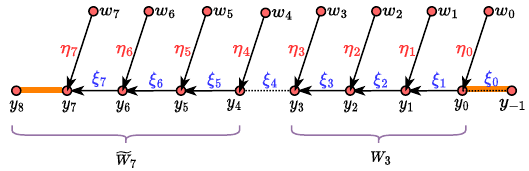}
	\caption{
    Propagation of signals from the edge $f=\{y_{0},y_{-1}\}$ to $e=\{y_{N},y_{N+1}\}$ when $N=7$. The signal $\xi_{j+1}$ that $y_{j+1}$ receives from $y_{j}$ is the scalar multiple $\hat{\theta}_{j}$ of the magnetization  $Z_{y_{j}}=q(\xi_{j},\eta_{j})$ at $y_{j}$. 
    }
	\label{fig:Tree4Terms}
\end{figure}

Using the notation introduced above and bounding $(1-2c_{\ref{eqn:pHatBounds}}\delta)^{N}\le 1$, we can rewrite the bound \eqref{eqn:hessianBound1} on the off-diagonal entries of the Hessian as 
\begin{align}\label{eq:Hessian_offdiag_bd} 
	\bigg| \frac{\partial^2}{\partial\htheta_e\partial\htheta_f} \ell(\hparam;\sigma|_{L})\bigg|
    \le   \frac{1}{(1+ \xi_{N+1} \eta_{N+1} )^2} \prod_{{j=0}}^{N} 
	\frac{1-\eta_{j}^2}{(1+\eta_{j} \xi_{j} )^2}.
\end{align}
Due to the recursion in  \eqref{eq:def_eta_xi} that $\xi_{j+1}$ satisfies, the $N$-fold product on the right-hand side above is in fact a function only of  $\xi_{1},\eta_{1},\dots,\eta_{N}$. Note that by Claim \ref{claim:unsigned_mag}, $\sigma_{w_1}\eta_{1},\dots,\sigma_{w_N}\eta_{N}$ are independent of each other and of the signals $\sigma_{y}$ for $y\in \gamma$. However, the terms in the product are far from being independent since the chain of recursions mixes up these independent variables. This is the key difficulty that we face in analyzing the off-diagonal entries in the Hessian. We will handle this issue by grouping four consecutive terms in the product, while regarding the first $\xi_{j}$ term in each subgroup as a non-random `adversarial input' $\xi_{j}^{\circ}$ for the four-step recursion. We now make this sketch more precise.

Note that, under \eqref{eqn:pHatBounds} in Assumption \ref{assumption1}, we have $\htheta_{\{a,b\}}\le 1-2c_{\ref{eqn:pHatBounds}}\delta$ for all edges $\{a,b\}$ and $Z_{s}\in[-1,1]$ for any vertex $s$, so we have the 
bounds 
\begin{align}\label{eqn:xiandetabound}
	|\xi_{j}|, |\eta_{j}|  \le 1-2c_{\ref{eqn:pHatBounds}}\delta.
\end{align}
Now consider the following product of four consecutive terms in the product in \eqref{eq:Hessian_offdiag_bd}:
\begin{align}
    \prod_{j=i-3}^{i} \frac{(1-\eta_j^2)}{(1+\xi_j\eta_j)^2} \qquad \textup{for $3\le i \le N$}. 
\end{align}
The above is a function of random variables $\xi_{i-3},\eta_{i-3},\eta_{i-2},\eta_{i-1},\eta_{i}$ since they determine the value of $\xi_{i-2},\xi_{i-1},\xi_{i}$ through  the recursion $\xi_{j+1} := 
	\htheta_j q(\eta_j,\xi_j)$. Since $|\xi_{i-3}|\le  1-2c_{\ref{eqn:pHatBounds}}\delta$, if we fix a deterministic value $x$ with absolute value $\le  1-2c_{\ref{eqn:pHatBounds}}\delta$ and define random variables $\xi_{i-2}^{\circ},\xi_{i-1}^{\circ},\xi_{i}^{\circ}$ using the recursion $\xi_{j+1}^{\circ} := 
	\htheta_j q(\eta_j,\xi_j^{\circ})$ for $j=i-3,i-2,i-1$ with $\xi_{i-3}^{\circ}=x$, then 
    \begin{align}\label{eq:def_W_i}
   W_{i}=W_{i}(\eta_{i-3},\eta_{i-2},\eta_{i-1},\eta_i) &: =   \sup_{|x|\le 1-2c_{\ref{eqn:pHatBounds}}\delta}  \prod_{j=i-3}^{i} \frac{(1-\eta_j^2)}{(1+\xi_j^{\circ}\eta_j)^2}\\
   &\ge  \prod_{j=i-3}^{i} \frac{(1-\eta_j^2)}{(1+\xi_j\eta_j)^2}  \quad \textup{for $3\le i \le N$},
    \end{align}
    where the inequality above holds almost surely with respect to the randomness of $\xi_{i-3},\eta_{i-3},\eta_{i-2},\eta_{i-1},\eta_{i}$. In the above we have defined random variables $W_{i}$ for $3\le i \le N$.  Similarly, define 
    \begin{align}
      \widetilde{W}_N&:=\widetilde{W}_i(\eta_{N-3},\eta_{N-2},\eta_{N-1},\eta_N,\eta_{N+1}) \\
      &=\sup_{|x| \le 1-2c_{\ref{eqn:pHatBounds}}\delta}\frac{1}{(1+\xi^\circ_{N+1}\eta_{N+1})^2}\prod_{j=N-3}^N\frac{(1-\eta_j^2)}{(1+\xi^\circ_j\eta_j)^2} ,\label{eq:def_widetildeW_i} \\
     &\qquad \textup{where $\xi_{j+1}^\circ := 
    \htheta_j q(\eta_j,\xi^\circ_j)$ for $j=N-3,N-2,N-1$ with $\xi^\circ_{N-3}=x$}. \nonumber 
    \end{align}
    Then the following inequality holds almost surely: 
    \begin{align*}
         \frac{1}{(1+\xi_{N+1}\eta_{N+1})^2}\prod_{j=N-3}^N\frac{(1-\eta_j^2)}{(1+\xi_j\eta_j)^2} \le \widetilde{W}_N. 
    \end{align*}
    Lastly, denote 
    \begin{align}
    R_r&:= R_r(\xi_0,
    \eta_0,\dotsm,\eta_{r-1}) =  \prod_{j=0}^{r-1} \frac{(1-\eta_j^2)}{(1+\xi_j\eta_j)^2} \quad  \textup{for } 0\le r\le 3, \label{eq:def_R_N}
    \end{align}
    where we set $R_0 \equiv 1$ using the convention of setting the empty product to one. Note that 
    in the definition of $R_{r}$, we have not replaced the random variable $\xi_{0}$ with a deterministic adversarial input as we did in the definition of $W_{i}$ and $\widetilde{W}_{i}$ above.

With \eqref{eq:Hessian_offdiag_bd} and the notations introduced above, we deduce the following lemma.

\begin{lemma}\label{lem:off_diagonal_ind_bd} For each edge $e\neq f$, and the notation in \eqref{eq:Hessian_offdiag_bd}--\eqref{eq:def_widetildeW_i}, with probability 1 it holds that
\begin{align}\label{eq:Hessian_offdiag_bd_bd}
	\bigg| \frac{\partial^2}{\partial\htheta_e\partial\htheta_f} \ell(\hparam;\sigma|_{L})\bigg| 
	\le	    \widetilde{W}_N W_{N-4}W_{N-8}\dotsm W_{r+3}R_r\qquad \textup{ where } r := (N+1)\mod 4. 
\end{align} 
\end{lemma} 
\noindent Note that $N\equiv r+3 \, (\textup{mod 4})$ and so each term in the product on the right-most term in \eqref{eq:Hessian_offdiag_bd} is accounted for in the bound exactly once. See the illustration in Figure \ref{fig:Tree4Terms}.

The significance of the bound in \eqref{eq:Hessian_offdiag_bd_bd} is that the constituent random variables are all independent from each other as stated in Lemma \ref{lem:independentWj}.  This is in contrast to the long-range correlations (through chains of recursions) between terms in the upper bound in \eqref{eq:Hessian_offdiag_bd}. Its proof can be found in Section \ref{sec:pf_off_diagonal}.

\begin{lemma}\label{lem:independentWj} 
Suppose that $N\ge 4$ and write $r = r_N =(N+1)\mod 4$.
    Then the random variables $\widetilde{W}_N, W_{N-4}, \dotsm, W_{r+3}$ and $R_r$ are independent of each other.
\end{lemma}

It is important to note that the random variables that appear in the expression on the right-hand side of \eqref{eq:Hessian_offdiag_bd_bd} above are independent by Lemma \ref{lem:independentWj} so that we can easily bound its moments. 
In other words, we have decoupled the long chain of dependence into independent blocks by ``adversarializing'' the input to each block of recursions. In the following sections, we will show that these variables have small expectation and finite variance. This will imply that their product has mean exponentially small in $N$ %(the distance between the two edges $e$ and $f$) 
and finite variance. The former is enough to deduce the population level result (Theorem \ref{cor:MLE_landscape}), while the latter is used to deduce finite-sample results.

	\section{Preliminary Lemmas}
\label{sec:preliminary_lemmas}	
    In this section, we establish some preliminary lemmas on the magnetization and derivatives of the log-likelihood function.

    We first prove Lemma \ref{lem:derivative} on the derivatives of the log-likelihood function using magnetization.

	\begin{proof}[\textbf{Proof of Lemma \ref{lem:derivative}}]

		In fact, part \textbf{(i)} is already established in \cite{clancy2025likelihood}.
  
		Next we turn to \eqref{eqn:hessianTerms} in part \textbf{(ii)}, which follows from further differentiating the first derivative above and using the chain rule. 
		This time let us fix two edges in the tree $T$, $e = \{x,y\}$ and $f = \{u,v\}$. Without loss of generality, we suppose that the edge $f\in E(T_y)$, where $T_y$ is the subtree of $T\setminus\{e\}$ rooted at $y$, and that $u$ is the end of $f$ closer to $y$. Enumerate the vertices in the path from $y$ to $u$ by $y = y_N, y_{N-1},\dotsm, y_1, y_0 = u$ where $d(u,y) = N$. We write $y_{-1} = v$ and $y_{N+1} = x$. Note that for every $j = 0,1,\dots, N$, the vertex $y_j$ has neighbors $y_{j+1}, y_{j-1}$ and a third vertex which we denote by $w_j$. See Figure~\ref{fig:TREEhess}.
		
		For a vertex $a\in T_y$, we write $Z_a$ for the magnetization of $a$ w.r.t. the  descendant subtree of $T_y$ obtained by removing the edge $\{a,b\}$, where $b$ is the unique neighbor of $a$ which is closer to $y$. We write $Z_x$ for the root magnetization in the subtree $T_x$ of  $T\setminus\{e\}$ rooted at $x$. 
		We will use that, for every $j = 0,1,\dots, N$, 
		\begin{equation}\label{eqn:recursionZyj}
			Z_{y_j} = q\left(\htheta_{\{y_j,w_j\}} Z_{w_j} , \htheta_{\{y_j,y_{j-1}\}} Z_{y_{j-1}} \right),
		\end{equation}
		by~\eqref{eqn:recursionBorg2}.
		We note that 
		\begin{equation}\label{eqn:derivativeQ1}
			\frac{\partial}{\partial y}\frac{xy}{1+axy} = \frac{x}{(1+axy)^2}
		\end{equation}
		and 
		\begin{equation}\label{eqn:derivativeQ2}
			\frac{\partial}{\partial y}q(ax,by) = b\frac{1-(ax)^2}{(1+abxy)^2}.
		\end{equation}
		
		We use \eqref{eqn:derivative} (with $x,y$ rather $u,v$) and observe that $Z_x$ does not depend on $\htheta_f$
		to get
		\begin{align}\label{eq:Hessian_second_Der_pf1}
			\frac{\partial^2}{\partial \htheta_e \partial \htheta_f}\ell(\hparam; \sigma) &= \frac{\partial}{\partial\htheta_f} \frac{Z_xZ_y}{1+\htheta_e Z_xZ_y}  = \frac{ Z_x}{(1+\htheta_e Z_x Z_y)^2} \frac{\partial}{\partial\htheta_f} Z_y.
		\end{align} 
  Since $Z_{y} = Z_{y_N}$ we can use \eqref{eqn:recursionZyj} and the derivative \eqref{eqn:derivativeQ2} to get
  \begin{align*}
      \frac{\partial}{\partial \htheta_f}Z_{y_N} =  \htheta_{\{y_N,y_{N-1}\}} \frac{ \left(1-(\htheta_{\{y_N,w_N\}} Z_{w_N})^2\right)}{\left(1+\htheta_{\{y_N,w_N\}}\htheta_{\{y_N,y_{N-1}\}} Z_{w_N}Z_{y_{N-1}}\right)^2} 
			\frac{\partial}{\partial \htheta_f} Z_{y_{N-1}}.
  \end{align*}
		Continuing by induction, this is
		\begin{align}\nonumber
			\nonumber&=  \frac{\htheta_{\{y_N,y_{N-1}\}} \left(1-(\htheta_{\{y_N,w_N\}} Z_{w_N})^2\right)}{\left(1+\htheta_{\{y_N,w_N\}}\htheta_{\{y_N,y_{N-1}\}} Z_{w_N}Z_{y_{N-1}}\right)^2}
			\frac{\partial}{\partial \htheta_f} q\left(\htheta_{\{y_{N-1},w_{N-1}\}} Z_{w_{N-1}}, \htheta_{\{y_{N-1},y_{N-2}\}} Z_{y_{N-2}}\right)\\
			\nonumber&=\cdots\\
			\nonumber&=  \prod_{j=1}^{N} \frac{\htheta_{\{y_j,y_{j-1}\}} \left(1- (\htheta_{\{y_j,w_j\}} Z_{w_j})^2 \right)}{\left(1+\htheta_{\{y_j,w_j\}} \htheta_{\{y_j,y_{j-1}\}} Z_{w_j} Z_{y_{j-1}}\right)^2} \frac{\partial}{\partial \htheta_f} Z_u\\
			\nonumber &= Z_v\prod_{j=1}^N \htheta_{\{y_j,y_{j-1}\}} \prod_{j=0}^{N} \frac{ \left(1- (\htheta_{\{y_j,w_j\}} Z_{w_j})^2 \right)}{\left(1+\htheta_{\{y_j,w_j\}} \htheta_{\{y_j,y_{j-1}\}} Z_{w_j} Z_{y_{j-1}}\right)^2}.
		\end{align}
		The last equality used $u = y_0, v = y_{-1}$ and 
		\begin{align*}
			\frac{\partial}{\partial \htheta_f} Z_u &= \frac{\partial}{\partial \htheta_f} q( \htheta_{\{y_0,w_0\}}Z_{w_0}, Z_v \htheta_f)= Z_v\frac{1-(\htheta_{\{y_0,w_0\}} Z_{w_0})^2}{(1+\htheta_{\{y_0,w_0\}}\htheta_{\{y_{0},y_{-1}\}} Z_{w_0}Z_{y_{-1}})^2},
		\end{align*}
		and the fact that 
		$Z_v$ does not depend on 
		$\htheta_f$. This shows the formula \eqref{eqn:Z_derivative} for $\frac{\partial}{\partial \htheta_f} Z_{y}$. Then by using \eqref{eq:Hessian_second_Der_pf1}, we obtain \eqref{eqn:hessianTerms}, as desired.

        Now we show \textbf{(iii)}. The claimed bound \eqref{eqn:third_order_der_bd} trivially holds for all $e_1=e_2=e_3$, and using Lemma \ref{lem:derivative}\textbf{(i,ii)} it is also easily checked (using the symmetry of mixed partial derivatives) if $e_i = e_j = e$ and $e_k = f$ for distinct $i,j,k$ since
\begin{align*}
\frac{\partial^3}{\partial\htheta_f\partial \htheta_e^2 } \ell(\hparam;\sigma) &= \frac{\partial^2}{\partial\htheta_f\partial \htheta_e} \frac{Z_xZ_y}{(1+\htheta_e Z_xZ_y)^2} = \frac{\partial}{\partial \htheta_f} \frac{-2Z_x^2Z_y^2}{(1+\htheta_e Z_xZ_y)^3}\\
&=\frac{-2Z_x^2 Z_y(2-\htheta_eZ_xZ_y)}{(1+\htheta_e Z_xZ_y)^4} \frac{\partial}{\partial \htheta_f} Z_y.
\end{align*}

Observe that in \eqref{eqn:Z_derivative}, since each of the denominators is at least $(2c_{\ref{eqn:pHatBounds}}\delta)^2$, we get 
        \begin{align}
         \left|\frac{\partial}{\partial \htheta_f}Z_y\right| \le \frac{1}{(2c_{\ref{eqn:pHatBounds}}\delta)^{2 d}},
     \end{align}
    where $d $ denotes the maximum distance from $y$ to a leaf in $T_y$. Hence we have  
\begin{equation*}
    \left|\frac{\partial^3}{\partial\htheta_f\partial \htheta_e^2 } \ell(\hparam;\sigma)\right| \le \frac{6}{(2c_{\ref{eqn:pHatBounds}}\delta)^4} \frac{1}{(2c_{\ref{eqn:pHatBounds}}\delta)^{2d}} = 6 (2c_{\ref{eqn:pHatBounds}} \delta)^{-2\operatorname{diam}(T)-4}.
\end{equation*}
Therefore, we just check when $e_1,e_2,e_3$ are distinct. There are two cases to consider (again using symmetry of mixed partials). Say $e_3 = \{x,y\}$. Either $e_1\in T_x$ and $e_2\in T_y$ or both $e_1,e_2\in T_y$.  

In the case of the former, we see
\begin{align*}
    \frac{\partial^3}{\partial\htheta_{e_1}\partial \htheta_{e_2}\partial \htheta_{e_3}} \ell(\hparam;\sigma) = \frac{\partial^2}{\partial\htheta_{e_1}\partial \htheta_{e_2}} \frac{Z_xZ_y}{1+\htheta_{e_3} Z_xZ_y} = \frac{1}{(1+\htheta_{e_3}Z_xZ_y)^2} \left( Z_x \frac{\partial}{\partial \htheta_{e_1}} Z_x + Z_y \frac{\partial}{\partial \htheta_{e_2}} Z_y\right)
\end{align*}
and, since the sum of distance from $x$ to any leaf in $T_x$ and the distance from $y$ to any leaf in $T_y$ is at most the diameter,
\begin{align*}
     \left|\frac{\partial^3}{\partial\htheta_{e_1}\partial \htheta_{e_2}\partial \htheta_{e_3}} \ell(\hparam;\sigma) \right| = \frac{1}{(2c_{\ref{eqn:pHatBounds}}\delta)^2} \frac{2}{(2c_{\ref{eqn:pHatBounds}} \delta)^{\operatorname{diam}(T)}}. 
\end{align*}

The latter case is a bit harder. In this case we will assume that $e_3 = e = \{x,y\}$ and $e_2 = f = \{u,v\}$ as in Lemma \ref{lem:derivative}. Thus
\begin{align*}
&     \frac{\partial^3}{\partial\htheta_{e_1}\partial \htheta_{e_2}\partial \htheta_{e_3}} \ell(\hparam;\sigma)  = \frac{\partial}{\partial \htheta_{e_1}} \frac{ Z_x Z_v}{(1+\htheta_e Z_x Z_y)^2} \prod_{j=1}^N \htheta_{\{y_j,y_{j-1}\}} \prod_{j=0}^{N} \frac{ \left(1- (\htheta_{\{y_j,w_j\}} Z_{w_j})^2 \right)}{\left(1+\htheta_{\{y_j,w_j\}} \htheta_{\{y_j,y_{j-1}\}} Z_{w_j} Z_{y_{j-1}}\right)^2}
\end{align*}
Note, we can suppose that $e_1 \neq \{y_j,y_{j-1}\}$ for any of the edges on the path from $e$ to $f$ as well as also that $e_1\notin T_v$ as these are covered by the previous case. Therefore either $e_1 = \{y_k,w_k\}$ for some $k$ or $e_1\in T_{w_k}$ for some $k$. If it is an edge $e_1 = \{y_k,w_k\}$ then $Z_{y_j}$ depends on $\htheta_{e_1}$ for $j\ge k$ but none of the other magnetizations appearing in the right-hand side of the above equation; while if $e_1\in T_{w_k}$ then there is the additional dependence on $Z_{w_k}$. Note that for any $N\ge j>k$ and $e_1 = \{y_k,w_k\}$ it holds that
\begin{align*}
    \frac{\partial}{\partial\htheta_{e_1}}  &\frac{1-(\htheta_{\{y_j,w_j\}} Z_{w_j})^2}{(1+\htheta_{\{y_j,w_j\}} \htheta_{\{y_j,y_{j-1}\}} Z_{w_j} Z_{y_{j-1}})^2} \\
    &= \frac{2\left(1-(\htheta_{\{y_j,w_j\}} Z_{w_j})^2\right) \htheta_{\{y_j,w_j\}} \htheta_{\{y_j,y_{j-1}\}} Z_{w_j} \left(1-(\htheta_{\{y_j,w_j\}} \htheta_{\{y_j,y_{j-1}\}} Z_{w_j})^2\right)}{(1+\htheta_{\{y_j,w_j\}} \htheta_{\{y_j,y_{j-1}\}} Z_{w_j} Z_{y_{j-1}})^3} \frac{\partial}{\partial \htheta_{e_1}} Z_{y_j}\\
    &= \frac{1-(\htheta_{\{y_j,w_j\}} Z_{w_j})^2}{(1+\htheta_{\{y_j,w_j\}} \htheta_{\{y_j,y_{j-1}\}} Z_{w_j} Z_{y_{j-1}})^2} E_1\qquad\textup{(say)}
\end{align*} as well as
\begin{align*}
    \frac{\partial}{\partial\htheta_{e_1}}  &\frac{1-(\htheta_{\{y_k,w_k\}} Z_{w_k})^2}{(1+\htheta_{\{y_k,w_k\}} \htheta_{\{y_k,y_{k-1}\}} Z_{w_k} Z_{y_{k-1}})^2} =-\frac{2Z_{w_k} (\htheta_{\{y_{k},w_k\}} Z_{w_k} + \htheta_{\{y_k,y_{k-1}\}}Z_{y_{k-1}})}{(1+\htheta_{\{y_k,w_k\}} \htheta_{\{y_k,y_{k-1}\}} Z_{w_k} Z_{y_{k-1}})^3}
    \\&= \frac{1-(\htheta_{\{y_k,w_k\}} Z_{w_k})^2}{(1+\htheta_{\{y_k,w_k\}} \htheta_{\{y_k,y_{k-1}\}} Z_{w_k} Z_{y_{k-1}})^2} E_2
    \\
    \frac{\partial}{\partial\htheta_{e_1}}  & \frac{Z_xZ_v}{(1+\htheta_e Z_xZ_y)^2} = \frac{-2Z_x^2 Z_vZ_y}{(1+\htheta_e Z_xZ_y)^3} \frac{\partial}{\partial_{e_1}}Z_y =  \frac{Z_xZ_v}{(1+\htheta_e Z_xZ_y)^2} E_3 \textup{ (say).}
\end{align*} Here $E_1,E_2, E_3$ are (signed) errors satisfying
\begin{align*}
    \left|E_1\right|&\le \frac{2}{(2c_{\ref{eqn:pHatBounds}}\delta)} \times\frac{1}{(2c_{\ref{eqn:pHatBounds}}\delta)^{2\operatorname{diam}(T)}}\\
    \left|E_2\right|&\le \frac{4}{(2c_{\ref{eqn:pHatBounds}}\delta)^2}\\
    \left|E_3\right|&\le \frac{2}{(2c_{\ref{eqn:pHatBounds}}\delta)}\times \frac{1}{(2c_{\ref{eqn:pHatBounds}}\delta)^{2\operatorname{diam}(T)}}.
\end{align*}
Therefore by the product rule when $e_1 = \{y_k,w_k\}$ then
\begin{align*}
    &\left|\frac{\partial^3}{\partial\htheta_{e_1}\partial \htheta_{e_2}\partial \htheta_{e_3}} \ell(\hparam;\sigma)\right| \\
    &\qquad = \left|\frac{\partial}{\partial \htheta_{e_1}} \frac{ Z_x Z_v}{(1+\htheta_e Z_x Z_y)^2} \prod_{j=1}^N \htheta_{\{y_j,y_{j-1}\}} \prod_{j=0}^{N} \frac{ \left(1- (\htheta_{\{y_j,w_j\}} Z_{w_j})^2 \right)}{\left(1+\htheta_{\{y_j,w_j\}} \htheta_{\{y_j,y_{j-1}\}} Z_{w_j} Z_{y_{j-1}}\right)^2}\right|\\
    &\qquad =\left|\frac{ Z_x Z_v}{(1+\htheta_e Z_x Z_y)^2} \prod_{j=1}^N \htheta_{\{y_j,y_{j-1}\}} \prod_{j=0}^{N} \frac{ \left(1- (\htheta_{\{y_j,w_j\}} Z_{w_j})^2 \right)}{\left(1+\htheta_{\{y_j,w_j\}} \htheta_{\{y_j,y_{j-1}\}} Z_{w_j} Z_{y_{j-1}}\right)^2}  \right|\\
    &\qquad \qquad \times \left(\operatorname{diam}(T) \max(|E_1|,|E_2|,|E_3|)\right)\\
    &\qquad \le \frac{4\operatorname{diam}(T)}{(2c_{\ref{eqn:pHatBounds}}\delta)^{4\operatorname{diam}(T)+1}} 
\end{align*} where the exponent in the denominator is $2\operatorname{diam}(T) + (2\operatorname{diam}(T) + 1)$ which is the worst-case bound for each of the denominators in the Hessian and maximum bound from $E_3$, respectively. 

The bound for whenever $e_1\in T_{w_k}$ is similar, except the error for the corresponding ``$E_2$ term'' is the same as the bound for $E_3$ above. We omit the details.
\end{proof}

Next, we prove the uniform bound \eqref{eqn:uniformbound} on the entries of the Hessian stated in Lemma \ref{lem:uniformHessian}. 

\begin{proof}[\textbf{Proof of Lemma \ref{lem:uniformHessian}}]
        We start with the simple observations that from \eqref{eqn:hessianTerms} and the fact that $\hparam\in [0,1]^E$, and $1+\hparam Z_xZ_y\ge 2c_{\ref{eqn:pHatBounds}}\delta$ that
\begin{align*}
			&\left| \frac{\partial^2}{\partial \htheta_e\partial \htheta_f} \ell(\hparam;\sigma)\right| \le  \left( \frac{\htheta_e Z_x Z_v}{(1+\htheta_e Z_x Z_y)^2} \prod_{j=1}^N\htheta_{\{y_j,y_{j-1}\}} \right)
            \prod_{j=0}^{N} \frac{ \left(1- (\htheta_{\{y_j,w_j\}} Z_{w_j})^2 \right)}{\left(1+\htheta_{\{y_j,w_j\}} \htheta_{\{y_j,y_{j-1}\}} Z_{w_j} Z_{y_{j-1}}\right)^2}\\
            &\le \frac{1}{(2c_{\ref{eqn:pHatBounds}}\delta)^2}\prod_{j=0}^{N} \frac{ \left(1- (\htheta_{\{y_j,w_j\}} Z_{w_j})^2 \right)}{\left(1+\htheta_{\{y_j,w_j\}} \htheta_{\{y_j,y_{j-1}\}} Z_{w_j} Z_{y_{j-1}}\right)^2}.
\end{align*}
By Proposition \ref{prop:twoTerms}, there is some constant $C>0$ (depending only the constants in \ref{assumption1}) such that
\begin{align*}
    \frac{ \left(1- (\htheta_{\{y_j,w_j\}} Z_{w_j})^2 \right)}{\left(1+\htheta_{\{y_j,w_j\}} \htheta_{\{y_j,y_{j-1}\}} Z_{w_j} Z_{y_{j-1}}\right)^2} \frac{ \left(1- (\htheta_{\{y_{j+1},w_{j+1}\}} Z_{w_{j+1}})^2 \right)}{\left(1+\htheta_{\{y_{j+1},w_{j+1}\}} \htheta_{\{y_{j+1},y_{j}\}} Z_{w_{j+1}} Z_{y_{j}}\right)^2}\le \frac{C}{\delta}
\end{align*}
for all $j = 0,1,2,\dotsm, N-1$. 
Note that for each $j$ we have $ \left(1+\htheta_{\{y_j,w_j\}} \htheta_{\{y_j,y_{j-1}\}} Z_{w_j} Z_{y_{j-1}}\right)\ge 2c_{\ref{eqn:pHatBounds}}\delta$ since $\htheta_e\le 1-2c_{\ref{eqn:pHatBounds}}\delta$ and $Z_v\in[-1,1]$ for all $v$. It follows that
\begin{align}
    \left|\prod_{j=0}^{N} \frac{ \left(1- (\htheta_{\{y_j,w_j\}} Z_{w_j})^2 \right)}{\left(1+\htheta_{\{y_j,w_j\}} \htheta_{\{y_j,y_{j-1}\}} Z_{w_j} Z_{y_{j-1}}\right)^2}\right| &\le (C/\delta)^{N/2} \textup{ when }N \textup{ is odd},\\
    \left|\prod_{j=0}^{N} \frac{ \left(1- (\htheta_{\{y_j,w_j\}} Z_{w_j})^2 \right)}{\left(1+\htheta_{\{y_j,w_j\}} \htheta_{\{y_j,y_{j-1}\}} Z_{w_j} Z_{y_{j-1}}\right)^2}\right| &\le (C/\delta)^{ N-1/2} \frac{1}{(2c_{\ref{eqn:pHatBounds}}\delta)^2} \textup{ when }N \textup{ is even}.
\end{align}
The desired statement follows from $N\le \operatorname{diam}(T)$.
\end{proof}

	%\subsection{Insensitivity of magnetization to parameters} \label{ssec:MagLemmas}
	
 Next, we recall a key result from Clancy, Lyu, Roch, and Sly \cite[Thm. 2.3]{clancy2025likelihood} about magnetization is the following insensitivity to parameters. 

\begin{theorem}[Insensitivity of magnetization to parameters]\label{thm:Robust} There exist constants $\delta_{\ref{eqn:antiReconstruction}}, 
c_{\ref{eqn:Reconstruct}},C_{\ref{eqn:Reconstruct}}, c_{\ref{eqn:antiReconstruction}}, C_{\ref{eqn:antiReconstruction}}>0$ depending only on the constants in \ref{assumption1} such that the following holds for any unrooted binary tree $T$ and $\delta\le \delta_{\ref{eqn:antiReconstruction}}$. 
	Fix a descendant subtree $T_{u}$ of a node $u$, and let $L_{u}$ denote the set of all leaves in $T_{u}$ and suppose that $\hparam\in \hParam_0(\delta)$ and $\param^*\in \Param_0(\delta)$. 
    \vspace{0.1cm}
	\begin{description}
		\item[(i)] (\textit{Upper tail}) 
        \vspace{-0.3cm} 
		\begin{equation}\label{eqn:Reconstruct}
			\PR_{\param^{*}}\left(\sigma_u Z_u^{\hparam, T_{u}}(\sigma_{L_{u}}) 
   \ge 1-C_{\ref{eqn:Reconstruct}} \delta^2  \right) \ge 1- c_{\ref{eqn:Reconstruct}}\delta. 
		\end{equation} 
        \vspace{0.1cm}

		\item[(ii)] (\textit{Lower tail}) 
		\vspace{-0.3cm}
    \begin{equation}\label{eqn:antiReconstruction}
			\PR_{\param^{*}} \left(\sigma_u Z_u^{\hparam, T_{u}}(\sigma_{L_{u}}) 
   \le - c_{\ref{eqn:antiReconstruction}}  \right) \le C_{\ref{eqn:antiReconstruction}}\delta^2.
		\end{equation}
	\end{description}
\end{theorem} 

	To conceptualize Thm. \ref{thm:Robust}, we introduce the following trichotomy of magnetization. We say we have a `good reconstruction' at node $u$ if $\sigma_{u}Z_{u} \ge 1-C_{\ref{eqn:Reconstruct}} \delta^2$, `severe failure' if $\sigma_{u}Z_{u}\le  -c_{\ref{eqn:antiReconstruction}}$, and `moderate failure' otherwise. Then Thm. \ref{thm:Robust} states the following probability bounds for each of the three tiers of magnetization.   
	\begin{align}\label{eq:recons_tiers}
		\begin{cases}
			\textup{Good reconstruction at $u$} \quad \Longleftrightarrow \quad 	\sigma_{u}Z_{u} \ge 1-C_{\ref{eqn:Reconstruct}} \delta^2 &\textup{(with prob. $\ge 1- c_{\ref{eqn:Reconstruct}}\delta$)} \\
			\textup{Moderate failure at $u$}  \quad \Longleftrightarrow \quad 	\sigma_{u}Z_{u}\in ( -c_{\ref{eqn:antiReconstruction}}, 1-4C_{\ref{eqn:Reconstruct}}\delta^2) &\textup{(with prob. $\le c_{\ref{eqn:Reconstruct}}\delta$)} \\
			\textup{Severe failure at $u$}  \quad \Longleftrightarrow \quad 	\sigma_{u}Z_{u}\le  -c_{\ref{eqn:antiReconstruction}} 
			&\textup{(with prob. $\le C_{\ref{eqn:antiReconstruction}}\delta^{2}$)}
		\end{cases}
	\end{align}
    This result is a crucial ingredient for our analysis in Sections \ref{sec:diagonal_Hessian_pf} and \ref{sec:Hessian_off_diag_pf}, and the supplement to \cite{CLR:25}.

 An important property of magnetization at distinct nodes is that they are conditionally independent given the spins at intermediate. More precisely, suppose we have two node-disjoint descendant subtrees $T_{u}$  and $T_{v}$. Then for any node $w$ along the shortest path between $u$ and $v$, 
		\begin{align}\label{eq:Z_conditional_independence}
			Z_{u} \coprod   Z_{v} \,|\, \sigma_{w}. 
		\end{align}
		which follows from the Markov property of the CFN model and the fact that $Z_{u}$ is determined by $\sigma_{L_{u}}$. In fact, the 
        `unsigned magnetizations' $\sigma_{u}	Z_{u}$ are independent as long as the supporting descendant subtrees are node-disjoint, as stated in Claim 3.1 from \cite{clancy2025likelihood} which we now recall.

	\begin{claim}[Independence of unsigned magnetization]\label{claim:unsigned_mag}
		The following hold: 
		\begin{description}[leftmargin=0.8cm] 
			\item[(i)] (\textit{Independence from root spin}) Let $u$ be a node in $T$ with a descendant subtree $T_{u}$ and corresponding magnetization $Z_{u}$. Then the  `unsigned magnetization' $\sigma_{u}Z_{u}$ is independent from $\sigma_{u}$ under $\PR_{\param^{*}}$. 

            \vspace{0.1cm}
			\item[(ii)] (\textit{Independence between unsigned magnetizations})  Let $v_{1},\dots,v_{k}$ be nodes in $T$ and suppose there are corresponding descendant subtrees $T_{v_{1}},\dots,T_{v_{k}}$ that are node-disjoint. Let $Z_{v_{i}}$ denote the corresponding magnetization at $v_{i}$ for $i=1,\dots, k$.
			Then  $\sigma_{v_{i}} Z_{v_{i}}$s for $i=1,\dots,k$ are independent under $\PR_{\param^{*}}$. 
		\end{description}
        
	\end{claim}

Next, we deduce a useful corollary of Theorem \ref{thm:Robust}. It states that the unsigned magnetization $\sigma_{x}Z_{x}$ behaves similarly as $\sigma_{u}Z_{x}$, where $u$ is the parent of $x$. This result will be used to prove Lemma \ref{lem:k123} in Section \ref{sec:Hessian_off_diag_pf}. 

 \begin{corollary}\label{cor:Robust} 
Suppose the two children of $u$ are $x$ and $y$ with respective descendant subtrees $T_x$ and $T_y$. Under the hypothesis of Theorem \ref{thm:Robust}, there exists a constant $C_{\ref{eqn:descendent}}>0$ such that 
    \begin{equation}\label{eqn:descendent}
               \PR_{\param^{*}}\left(\sigma_u Z_x^{\hparam}(\sigma_{L_x}) < 0\right) \le C_{\ref{eqn:descendent}}\delta \quad\textup{and}\quad \PR_{\param^{*}}\left(\sigma_u Z_x^{\hparam}(\sigma_{L_x})<0, \sigma_u Z_y^{\hparam}(\sigma_{L_y})< 0\right) \le C_{\ref{eqn:descendent}}\delta^2.
            \end{equation} 
\end{corollary} 

\begin{proof}
For the first bound, observe that 
    \begin{align*}
           \PR_{\param^{*}}\left(\sigma_u Z_x^{\hparam}(\sigma_{L_x}) < 0\right)  &=    \PR_{\param^{*}}\left(\sigma_x Z_x^{\hparam}(\sigma_{L_x}) < 0,\, \sigma_{x}=\sigma_{u}\right)  +    \PR_{\param^{*}}\left(-\sigma_x Z_x^{\hparam}(\sigma_{L_x}) < 0,\, \sigma_{x}\ne \sigma_{u}\right) \\
           &\le    \PR_{\param^{*}}\left(\sigma_x Z_x^{\hparam}(\sigma_{L_x}) < 0\right)  +    \PR_{\param^{*}}\left( \sigma_{x}\ne \sigma_{u}\right). 
    \end{align*}
    The first probability in the last expression is at most $c_{\ref{eqn:Reconstruct}}\delta$ by Theorem \ref{thm:Robust}\textbf{(i)} and the second one is $p_{ux}\le C_{\ref{eqn:pBounds}}\delta$.

    For the second bound, we use a similar argument to obtain 
     \begin{align*}
         & \PR_{\param^{*}}\left(\sigma_u Z_x^{\hparam}(\sigma_{L_x})<0, \sigma_u Z_y^{\hparam}(\sigma_{L_y})< 0\right) \\
         &\qquad \le \PR_{\param^{*}}\left(\sigma_x Z_x^{\hparam}(\sigma_{L_x})<0, \sigma_u Z_y^{\hparam}(\sigma_{L_y})< 0\right) +  \PR_{\param^{*}}\left( \sigma_{x}\ne \sigma_{u}, \sigma_u Z_y^{\hparam}(\sigma_{L_y})< 0 \right)\\
         &\qquad \le \PR_{\param^{*}}\left(\sigma_x Z_x^{\hparam}(\sigma_{L_x})<0, \sigma_y Z_y^{\hparam}(\sigma_{L_y})< 0\right) +\PR_{\param^{*}}\left(\sigma_x Z_x^{\hparam}(\sigma_{L_x})<0, \sigma_{x}\ne \sigma_{u}\right)\\
         & \qquad \qquad + \PR_{\param^{*}}\left( \sigma_{x}\ne \sigma_{u}, \sigma_y Z_y^{\hparam}(\sigma_{L_y})< 0 \right) + \PR_{\param^{*}}\left( \sigma_{x}\ne \sigma_{u}, \sigma_{y}\ne \sigma_{u} \right).
    \end{align*}
    The four terms in the last expression are all of order $O(\delta^{2})$ from the first bound in \eqref{eqn:descendent} and the independence between unsigned magnetization in Claim \ref{claim:unsigned_mag}. 
\end{proof}

\section{Proof of Theorem \ref{thm:DiagDominant}}
\label{sec:Hessian_off_diag_pf}

The rest of the main text is devoted to proving Theorem \ref{thm:DiagDominant}. %Its Part \textbf{(i)} stating that the population Hessian has large diagonal entries can be derived without much further preparation. 

%A major technical difficulty in this work is analyzing the off-diagonal entries of the Hessian, which involves a product of dependent random variables as in \eqref{eqn:hessianTerms} correlated by the ``magnetization recursion.'' Furthermore, the denominators in \eqref{eqn:hessianTerms} can be very small. Namely, each of them is (at worst) $\Omega(\delta^2)$, and as there can be as many as $\operatorname{diam}(T)$, there is the possibility that the off-diagonal terms are exponentially large. Roughly speaking, we handle this difficulty by grouping consecutive terms in \eqref{eqn:hessianTerms} and handling the expectations of such groupings.

\subsection{Bounding the diagonal entries of the Hessian}
\label{sec:diagonal_Hessian_pf}

Using Theorem \ref{thm:Robust}, we prove the estimate for the diagonal entries of the Hessian given in \eqref{eqn:HessianDiag} in Theorem  \ref{thm:DiagDominant}.

\begin{proof}[\textbf{Proof of \eqref{eqn:HessianDiag} in Theorem  \ref{thm:DiagDominant}}]
		Let us fix an edge $e = \{x,y\}$. Note that, by Lemma \ref{lem:derivative},
	\begin{equation*}
		\frac{\partial^2}{\partial \htheta_e^2} \ell(\hparam) = \frac{\partial}{\partial \htheta_e}\frac{Z_xZ_y}{1+Z_xZ_y\htheta_e} = -\frac{(Z_xZ_y)^2}{(1+Z_xZ_y \htheta_e)^2}. 
	\end{equation*} 
    Heuristically, the magnetizations $Z_{x}$ and $Z_{y}$ are likely to be close to the spins $\sigma_{x}$ and $\sigma_{y}$ respectively. Hence we should have the following approximation 
    \begin{align}\label{eq:hessian_diag_approx}
    \frac{(Z_xZ_y)^2}{(1+Z_xZ_y \htheta_e)^2} \approx \frac{(\sigma_x\sigma_y)^2}{(1+\sigma_x\sigma_y \htheta_e)^2}.    
    \end{align}
	The expectation of the right-hand side above can be easily computed. Indeed, considering whether there is a flip or not on the edge $e=\{x,y\}$ and recalling \eqref{eqn:peDef}, 
	\begin{align*}
		\E_{\param^{*}}\left[ \frac{(\sigma_x\sigma_y)^2}{(1+\sigma_x\sigma_y \htheta_e)^2}\right] 
        &= \underbrace{\frac{1}{(1+\htheta_e)^2} \frac{1+\theta_{e}}{2}}_{\Theta(1)}  + \underbrace{\frac{1}{(1-\htheta_e)^2}\frac{1-\theta_{e}}{2}}_{\Theta(\delta^{-1})} = \Theta(\delta^{-1}). 
	\end{align*} 
    
    In order to rigorously justify the above heuristic calculation, we need to control the error of the approximation in \eqref{eq:hessian_diag_approx}. 
        To do so, we use the trichotomy of magnetizations in \eqref{eq:recons_tiers} due to the reconstruction theorem (Thm. \ref{thm:Robust}) and  the independence properties of unsigned magnetizations in Claim \ref{claim:unsigned_mag}. 
		Define events 
	\begin{align}\label{eq:events_grad_thm_pf}
			\begin{cases}
				\quad 	F&:= \{ \textup{flip on the edge $e=\{x,y\}$} \}= \{\sigma_x\neq \sigma_y\} \\
				\quad	A&:= \{ \textup{good reconstruction at both $x$ and $y$}\} \\
				\quad	M&:= \{ \textup{one moderate failure and one good reconstruction among $x$ and $y$} \}\\
				\quad	E&:= A \cup (F^c\cap M). 
			\end{cases}
		\end{align}

        \medskip\noindent\textbf{Upper bound.} Observe that the function $t\mapsto \frac{t^2}{(1+t\htheta_e)^2}$ is maximized at $t = -1$ and attains the value $\frac{1}{(2c_{\ref{eqn:pHatBounds}}\delta)^2} = \Theta(\delta^{-2})$ at $t = -1$.  Therefore, extreme negative values of $\frac{\partial^2}{\partial\htheta_e^2} \ell(\hparam; \sigma)$ should occur when $Z_xZ_y$ is close to $-1$. This is precisely what happens on the event $F\cap A$. 
	In order to show the upper bound in \eqref{eqn:HessianDiag}, note that it suffices to show that 
		\begin{equation}\label{eqn:diag2}
			\E_{\param^{*}} \left[ \frac{(Z_xZ_y)^2}{(1+\htheta_e Z_xZ_y)^2}\mathbf{1}_{[F\cap A]}\right]  = \Omega(\delta^{-1}), 
		\end{equation}
		since then 
		\begin{align*}
			\E_{\param^{*}} \left[ \frac{(Z_xZ_y)^2}{(1+\htheta_e Z_xZ_y)^2}\right] \ge   \E_{\param^{*}} \left[ \frac{(Z_xZ_y)^2}{(1+\htheta_e Z_xZ_y)^2}\mathbf{1}_{[F\cap A]}\right]  = \Omega(\delta^{-1}).
		\end{align*} 
		
		To show \eqref{eqn:diag2}, note that on the event $F\cap A$ we have $Z_xZ_y\le -1+2C_{\ref{eqn:Reconstruct}}\delta^2$ since 
        \begin{align}
            -Z_xZ_y = (\sigma_{x}Z_{x})(\sigma_{y}Z_{y}) \ge (1-C_{\ref{eqn:Reconstruct}} \delta^2 )^{2} = 1 - 2 C_{\ref{eqn:Reconstruct}} \delta^2 \ge 0. 
        \end{align}
        So we get 
		\begin{equation*}
			2c_{\ref{eqn:pHatBounds}}\delta\le  1+\htheta_eZ_xZ_y \le 1+(1-2c_{\ref{eqn:pHatBounds}} \delta)(-1+2 C_{\ref{eqn:Reconstruct}} \delta^2) \le 2c_{\ref{eqn:pHatBounds}}\delta+2 C_{\ref{eqn:Reconstruct}}\delta^2,
		\end{equation*} 
		where the lower bound above holds as $Z_xZ_y\ge -1$ and \ref{assumption1}. That is, on $F\cap A$ the term $1+\htheta_e Z_xZ_y = \Theta(\delta)$. Similarly, $(Z_xZ_y)^2 = (-Z_{x}Z_{y})^{2} \ge  (1 - 2 C_{\ref{eqn:Reconstruct}} \delta^2)^{2} =\Omega(1)$ on the event $F\cap A$ again. Moreover, by the reconstruction theorem (Thm. \ref{thm:Robust}) and the independence property (Claim \ref{claim:unsigned_mag}), $\PR(F\cap A) = \PR(F)\, \PR(A) \ge c_{\ref{eqn:pBounds}}\delta (1-C_{\ref{eqn:Reconstruct}}\delta)^2 = \Omega(\delta)$. It follows that 
		\begin{equation*}
			\E_{\param^{*}}\left[ \frac{(Z_xZ_y)^2}{(1+\htheta_eZ_xZ_y)^2} \mathbf{1}_{[F\cap A]}\right] = \frac{\Omega(1)}{\Theta(\delta^2)}\PR(F\cap A) =  \Omega(\delta^{-1}).
		\end{equation*}

    \medskip\noindent\textbf{Lower bound.} To establish the lower bound in  \eqref{eqn:HessianDiag}, we need a comparable upper bound for $\frac{(Z_xZ_y)^2}{(1+\htheta_e Z_xZ_y)^2}$. We do this by partitioning the sample space into $F^{c}\cap E$ and $F\cup E^c$. Note that under Assumption \ref{assumption1} (in particular \eqref{eqn:pHatBounds}) we have always  $1+\htheta_eZ_xZ_y\ge 2c_{\ref{eqn:pHatBounds}}\delta$ and $(Z_xZ_y)^2\le 1$.
		This gives the upper bound
		\begin{equation}\label{eq:Hessian_diag_pf1}
			\frac{(Z_xZ_y)^2}{(1+\htheta_e Z_xZ_y)^2}\le \frac{1}{(2c_{\ref{eqn:pHatBounds}}\delta)^2} 
		\end{equation} 
		We claim that on the event $F^{c}\cap E=(F^{c}\cap A) \cup (F^{c}\cap M)$, the following stronger upper bounds of $O(1)$ hold: 
		\begin{align}\label{eq:Hessian_diag_pf2}
			\frac{(Z_x Z_y)^2}{(1+\htheta_e Z_xZ_y)^2}  \mathbf{1}_{[F^c\cap A]} &\le 1 \quad  \text{and} \quad  \frac{(Z_xZ_y)^2}{(1+\htheta_e Z_xZ_y)^2} \mathbf{1}_{[F^c\cap M]} \le \frac{1}{(1-c_{\ref{eqn:antiReconstruction}})^2}.
		\end{align} 
		Indeed, the fraction in front of the indicators in the above display will be large only when $Z_{x}Z_{y}$ is close to $-1$. This does not happen on the event $F^c\cap A$, since then $Z_xZ_y = \sigma_x\sigma_y Z_xZ_y\ge 0$. Since magnetizations are between $-1$ and $1$ and since $\hat{\theta}_{e}\ge 0$ by \ref{assumption1}, it follows that $Z_{x}Z_{y}\le 1 \le 1+ \hat{\theta}_{e}Z_{x}Z_{y}$. This yields the first inequality.  
		For the second inequality, note that on the event $F^{c}$ we have $\sigma_{x}\sigma_{y}=1$, and on the event $M$, the smallest possible value for $Z_xZ_y = \sigma_x\sigma_y Z_xZ_y$ is at least $-c_{\ref{eqn:antiReconstruction}}(1-C_{\ref{eqn:Reconstruct}} \delta^2)\ge -c_{\ref{eqn:antiReconstruction}}$ 
        so we also have the trivial upper bound of $1$ for the numerator. This shows the claim. 
		
		Now note that $\PR_{\param^{*}}(E^{c})\le \PR_{\param^{*}}(F)+\PR_{\param^{*}}(M)=O(\delta)$ by \ref{assumption1} and Thm. \ref{thm:Robust}. Hence $\PR_{\param^{*}}(F\cup E^{c})= O(\delta)$.
        Combining this with  \eqref{eq:Hessian_diag_pf1} and \eqref{eq:Hessian_diag_pf2}, we see that 
		\begin{align}
			\E_{\param^{*}}\left[\frac{(Z_xZ_y)^2}{(1+\htheta_e Z_xZ_y)^2}\right] 
			&= \E_{\param^{*}}\left[\frac{(Z_xZ_y)^2}{(1+\htheta_e Z_xZ_y)^2}\mathbf{1}_{F^{c}\cap E}\right]+ \E_{\param^{*}}\left[\frac{(Z_xZ_y)^2}{(1+\htheta_e Z_xZ_y)^2} \mathbf{1}_{F\cup E^c}\right] \nonumber 
			\\
			&=  O(1)  + O(\delta^{-2})\times O(\delta)	
            =O(\delta^{-1})\nonumber
		\end{align}
		This establishes the corresponding lower bound in \eqref{eqn:HessianDiag}.
	\end{proof}

%\subsection{Bounding the off-diagonal entries of the Hessian} 
%\label{sec:Hessian_off_diag_pf}

\subsection{Bounding the off-diagonal entries of the Hessian}
\label{sec:pf_off_diagonal}

In this section, we will prove the bound for the off-diagonal entries of the Hessian stated in Theorem \ref{thm:DiagDominant} assuming three technical lemmas. We will also deduce Corollary \ref{cor:MLE_landscape} from Theorem \ref{thm:DiagDominant}.

We start from the bound \eqref{eq:Hessian_offdiag_bd_bd} in Lemma \ref{lem:off_diagonal_ind_bd}. Throughout this section, we assume that we are using the notation for part of the tree between edges $e$ and $f$ in Section \ref{sec:sketch_proof_offdiagonal} (see also Figure \ref{fig:TREEhess}).

First, we provide the postponed proof of Lemma \ref{lem:independentWj}, which states that the random variables appearing in the bound \eqref{eq:Hessian_offdiag_bd_bd} are independent.

\begin{proof}[\textbf{Proof of Lemma \ref{lem:independentWj}}]
By construction, the random variables $\widetilde{W}_{N},W_{N-4},\dots,W_{r+3},R_{r}$ depend on disjoint sets of random variables $\{\eta_{N+1},\dots,\eta_{N-3}\}$, $\{\eta_{N-4},\dots,\eta_{N-7}\},\dots, \{\eta_{r+3},\dots,\eta_{r}\}$, $\{\eta_{r-1},\dots,\eta_{0}\}$. Thus these random variables are conditionally independent given the spins along the path $\gamma$ consisting of the nodes $y_{N+1},\dots,y_{0},y_{-1}$. The point of this statement is that they are not only conditionally independent, but also actually independent. In order to justify this, we argue similarly as in the proof of Claim \ref{claim:unsigned_mag} (see~\cite{clancy2025likelihood}) by showing that the conditional laws of the random variables $\widetilde{W}_{N},W_{N-4},\dots,W_{r+3},R_{r}$ do not depend on the spins at $y_{N+1},\dots,y_{0},y_{-1}$. 

To spell out the details, first recall that the unsigned magnetizations $\sigma_{w_{0}}Z_{w_{0}},\dots,\sigma_{w_{N}}Z_{w_{N}},\sigma_{x}Z_{x}$  are independent by Claim \ref{claim:unsigned_mag}. 
So the `unsigned signals' $\sigma_{w_{0}}\eta_{0},\dots,\sigma_{w_{N}}\eta_{N},
\sigma_{y_{N+1}}\eta_{N+1}$,  which are scalar multiples of the unsigned magnetizations, are also independent. Now note that for each $i$, 
\begin{equation*}
    W_i | \{\sigma_{y_{i-1}} = 1\} \overset{d}{=}  W_i | \{\sigma_{y_{i-1}} = -1\}. 
\end{equation*}
It follows that $W_{i}$ is determined by the four independent unsigned signals $\sigma_{w_i}\eta_{i},\dots,\sigma_{w_{i-3}}\eta_{i-3}$ that are also independent of all the remaining unsigned signals.
Similarly, the `residual term' $R_{r}$ is determined by the unsigned signals $\sigma_{w_j}\eta_{j}$ for $0\le j\le r-1$ as well as the $\widetilde{W}_N$ term only depends on the unsigned signals $\sigma_{w_N}\eta_{N},\dotsm, \sigma_{w_{N-3}} \eta_{N-3}$ as well as the unsigned magnetization $\sigma_xZ_x$ at $x$. 
We conclude that the random variables $\widetilde{W}_{N},W_{N-4},\dots,W_{r+3},R_{r}$ depend on disjoint sets of independent random variables, which yields the desired independence between them. 
\end{proof}

We now state four lemmas which imply Theorem \ref{thm:DiagDominant} \textbf{(ii)}. Their proof is relegated to the Section \ref{sec:Hessian_off_diag_pf}.

\begin{lemma}[Approximate distribution of $W_i$]\label{lem:4terms}
 There exist constants $\delta_{\ref{eqn:boundForFiniteSample}} >0$ and $C_{\ref{eqn:boundForFiniteSample}}<\infty$ such that the following holds for all binary trees $T$ and $\delta\le \delta_{\ref{eqn:boundForFiniteSample}}$. Suppose Assumption \ref{assumption1} holds. Then for all $e,f\in E(T)$ and for all $i=3,4,\dotsm,N = \dist(e,f)$ 
	\begin{equation}\label{eqn:boundForFiniteSample}
		W_i 
  \in \begin{cases}[0,C_{\ref{eqn:boundForFiniteSample}}\delta^2]&:\textup{ with probability at least }1-C_{\ref{eqn:boundForFiniteSample}}\delta\\
			(C_{\ref{eqn:boundForFiniteSample}}\delta^2,C_{\ref{eqn:boundForFiniteSample}}]&:\textup{ with probability at most }C_{\ref{eqn:boundForFiniteSample}}\delta\\
			(C_{\ref{eqn:boundForFiniteSample}},C_{\ref{eqn:boundForFiniteSample}}\delta^{-1}]&:\textup{ with probability at most }C_{\ref{eqn:boundForFiniteSample}}\delta^2\\
			(C_{\ref{eqn:boundForFiniteSample}}\delta^{-1},C_{\ref{eqn:boundForFiniteSample}}\delta^{-2}]&:\textup{ with probability at most }C_{\ref{eqn:boundForFiniteSample}}\delta^3\\
			(C_{\ref{eqn:boundForFiniteSample}}\delta^{-2},\infty)&:\textup{ with probability }0
		\end{cases}
	\end{equation}
 where $W_i$ are defined as in \eqref{eq:def_W_i}. 
	In particular, 
	\begin{equation}\label{eqn:4terms}
    \max_{3\le i\le N }	\,\E_{\param^{*}} [ W_{i} ]
		\le C_{\ref{eqn:boundForFiniteSample}}^2\delta \quad \textup{and} \quad 	\max_{3\le i\le N} \, 	\E_{\param^{*}}\left[ W_{i}^2\right] \le \frac{C^{3}_{\ref{eqn:boundForFiniteSample}}}{\delta}.
	\end{equation}
\end{lemma}

Lemma \ref{lem:4terms} states that the random variables $W_{3},W_{4},\cdots,W_N$ defined in \eqref{eq:def_W_i} have roughly the same distribution given by \eqref{eqn:boundForFiniteSample}. This easily implies the uniform bound on their first two moments in \eqref{eqn:4terms}. The first result on almost identical distribution at a first glance seems surprising since they could depend on very different subtrees. For instance, suppose the edge $e$ is close to the root of the binary tree $T$ and the other edge $f$ is near the leaves. Then traversing the path from $f$ to $e$, one encounters larger and larger subtrees $T_{w_{j}}$ injecting signals $\eta_{j}$ into the internal node $y_{j}$ (see Fig. \ref{fig:TREEhess}\textbf{b}). However, by the robust ancestral reconstruction (Thm. \ref{thm:Robust}), the signals $\eta_{j}$ roughly have the same distribution regardless of the supporting subtrees $T_{w_{j}}$. This, with the fact that the unsigned signals $\sigma_{w_{i}}\eta_{i}$ are independent, each $W_{k}$ is determined by four signals  $\eta_{i}, \eta_{i-1},\eta_{i-2}, \eta_{i-3}$ with approximately the same joint distribution. So it is not surprising that $W_{i}$s also have approximately the same distribution, as stated in Lemma \ref{lem:4terms}.

The next lemma bounds the first two moments of the term $\widetilde{W}_{N}$ in \eqref{eq:def_widetildeW_i}. Since it depends on an additional term from the magnetization $Z_{x}=\eta_{N+1}$, we have a slightly worse bound on its first two moments than those for $W_{i}$s in Lemma \ref{lem:4terms}.

\begin{lemma}[Moment bounds on $\widetilde{W}_N$ and $R_r$] \label{lem:5terms} 
	There exist constants $C_{\ref{eqn:4termsDiag}},C_{\ref{eqn:5termVar}}<\infty$ and $\delta_{\ref{eqn:4termsDiag}}>0$ depending only on $c_{\ref{eqn:pBounds}}, C_{\ref{eqn:pBounds}}, c_{\ref{eqn:pHatBounds}}, C_{\ref{eqn:pHatBounds}}$ such that the following holds for all binary trees $T$ and any $\delta\le \delta_{\ref{eqn:4termsDiag}}$. Suppose that Assumption \ref{assumption1} holds and let $e,f\in E(T)$ be any two edges such that $N = \dist(e,f)\ge 3$ and $r\in\{0,1,2,3\}$. Then for $\widetilde{W}_N$ and $R_r$ defined in \eqref{eq:def_widetildeW_i} and \eqref{eq:def_R_N} (resp.) 
	\begin{equation} \label{eqn:4termsDiag}
		\E_{\param^{*}}\left[ \widetilde{W}_{N} \right] \le 
		C_{\ref{eqn:4termsDiag}}\qquad\textup{and}\qquad \E_{\param^{*}}[R_r]\le C_{\ref{eqn:4termsDiag}}.
	\end{equation}
	Furthermore, 
	\begin{equation}\label{eqn:5termVar}
		\E_{\param^{*}}\left[ \widetilde{W}_N^{2} \right]\le \frac{C_{\ref{eqn:5termVar}}}{\delta^4}\qquad\textup{and}\qquad \E_{\param^{*}}\left[ R_r^{2} \right]\le \frac{C_{\ref{eqn:5termVar}}}{\delta^3}.
	\end{equation}
\end{lemma}

As we will see below, Lemma \ref{lem:4terms} and Lemma \ref{lem:5terms} give us good control on the first two moments of $\left| \frac{\partial^2}{\partial \htheta_e\partial\htheta_f} \ell(\hparam; \sigma)\right|$ whenever $N=\operatorname{dist}(e,f)\ge 4$. 
To handle the near-diagonal terms $(N=0,1,2)$ we include the following lemma.
\begin{lemma}[Bounds on the Hessian near diagonal]\label{lem:k123}\label{lem:nearDiagonal}
  There exist constants
    $C_{\ref{eqn:nearDiag}}, C_{\ref{eqn:nearDiag2}}>0$ and $\delta_{\ref{eqn:nearDiag}}>0$ depending only on $c_{\ref{eqn:pBounds}},C_{\ref{eqn:pBounds}}, c_{\ref{eqn:pHatBounds}}, C_{\ref{eqn:pHatBounds}}$ such that for any binary tree $T$ and any $\delta\le \delta_{\ref{eqn:nearDiag}}$ the following holds. Suppose Assumption \ref{assumption1} holds. Then for any distinct $e,f$ with $N=\dist(e,f)\in\{0,1,2\}$ it holds that
	\begin{equation}\label{eqn:nearDiag}
		\E_{\param^{*}}\left[ \left|\frac{\partial^2}{\partial \htheta_e\partial\htheta_f} \ell(\hparam; \sigma|_{L})\right|\right]
		\le C_{\ref{eqn:nearDiag}} 
	\end{equation}
 and 
 \begin{equation}\label{eqn:nearDiag2}
     \E_{\param^{*}}\left[ \left|\frac{\partial^2}{\partial \htheta_e\partial\htheta_f} \ell(\hparam; \sigma|_{L})\right|^2\right]
		\le \begin{cases}
		    C_{\ref{eqn:nearDiag2}} \delta^{-2} &: N = 0,1\\
              C_{\ref{eqn:nearDiag2}} \delta^{-3} &: N = 2
		\end{cases}.
 \end{equation}
\end{lemma}

We now show how these lemmas imply the moment bounds \eqref{eqn:HessiangOffDiag} on the off-diagonal entries in the Hessian. 
\begin{proof}[\textbf{Proof of \eqref{eqn:HessiangOffDiag} in Theorem \ref{thm:DiagDominant} assuming Lemmas \ref{lem:4terms}--\ref{lem:nearDiagonal}}]
	Fix two distinct edges $e$ and $f$ and write $\text{dist}(e,f)=N\ge 0$. (Observe that $\text{dist}$ is not a metric as we say that $e = \{x,y\}, f = \{y,z\}$ have $\textup{dist}(e,f) = 0$.) Recall the decomposition  \eqref{eq:Hessian_offdiag_bd_bd}
	\begin{align}\nonumber 
		\bigg| \frac{\partial^2}{\partial\htheta_e\partial\htheta_f} \ell(\hparam;\sigma|_{L})\bigg| 
		\le   \widetilde{W}_N \underset{\lfloor{\frac{N-3}{4}} \rfloor \textup{ many ``}W_i\textup{'' terms}}{\underbrace{W_{N-4}W_{N-8}\dotsm W_{r+3}}}R_r\qquad \textup{ where } r := (N+1)\mod 4 
	\end{align}
	where the random variables in the right-hand side are independent (by Lemma \ref{lem:independentWj}). Hence, taking expectation under the population model and using Lemmas \ref{lem:4terms}, \ref{lem:5terms}, and Lemma \ref{lem:nearDiagonal} we get 
	\begin{align} 
		\nonumber\E_{\param^{*}} \left[ 	\bigg| \frac{\partial^2}{\partial\htheta_e\partial\htheta_f} \ell(\hparam;\sigma|_{L})\bigg|  \right] \le \begin{cases}
		    	C^2_{\ref{eqn:4termsDiag}} (C_{\ref{eqn:4terms}}\delta)^{\lfloor (N-3)/4 \rfloor} &: N\ge 3\\
                C_{\ref{eqn:nearDiag}} &: N = 0,1,2
		\end{cases}.  
	\end{align}
	This shows \eqref{eqn:HessiangOffDiag}. 
\end{proof}

\begin{remark}[Bounds on the second moment of the entries in the Hessian]
    From the estimates we have, it is easy to obtain the following bounds on the second moment of the entries of the Hessian:
	\begin{align} \label{eqn:HessiangOffDiag_Var}
			 \E_{\param^{*}} \left[ 	\bigg| \frac{\partial^2}{\partial\htheta_e\partial\htheta_f} \ell(\hparam;\sigma|_{L})\bigg|^{2}  \right]
				& \le\begin{cases}      
                      C_{\ref{eqn:HessiangOffDiag_Var}}\delta^{-2} &:\textup{dist}(e,f) = 0,1\\          
                      C_{\ref{eqn:HessiangOffDiag_Var}}\delta^{-3} &:\textup{dist}(e,f) = 2\\          
                      \left( C_{\ref{eqn:HessiangOffDiag_Var}}/\delta \right)^{\lfloor \frac{\textup{dist}(e,f)-3}{4}\rfloor }\delta^{-7} &: \textup{dist}(e,f)\ge 3\\  
				\end{cases}
	   \end{align}
Indeed, for $N\ge 3$, by Lemmas \ref{lem:4terms} and \ref{lem:5terms},
	\begin{align} 
  \E_{\param^{*}} \left[ 	\bigg| \frac{\partial^2}{\partial\htheta_e\partial\htheta_f} \ell(\hparam;\sigma|_{L})\bigg|^{2}  \right]
& \le \E_{\param^{*}}\left[ \widetilde{W}_N^2 {W_{N-4}^2W_{N-8}^2\dotsm W_{r+3}^2}R_r^2\right] 
  \le \left(\frac{C_{\ref{eqn:4terms}}^3}{\delta}\right)^{\lfloor \frac{N-3}{4}\rfloor} \frac{C_{\ref{eqn:5termVar}}^2}{\delta^7}.\nonumber
	\end{align}
	This shows \eqref{eqn:HessiangOffDiag_Var} for $N\ge 3$. Equation \eqref{eqn:nearDiag2} in Lemma \ref{lem:nearDiagonal} gives precisely the bounds \eqref{eqn:HessiangOffDiag_Var} for $N = 0,1,2$.  
\end{remark}

Next, we deduce Corollary \ref{cor:MLE_landscape}  from  Theorem \ref{thm:DiagDominant} and Gershgorin's circle theorem.

\begin{proof}[\textbf{Proof of Corollary} \ref{cor:MLE_landscape}]
	Let $\mathbf{H}=\mathbf{H}(\hparam)$ denote the expected Hessian at a parameter $\hparam$ satisfying Assumption \ref{assumption1} (see \eqref{eq:def_pop_hessian}). Its size is $|E|\times |E|$, where $|E|$ is the number of edges in the tree $T$. By Theorem \ref{thm:DiagDominant}, 
	we have that for edge $e$, 
	\begin{align}
 \nonumber 
		-\frac{\widetilde{C}_{\ref{eqn:HessianDiag}}}{\delta}\le    \mathbf{H}_{e,e} \le -\frac{C_{\ref{eqn:HessianDiag}}}{\delta}. 
	\end{align}
	Furthermore, since 
	there are at most $2^{k+1}$ many edges at distance $k$ from an edge $e_{i}$, \eqref{eqn:HessiangOffDiag} in Theorem \ref{thm:DiagDominant} yields that, whenever  $\delta$ is small enough so that 
	$16 C_{\ref{eqn:HessiangOffDiag}}\delta < 1/2$, for each edge $e$ fixed, 
	\begin{align}
 \nonumber 
		\sum_{f\neq e}|\mathbf{H}_{e,f}|\le \sum_{k=0}^{\infty} 2^{k+1}(C_{\ref{eqn:HessiangOffDiag}}\delta)^{ \lfloor \frac{(k-1)\lor 0}{4}\rfloor  } \le 2 + 4 \sum_{k=0}^{\infty} (16  C_{\ref{eqn:HessiangOffDiag}}   \delta)^{k/4}  \le 2 + \frac{4}{1-2^{-1/4}}<26.
	\end{align}
	Thus we deduce 
	\begin{align}\label{eq:expected_hessian_eval_bd}
		-\frac{\widetilde{C}_{\ref{eqn:HessianDiag}}}{\delta} -26 \le 	\mathbf{H}_{e,e} - \sum_{f\ne e} |\mathbf{H}_{e,f}|  \le 	-\frac{C_{\ref{eqn:HessianDiag}}}{\delta} + 26 \quad \textup{for all edge $e$ in $T$}.  
	\end{align}
	Recall that by Gershgorin's circle theorem, the eigenvalues of a square matrix are contained in the union of all disks centered at the diagonal entries with radius the absolute sum of the off-diagonal entries in the corresponding rows. Since the Hessian is real  symmetric, it has real eigenvalues. Then \eqref{eq:expected_hessian_eval_bd} yields the desired bounds on the eigenvalues of $\mathbf{H}$ in \eqref{eq:expected_Hessian_eval_range}. 
\end{proof}

\section{Proof of Key Technical Lemmas}

In this section, we will prove Lemmas \ref{lem:4terms}, \ref{lem:5terms}, and \ref{lem:nearDiagonal}, which were used to prove Theorem \ref{thm:DiagDominant}  in Section \ref{sec:Hessian_off_diag_pf}.

    For notational simplicity, we will prove Lemmas \ref{lem:4terms} and \ref{lem:5terms} for the case of $W_{4}$ and $\widetilde{W}_{4}$,  respectively. By shifting the index, we can conclude Lemmas \ref{lem:4terms} and \ref{lem:5terms} for the general case. We will also write $\xi_j$ instead of $\xi_j^\circ$
    so that the relevant variables are $\xi_1,\dotsm,\xi_4$ and $\eta_1,\dotsm,\eta_4$ (see Figure \ref{fig:Tree4Terms1}). More precisely, these variables are 
    \begin{align}\label{eq:W4_W4_tilde}
     W_{4}&=\sup_{|\xi_1|\le 1-2c_{\ref{eqn:pHatBounds}}\delta}  \prod_{j=1}^4 \frac{(1-\eta_j^2)}{(1+\xi_j \eta_j)^2}, \\
    \nonumber  \widetilde{W}_4 &= \sup_{|\xi_1|\le 1-2c_{\ref{eqn:pHatBounds}}\delta} \frac{1}{(1+\xi_5\eta_5)^2} \prod_{j=1}^4 \frac{(1-\eta_j^2)}{(1+\xi_j\eta_j)^2}.
    \end{align}

\begin{figure}[b!]
	\centering
	\includegraphics[width=5in]{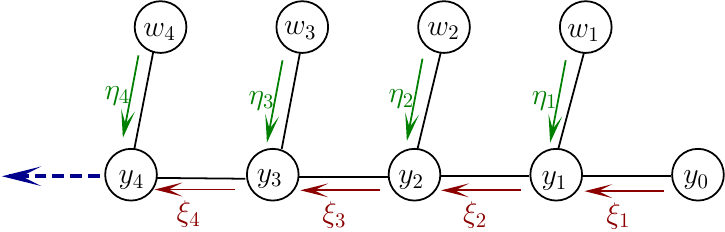}
	\caption{ 
    Signals used to define $W_{4}$ in \eqref{eq:def_W_i}. Note that $W_{4}$ does not depend on $\xi_{1}$ due to the adversarialization. 
    }
	\label{fig:Tree4Terms1}
    \end{figure}

\subsection{Observations about $q$}

The next three claims were established in \cite{clancy2025likelihood} to prove Thm. \ref{thm:Robust}. 
	\begin{claim}[Magnetization recursion: two strong signals]\label{claim:quickBound1}
		Fix any $\eps\in [0,\frac{1}{2})$. If $s,t$ are such that $1-\eps\le s,t\le 1$ then
		\begin{equation*}
			q(s,t) \ge 1-\frac{4}{5}\eps^2\qquad \textup{ and }\qquad q(-s,-t)\le -1+\frac{4}{5}\eps^2.
		\end{equation*}
	\end{claim}

	\begin{claim}[Magnetization recursion: corruption at distance $3$]\label{claim:quickBound2}
		Suppose that $0<a<A$, $0<B$ and $\delta\in [0,1]$ are such that $a<A/2$, $\delta<a/2$, 
        and  $\left(\frac{2 A^2}{a} + B\right)\delta<1/2$.
        If $s_1 \in [-1+a\delta,1]$ and, 
		for each $j\in\{2,3,4\}$ and $i\in\{1,2\}$,
		\begin{equation*}
			s_j\in [1-A\delta,1] \qquad\text{and}\qquad t_i\in[1-B\delta,1].
		\end{equation*}
		        Then we have  \begin{align}
  t_2 q(t_1q(s_1,s_2), s_3)
		\nonumber &\geq 1 - \left(\frac{2 A^2}{a} + B\right)\delta,\\
			q\Bigg( t_2q\Big(t_1q(s_1,s_2)  , s_3\Big), s_4  \Bigg) &\ge 1-\frac{4}{5}\left(\frac{2 A^2}{a} + B\right)\delta^2
			\label{eqn:poly1}.
		\end{align}
	\end{claim}

 \begin{claim}\label{claim:strongopposites.}
     Fix any $0<a<A$ and let $\delta>0$ with $A\delta<1$. For any $s,t \in [1-A\delta, 1-a\delta]$ then
     \begin{equation*}
        -1+\frac{a}{A}\le q(s,-t) = q(-s,t) \le 1-\frac{a}{A}
     \end{equation*}
     More generally, if $t\in [1-A\delta, 1-a\delta]$ and $|s|\le 1-a\delta$ then
     \begin{equation*}
         q(t,s) \ge -1 + \frac{a}{A}.
     \end{equation*}
 \end{claim}

 The following claims will be used several times in the sequel. 
 \begin{claim} \label{claim:qRecurseClaim} Fix a constant $\kappa\in(0,1)$. Suppose that there are $x_1,x_2,\dotsm, x_m\in[-\kappa,1]$ and $t_1,\dotsm, t_m\in[0,1]$. Then there exists another constant $\kappa'\in(0,1)$ depending only on $m$ and $\kappa$ such that $ y_1:= x_1 \ge -\kappa'$ and for all $j = 2,\dotsm, m$
     \begin{equation*}
         y_j := q(x_j, t_j y_{j-1}) \ge -\kappa'.
     \end{equation*}
 \end{claim}

 \begin{proof}
   Note $q(\cdot,\cdot)$ is non-decreasing in both arguments, so 
    \begin{align*}
    \inf\{q(x,ty): x,y\in[ -\kappa,1], t\in[0,1]\} \geq \inf\{q(x,y): x,y\in[-\kappa,1] \}  = \frac{-2\kappa}{1+\kappa^2}\in(-1,0).
    \end{align*}
    The general $m$ case holds by induction.
 \end{proof}

 \begin{claim}\label{claim:infxy}

 		Fix any $\kappa \in(0,1)$. Then
\begin{equation*}\label{eqn:infxy}
\inf_{x,y\in [-\kappa,1]} 1+xy = 1-\kappa.
		\end{equation*}
\end{claim}
\begin{proof}
   If $x\in [0,1]$, then $\inf_{y\in[-\kappa,1]} 1+xy = 1-\kappa x\ge 1-\kappa$. By symmetry, the same holds if $y\in[0,1]$. Finally, if $x,y\in[-\kappa,0)$, then $1+xy\ge1$. 
\end{proof}

 \subsection{General preliminary bounds} 
	
	We first establish some elementary facts that will be useful in our subsequent analysis. 

It is natural to analyze some of the algebraic properties of the recursion $q$ in \eqref{eq:q_def_recursion}. This next lemma underpins some subsequent bounds.
    \begin{lemma}[Reversing the recursion] \label{lem:swap} Let $\xi_1,\eta_1\in(-1,1)$ and $\eta_2\in[-1,1]$ and $\htheta_1\in(0,1)$. Let $\xi_2:= \htheta_1 q(\eta_1,\xi_1) = \htheta_1 \frac{\eta_1+\xi_1}{1+\eta_1\xi_1}$. Then the following hold:
    \begin{description}
        \item[\textbf{(i)}] If $\tilde{\eta}_2 = \eta_2$ and $\tilde{\eta}_1 = q(\htheta_1\tilde{\eta}_2, \eta_1)$ then 
        \begin{equation*}
            (1+\xi_2\eta_2)(1+\xi_1\eta_1) = (1+\theta_1 \eta_1\tilde{\eta}_2)(1+\xi_1\widetilde{{\eta}}_1).
        \end{equation*}
        \item[\textbf{(ii)}] If $\tilde{\eta}_2 = \eta_2$ and $\tilde{\eta}_1 = q(\htheta_1\tilde{\eta}_2, \eta_1)$ then
        \begin{equation*}
            {(1+\htheta_1 \eta_1\widetilde{{\eta}}_2)^2 (1-\tilde{\eta}_1^2)} = (1-{\eta}_1^2)({1-(\htheta_1\tilde{\eta}_2)^2})
        \end{equation*}
        and, by rearranging, 
        \begin{equation*}
            \frac{1-{\eta}_1^2} {(1+\htheta_1 \eta_1\widetilde{{\eta}}_2)^2 (1-\tilde{\eta}_1^2)} = \frac{1}{1-(\htheta_1\tilde{\eta}_2)^2}
        \end{equation*}
    \end{description}
    \end{lemma}
    \begin{proof}
        For \textbf{(i)}, we have
        \begin{align*}
             (1+\xi_2\eta_2)(1+\xi_1\eta_1) &= \left((1+\xi_1\eta_1) + \htheta_1(\xi_1+\eta_1)\eta_2\right)\\
&= \left(1+\htheta_1\eta_1\eta_2 + \xi_1(\eta_1+\htheta_1\eta_2)\right)
             = (1+\theta_1 \eta_1\tilde{\eta}_2)(1+\xi_1\widetilde{{\eta}}_1).
        \end{align*}

        For \textbf{(ii)} it is similar. To begin note that 
        \begin{equation*}
            (1+ab)^2 -(a+b)^2 = (1-a^2)(1-b^2).
        \end{equation*}Since $\tilde{\eta}_1 = \frac{\htheta_1\tilde{\eta}_2 + \eta_1}{1+\htheta_1\eta_1\tilde{\eta}_2}$ and so 
        \begin{align*}
            (1+\htheta_1\eta_1\tilde{\eta}_2)^2 (1-\tilde{\eta}_1^2) &=  (1+\htheta_1\eta_1\tilde{\eta}_2)^2 - (\htheta_1\tilde{\eta}_2 + \eta_1)^2\\
            &= \left(1-(\htheta_1\tilde{\eta}_2)^2\right)\left(1-\eta_1^2\right).
        \end{align*}
    \end{proof}

Let us recall that $\eta_j = \htheta_j Z_{w_j}$. Since $\htheta_j=1-O(\delta)$ by Assumption \ref{assumption1}, we can replace the magnetization $Z_{w_{j}}$ in the reconstruction theorem (Thm. \ref{thm:Robust}) by the signal $\eta_{j}$, as in the following corollary.
\begin{corollary}\label{cor:eta} 
    There exists a constant $C_{\ref{eqn:etareconstruct}}\ge 4C_{\ref{eqn:pHatBounds}}>0$ such that for all binary trees $T$ and $\delta \le \delta_{\ref{eqn:Reconstruct}}$ the following hold. Suppose Assumption \ref{assumption1} holds, and let $\eta_j$ be as in \eqref{eq:def_eta_xi}. Then a.s. $|\eta_j|\le 1-2c_{\ref{eqn:pHatBounds}}\delta$ for $j\neq N+1$ and for all $j = 0,1,2,\dotsm, N+1$
    \begin{align}
       \label{eqn:etareconstruct} &\PR_{\param^{*}}\left(\sigma_{w_j} \eta_j \ge 1- C_{\ref{eqn:etareconstruct}}\delta \right) \ge 1-c_{\ref{eqn:Reconstruct}} \delta
    & &\PR_{\param^{*}}\left(\sigma_{w_j} \eta_j <-c_{\ref{eqn:antiReconstruction}}\right) \le c_{\ref{eqn:antiReconstruction}} \delta^2.
    \end{align} Moreover, the bounds in \eqref{eqn:etareconstruct} hold for $\sigma_v\xi_0$ replacing $\sigma_{w_j}\eta_j$. 
\end{corollary}

 To obtain bounds for products  
 of two consecutive terms in \eqref{eq:Hessian_offdiag_bd}, 
 let us consider the function 
	\begin{equation*}
		F(\htheta_1,\eta_1,\eta_2,\xi_1):=\frac{(1-\eta_1^2)(1-\eta_2^2)}{(1+\xi_2\eta_2)^2(1+\xi_1\eta_1)^2},\qquad \textup{where}\quad \xi_2 = \htheta_1 \frac{\eta_1+\xi_1}{1+\eta_1\xi_1}.
	\end{equation*}
	Recall that $Z_a\in[-1,1]$ a.s. so under Assumption \ref{assumption1} (in particular \eqref{eqn:pHatBounds}), our new random variables $\xi_{j},\eta_{j}$ satisfy \eqref{eqn:xiandetabound}
a.s. and $\htheta_i \in [1-2C_{\ref{eqn:pHatBounds}}\delta,1-2c_{\ref{eqn:pHatBounds}}\delta]$.
	The term $F$ appears when grouping two consecutive terms in the product in \eqref{eqn:hessianTerms}. The following proposition tells us that the worst-case scenario for a product of two terms is $O(\delta^{-1})$ and will be (using Thm. \ref{thm:Robust}) $O(1)$ with probability $1-O(\delta)$. We state the result slightly more generically for some future applications in the proof of Lemma \ref{lem:nearDiagonal} (see \eqref{eq:hessian_off_bd_N1_pf1} in particular).
	\begin{prop}\label{prop:twoTerms}
		Let $A>a>0$ and $B>a$. Suppose that $\htheta_1\in [1-A\delta,1-a\delta]$ and suppose that $B\delta, A\delta<1$.
        Then the following hold:
		\begin{enumerate}
			\item (\textit{Generic signals}) 			\begin{equation*}\label{eqn:2Terms}
				\sup_{|\xi_1|\le 1-a\delta}  F(\htheta_1,\eta_1,\eta_2,\xi_1) \le \frac{16\vee (8/a)}{\delta} \qquad \textup{for all }\eta_{1},\eta_{2} \textup{ s.t. } |\eta_{j}|\le 1-a\delta;
			\end{equation*}
			\item (\textit{Strong signals}) 
			\begin{equation*}\label{eqn:strongEtas}
			\sup_{|\xi_1|\le 1-a\delta} 	F(\htheta_1,\eta_1,\eta_2,\xi_1) \le \frac{4B^4}{a^2(B-a)^2}\qquad \textup{for all }\eta_1,\eta_2\textup{ s.t. }|\eta_j|\in [1-B\delta,1-a\delta].
			\end{equation*}
		\end{enumerate}
	\end{prop}
	\begin{proof}
        Fix any $\xi_1\in[-1+a\delta,1-a\delta]$. 
		First rewrite $F = F(\htheta_1,\eta_1,\eta_2,\xi_1)$ as 
		\begin{equation}
			F = \frac{(1-\eta_1^2)(1-\eta_2^2)}{(1 +\eta_1\xi_1 + \eta_2\eta_1\htheta_1 + \eta_2\htheta_1\xi_1)^2} = \frac{(1-\eta_1^2)(1-\eta_2^2)}{\displaystyle (1+\htheta_1\eta_1\eta_2)^2\left(1+\xi_1 \frac{\htheta_1\eta_2 +\eta_1}{1+\htheta_1\eta_1\eta_2} \right)^2}.\label{eqn:productDenom}
		\end{equation}
		Observe the denominator is $(1+\htheta_1\eta_1\eta_2)^2(1+\xi_1 q(\htheta_1\eta_2,\eta_1))^2$. 
		Note that if $x\in [-1+a\delta , 1-a\delta]$ and $y\in[-1,1]$ then $(1+xy)^2 \ge (1-(1-a\delta) |y|)^2$ and the equality can always be achieved. We also always have $1-\eta_2^2 \le 1-(\htheta_1 \eta_2)^2$ since $\htheta_1\in(0,1)$ and $\eta_2\in [-1,1]$. It then follows that
		\begin{equation*}
			F \le \frac{(1-\eta_1^2)(1-(\htheta_1\eta_2)^2)}{\displaystyle(1+\htheta_1\eta_1\eta_2)^2\left(1-(1-a\delta) \frac{|\htheta_1\eta_2+\eta_1|}{|1+\htheta_1\eta_1\eta_2|}\right)^2 }.
		\end{equation*}
		Let us define the function 
		\begin{equation*}
			G(x,y) = G(-x,-y) = G(y,x) =  \frac{(1-x^2)(1-y^2)}{\displaystyle(1+xy)^2\left(1-(1-a\delta)\frac{|x+y|}{|1+xy|}\right)^2}
		\end{equation*} To show part (1), it suffices to show
		\begin{equation}\label{eq:G_x_y_claim}
			\sup_{|x|,|y|\le 1-a\delta} G(x,y) \le \frac{16 \vee (8/a)}{\delta}.
		\end{equation}
		Let us first consider $x,y\in [0,1-a\delta]$. 
		This implies both $(1+x),(1+y)\le 2$ and $(1-x),(1-y)\ge 0$, and so $(1-x^2)(1-y^2)\le 2^2 (1-x)(1-y)$. It follows 
		\begin{align*}
			G(x,y)&\le \frac{(1-x^2)(1-y^2)}{\big(1+xy - x -y + a\delta (x+y)\big)^2}= \frac{(1-x^2)(1-y^2)}{\Big((1-x)(1-y) + a\delta (x+y)\Big)^2}\\
			&\le \frac{4(1-x)(1-y)}{\Big((1-x)(1-y) + a\delta (x+y)\Big)^2}\le \frac{4}{(1-x)(1-y)+ a\delta(x+y)} \le \begin{cases}
				16 &:x,y\le \frac{1}{2}\\
				\frac{8}{a\delta} &: \textup{otherwise.}
			\end{cases}
		\end{align*}
	Indeed, the last bound holds as follows from observing that 
 \begin{equation*}
     (1-x)(1-y)+a\delta(x+y) \ge \begin{cases}
     \frac{1}{2}\frac{1}{2} +0 &: \textup{ both } x,y\le \frac12\\
     0 + a\delta\left(\frac{1}{2}\right) &: \textup{ else}
     \end{cases}.
 \end{equation*}
		A similar analysis when $x,y\le 0$ implies that
		\begin{equation*}
			\sup\left\{ G(x,y): \substack{\displaystyle\textup{sgn}(x)=\textup{sgn}(y),\\ \displaystyle |x|,|y|\le 1-a\delta} \right\}\le \left(16\vee \frac{8}{a}\right) \delta^{-1}.
		\end{equation*}
		Let us now consider the case where $1-a\delta\ge x\ge y>0$ and bound $G(x,-y) = G(-x,y)$:
		\begin{align*}
			G(x,-y) &= \frac{(1-x^2)(1-y^2)}{\displaystyle(1-xy)^2\left(1-(1-a\delta) \frac{x-y}{1-xy} \right)^2}= \frac{(1-x^2)(1-y^2)}{\left(1-xy - x+y + a\delta (x-y)\right)^2}\\
			&= \frac{(1-x^2)(1-y^2)}{\left((1-x)(1+y) +a\delta (x-y)\right)^2} \\
			&=  \frac{(1+x)(1-y)}{(1-x)(1+y) + a\delta(x-y)}\,\, \frac{(1-x)(1+y)}{(1-x)(1+y) + a\delta(x-y)}\\
			&\le \frac{(1+x)(1-y)}{(1-x)(1+y) + a\delta(x-y)}\\
			&\le \frac{4}{a\delta} \le \frac{8/a}{\delta}.
		\end{align*} The second to last inequality uses $(1+x)(1-y)<4$  %\HL{$\le 2$?} 
        since $x,y\in(0,1)$ to bound the numerator and $(1-x)\ge a\delta$, $(1+y)\ge 1$ and $a\delta(x-y)>0$ to bound the denominator. This verifies \eqref{eq:G_x_y_claim}. 
		
		We now turn to the second bound. From equation \eqref{eqn:productDenom},
		\begin{equation}
  \nonumber
			\sup\left\{F(\htheta_1,\eta_1,\eta_2,\xi_1): \substack{\displaystyle \eta_1,\eta_2\in [1-B\delta, 1-a\delta]\\
				\displaystyle\htheta_1\in [1-A\delta, 1-a\delta] \\\displaystyle
				\xi_1\in [-1+a\delta, 1-a\delta]}\right\} \le \frac{(2B\delta)^2}{(a\delta)^2}\le \frac{4B^4}{a^2(B-a)^2}
		\end{equation}
        since $1-\eta_j^2 \le 1-(1-B\delta)^2\le 2B\delta$ and $q(\eta_1,\htheta_1\eta_2)\in[0,1]$ implies that $1+\xi_1q(\eta_1,\htheta_1\eta_2) \ge a\delta$. One can similarly obtain the same upper bound when both $\eta_1,\eta_2\in[-1+a\delta,-1+B\delta]$. 
        
		When $\eta_1,\eta_2$ are of opposite signs, we can see that by Claim \ref{claim:strongopposites.},
		\begin{align*}
			\sup\left\{ \left|q(x,y) \right| : \substack{
				\displaystyle |x|,|y|\in [1-B \delta, 1-a\delta]
				\displaystyle,\,\,\textup{sgn}(x)\neq \textup{sgn}(y)
			}\right\} 
			&\le  1- \frac{a}{B}\in(0,1),\label{eqn:twoStrongOPp}
		\end{align*}
        where the last equality holds provided that $B \delta< 1.$ This implies that $\xi_2\in[-1+a/B,1-a/B]$ and so $(1+\xi_2\eta_2)^2 \ge (1-a/B)^2$ by 
        Claim \ref{claim:infxy}. Since we have the generic bound $(1+\xi_1\eta_1)^2\ge (a\delta)^2$ we have
		\begin{align*}
			\sup\left\{F(\htheta_1,\eta_1,\eta_2,\xi_1): \substack{\displaystyle |\eta_1|,|\eta_2|\in [1-B\delta, 1-a\delta]\\
				\displaystyle\htheta_1\in [1-a\delta, 1-a\delta] \\\displaystyle
				\xi_1\in [-1+a\delta, 1-a\delta]\\
				\displaystyle\textup{sgn}(\eta_1)\neq \textup{sgn}(\eta_2)}\right\} \le \frac{(2B\delta)^2}{(a\delta)^2\left(1- a/B\right)^2}  = \frac{4B^4}{a^2 (B-a)^2}.
		\end{align*} 
	\end{proof}

 The next result is a useful bound analogous to Lemma \ref{lem:swap}.
	\begin{prop}\label{prop:4terms_upper_bd}
		Let $\eta_j, \xi_j$ satisfy the recursions in \eqref{eq:def_eta_xi}. Denote $\tilde{\eta}_{4}:=\eta_{4}$, $\tilde{\eta}_{3}:=q(\htheta_3\eta_4,\eta_3)$ and $\tilde{\eta}_{2}:=q(\htheta_2\tilde{\eta}_3,\eta_2)$. Under Assumption \ref{assumption1}, we have 
\begin{align}
\nonumber
  W_{4}
  \le \frac{1}{(2 c_{\ref{eqn:pHatBounds}}\delta)^2 } \frac{\prod_{j=1}^4 (1-\eta_j^2)}{\prod_{j=1}^3 (1+\htheta_j\eta_j{\tilde{\eta}_{j+1}})^2}.
    \end{align}
	\end{prop}

	\begin{proof} By repeatedly applying Lemma \ref{lem:swap}, we see that 
 \begin{align*}
     &(1+\xi_4\eta_4)(1+\eta_3\xi_3) = (1+\htheta_3 \eta_3 \tilde{\eta}_4)(1+\xi_3\tilde{\eta}_3)\\
     &(1+\xi_3\tilde{\eta}_3)(1+\xi_2\eta_2) = (1+\htheta_2 \eta_2 \tilde{\eta}_3)(1+\xi_2\tilde{\eta}_2)\\
     &(1+\xi_2\tilde{\eta}_2)(1+\xi_1\eta_1) = (1+\htheta_1 \eta_1 \tilde{\eta}_2)(1+\xi_1\tilde{\eta}_1)
 \end{align*} where $\tilde{\eta}_1 = q(\htheta_1 \tilde{\eta}_2,\eta_1)$. Hence
\begin{equation*}\label{eqn:lowerBoundDenom4terms}
	\prod_{j=1}^4 (1+\eta_j\xi_j)  = {\left( \prod_{j=1}^3 (1+\htheta_j\eta_j\tilde{\eta}_{j+1})\right) }{\left(1+\xi_1 \tilde{\eta}_1\right)}.
\end{equation*}
Now it suffices to see that the term $1+\xi_1\widetilde{\eta_1}\ge 2c_{\ref{eqn:pHatBounds}}\delta$. Indeed, recall that  $\xi_{1}, \eta_j\in[-1+2c_{\ref{eqn:pHatBounds}}\delta,1-2c_{\ref{eqn:pHatBounds}}\delta]$ by \eqref{eqn:pHatBounds} in Assumption \ref{assumption1}. 
%\snote{By what? Explain.}
Also, $q(x,y)\in[-1,1]$ for all $x,y\in(-1,1)$. Hence  $1+\xi_1 q(x,y)\ge 2c_{\ref{eqn:pHatBounds}}\delta$ for all $x,y\in(-1,1)$. 
	\end{proof}
We will also need a similar representation when dealing with $\widetilde{W}_4$. We include this as a lemma. It follows from repeated applications of Lemma \ref{lem:swap}\textbf{(i)}.
\begin{lemma}\label{lem:eqn:bacwardRecurse}
Let $\eta_j,\xi_j$ satisfy the recursions in \eqref{eq:def_eta_xi}. Then
\begin{equation}\nonumber
  \frac{1}{(1+\xi_5\eta_5)^2} \prod_{j=1}^4 \frac{(1-\eta_j^2)}{(1+\xi_j\eta_j)^2} = \frac{1}{(1+\xi_1 \tilde{\eta}_1)^2}\prod_{j=1}^4 \frac{(1-\eta_j^2)}{(1+\htheta_j\eta_j \tilde{\eta}_{j+1})^2}
\end{equation}
where $\tilde{\eta}_5 = \eta_5$ and for $i \in\{1,2,3,4\}$ the term $\tilde{\eta}_i = q(\htheta_i\tilde{\eta}_{i+1}, \eta_i)$.
\end{lemma}

    \subsection{Decomposition of the sample space}
    \label{sec:Aijk}

    We decompose the sample space into a collection of disjoint events depending on the tiers of the signals.
		\begin{enumerate}
			\item We say that there is a \textit{flip} 
            at two neighboring vertices $u,v\in \{y_1,\dots, y_4, w_1,\dots,w_4\}$ if 
            $\sigma_u\neq \sigma_v$;
			\item We say that reconstruction fails \textit{moderately at} $w_j$ if $\sigma_{w_j}\eta_j\in [ -c_{\ref{eqn:antiReconstruction}}, 1-C_{\ref{eqn:etareconstruct}}\delta]$;
			\item We say that reconstruction fails \textit{severely at} $w_j$ if $\sigma_{w_j} \eta_j< -c_{\ref{eqn:antiReconstruction}}$.
		\end{enumerate} 
        Using Corollary \ref{cor:eta}, the above events are closely related to whether or not the magnetizations in Thm. \ref{thm:Robust} reconstruct the true signal well. By Claim \ref{claim:unsigned_mag}, the above events are independent for $j=1,\dots,4$.

  In the remainder of the proof, when we say that there is a \textit{flip} (resp. \textit{failure}) we will implicitly refer to a flip happening between two of the vertices involved in equation \eqref{eqn:4terms} (resp. at one of the vertices in \eqref{eqn:4terms}). Define the events
		\begin{align}
			\nonumber M_j&:= \{\textup{there is a moderate failure at }w_j\}, \\
			\nonumber S_j&:=\{\textup{there is a severe failure at }w_j\}\\
            \label{eqn:AijkDef} A_{i,j,k} &:= \left\{\begin{matrix}  \textup{there are exactly }i\textup{  flips,} \\
            \textup{exactly }j\textup{ moderate failures, and exactly }k\textup{  severe failures}
            \end{matrix}
            \right\}.
		\end{align}
		These events $A_{i,j,k}$s are disjoint and
\begin{equation}\label{eqn:disjoint_union}
	\PR_{\param^{*}}\left(\bigcup_{i=0}^7 \bigcup_{j=0}^4 \bigcup_{k=0}^4 A_{i,j,k}\right) = 1.
		\end{equation}
		Now by the reconstruction theorem (Thm. \ref{thm:Robust}) or Corollary \ref{cor:eta},  
		\begin{align*}
			\PR_{\param^{*}}(M_j)\le c_{\ref{eqn:Reconstruct}} \delta \qquad \textup{and} \qquad \PR_{\param^{*}}(S_j) \le C_{\ref{eqn:antiReconstruction}}\delta^2.
		\end{align*}
		By the independence property (Claim \ref{claim:unsigned_mag}), we get 
\begin{equation}\label{eqn:AijkBounds}
			\PR_{\param^{*}}(A_{i,j,k}) \le \binom{7}{i}\binom{4}{j,k} \left(C_{\ref{eqn:pBounds}} \delta \right)^i  (c_{\ref{eqn:Reconstruct}}\delta)^j (C_{\ref{eqn:antiReconstruction}}\delta^2)^k = O(\delta^{i+j+2k}).
		\end{equation}
		Together with \eqref{eqn:disjoint_union}, it follows that 
		\begin{align}\nonumber %\label{eqn:bound1}
			\PR\left(\bigcup_{i+j+2k\le 2} A_{i,j,k} \right)\ge 1 - O(\delta^3).
		\end{align}

\begin{remark}
    Note that the signals $|\eta_j|\le \htheta_j\le 1-2c_{\ref{eqn:pHatBounds}}\delta$. This fact will be used frequently in the sequel.
\end{remark}

\begin{remark}\label{rmk:strongmagnotation} In the sequel we will frequently write the condition $\sigma_{w_j}\eta_j \ge 1-C_{\ref{eqn:etareconstruct}}\delta$ as either $\sigma_{w_j}\eta_j = 1-\Theta(\delta)$ (noting the previous remark) or $1-O(\delta)$. We will use the latter if only the lower bound is needed. We will only refer to the precise constant when needed for clarity. However, we will write $\eta_j \in[- c_{\ref{eqn:antiReconstruction}}, 1-C_{\ref{eqn:etareconstruct}}\delta]$ as $\eta_j\in[-c_{\ref{eqn:antiReconstruction}}, 1-O(\delta)]$ and keep explicit reference to $c_{\ref{eqn:antiReconstruction}}$.
\end{remark}

	\subsection{Proof of Lemma \ref{lem:4terms}}
	
	In this section, we prove Lemma \ref{lem:4terms}, which gives moment bounds for the random variable $W_{i}$ defined in \eqref{eq:def_W_i}. Recall that without loss of generality, we will only prove it for $W_{4}$ (see \eqref{eq:W4_W4_tilde}). 
	
	The starting point of our analysis is the following generic bound which follows directly from Corollary \ref{cor:eta} and Proposition \ref{prop:twoTerms}: Almost surely, 
	\begin{equation}\label{eq:W4_2_2}
	W_{4}
    \le \sup_{|\xi_1|\le 1-2c_{\ref{eqn:pHatBounds}}\delta}  F(\htheta_1,\eta_1,\eta_2,\xi_1)  \sup_{|\xi_3|\le 1-2c_{\ref{eqn:pHatBounds}}\delta}  F(\htheta_3,\eta_3,\eta_4,\xi_3) = O(\delta^{-2}).
	\end{equation} 
	Recalling that the off-diagonal entries of the Hessian are product of terms like $W_{4}$, the bound above is too crude to yield any meaningful control on the off-diagonal entries of the Hessian. Our proof of Lemma \ref{lem:4terms} proceeds by some careful case analysis of the value of $W_{4}$ depending on the three tiers (see \eqref{eq:recons_tiers}) of the signals $\eta_{1},\dots,\eta_{4}$.

        In order to show \eqref{eqn:boundForFiniteSample}, we will prove the (slightly stronger) claim below, which because of \eqref{eqn:AijkBounds}, is enough to conclude \eqref{eqn:boundForFiniteSample}.
        \begin{lemma}[Detailed Description of $W_4$]\label{lem:4termsAlt}
           Under the assumptions of Lemma \ref{lem:4terms},  \begin{equation}\label{eqn:boundForFiniteSample_pf}
		W_{4} \in \begin{cases}[0,C_{\ref{eqn:boundForFiniteSample}}\delta^2]&:\textup{on }A_{0,0,0}\\
         [0,C_{\ref{eqn:boundForFiniteSample}}\delta]&:\textup{on }A_{0,1,0}\\
			[0,C_{\ref{eqn:boundForFiniteSample}}]&:\textup{on }A_{1,0,0}\cup A_{2,0,0} \cup A_{3,0,0}\cup A_{0,2,0}\\
			[0,C_{\ref{eqn:boundForFiniteSample}}\delta^{-1}]&:\textup{ on }A_{0,3,0}\cup A_{1,1,0}\cup A_{2,1,0}\cup A_{0,0,1} \cup A_{1,0,1} \\
			[0,C_{\ref{eqn:boundForFiniteSample}}\delta^{-2}]&:\textup{ on the sample space }\Omega
		\end{cases}.
	\end{equation}
        \end{lemma}

		\begin{proof}

        We bound the value of $W_{4}$ on various events in \eqref{eqn:boundForFiniteSample_pf}. 
        \vspace{0.1cm}
  
		\textbf{On the event $A_{0,0,0}\cup A_{0,1,0} \cup A_{0,2,0}\cup A_{0,3,0}$:} On this event we do not have any signal flips, but we do allow for moderate failures of reconstruction. We suppose, without loss of generality, that $\sigma_{w_1} = 1$. Consequently, all the signals are $+1$ and by the definition of moderate failures we see that 
        \begin{align*}
          \eta_i\ge -c_{\ref{eqn:antiReconstruction}}
         \end{align*} 
        for each $i\in[4]$.
  
    Recall the representation in Proposition \ref{prop:4terms_upper_bd} using $\tilde{\eta}_j = q(\htheta_j\eta_{j+1}, \eta_j)$ and $\tilde{\eta}_4 = \eta_4.$ We see that on the event $\bigcup_{r = 0}^3A_{0,r,0}$ it holds that there is some constant $\kappa\in(0,1)$ (by Claim \ref{claim:qRecurseClaim}) such that  
		\begin{align*}
			\tilde{\eta}_j \ge -\kappa \qquad\textup{ for } j = 1,2,3,4.
		\end{align*} 
		It is easy to check using Claim \ref{claim:infxy} that for each $j = 1,2,3$ that 
  \begin{equation*}			1+\htheta_j\eta_j\tilde{\eta}_{j+1}\ge 1-\kappa.
		\end{equation*}
		
 On $A_{0,r,0}$ we have 
 $\#\{i: \sigma_{w_i}\eta_i\ge 1-O(\delta)\}= 4-r$. Therefore, using Proposition \ref{prop:4terms_upper_bd}, we see that 
		\begin{equation}\nonumber 
			\mathbf{1}_{A_{0,r,0}} W_{4} 
            \le \frac{1}{(2c_{\ref{eqn:pHatBounds}}\delta)^2}.\frac{\prod_{i=1}^4 (1-\eta_i^2)}{\prod_{i=1}^3 (1+\htheta_i\eta_i{\tilde{\eta}_{i+1}})^2} \le 
            \frac{O(\delta^{4-r})}{(2c_{\ref{eqn:pHatBounds}}\delta)^2(1-\kappa)^6}  = O(\delta^{2-r}).
		\end{equation}

  \vspace{0.1cm}
		\textbf{On the event $A_{1,0,0}\cup A_{2,0,0}\cup A_{3,0,0}$:} 
		We now turn to the event where we allow for flips but no failures of reconstruction. 
		As there are no failures on this event, each of the $\eta_j$'s satisfies $|\eta_j| = |\sigma_{w_j}\eta_j|=1-\Theta(\delta).$ 
		Hence, a direct application of the second conclusion of Proposition \ref{prop:twoTerms} implies
		\begin{align}
		&\mathbf{1}_{A_{1,0,0}\cup A_{2,0,0} \cup A_{3,0,0}}  W_{4} 
            \\
            &\qquad \le \mathbf{1}_{A_{1,0,0}\cup A_{2,0,0}\cup A_{3,0,0}} \sup_{|\xi_1|\le 1-2c_{\ref{eqn:pHatBounds}}\delta} \, F(\htheta_1,\eta_1,\eta_2,\xi_1) \sup_{|\xi_3|\le 1-2c_{\ref{eqn:pHatBounds}}\delta}   F(\htheta_3,\eta_3,\eta_4,\xi_3) = O(1).
		\end{align}

		\vspace{0.1cm}
		\textbf{On the event $A_{1,1,0}\cup A_{2,1,0} \cup A_{0,0,1}\cup A_{1,0,1}$:}  
		On this event, we allow for either a single moderate failure or a single severe failure but not both. Note that there are at least three successful reconstructions and hence on $A_{1,1,0}\cup A_{2,1,0}\cup A_{0,0,1}\cup A_{1,0,1}$, we have 
  \begin{equation*}
        \#\{i\in[4]: |\eta_i|\ge 1-O(\delta)\} \ge 3.
  \end{equation*} (Note that the count above can be 4 %we have $\ge 3$ 
  as the severe failure can also satisfy $|\eta_i|\ge 1-O(\delta)$.) By the pigeonhole principle, either $|\eta_{i}|\ge 1-O(\delta)$ for both $i=1,2$ or $|\eta_i|\ge 1-O(\delta)$ for both $i=3,4$. Then one application of each part of Proposition \ref{prop:twoTerms}  gives  that, on $A_{1,1,0}\cup A_{2,1,0} \cup A_{0,0,1}\cup A_{1,0,1}$, 
		\begin{align*}
        W_{4}
        &\le \sup_{|\xi_1|\le 1-2c_{\ref{eqn:pHatBounds}}\delta}   F(\htheta_1,\eta_1,\eta_2,\xi_1) \sup_{|\xi_3|\le 1-2c_{\ref{eqn:pHatBounds}}\delta}   F(\htheta_3,\eta_3,\eta_4,\xi_3)= O(\delta^{-1}).
		\end{align*}
		This finishes the proof.
	\end{proof}

		\subsection{Proof of Lemma \ref{lem:5terms}: The $\widetilde{W}_N$ term}

 In this section we prove the bounds on $\widetilde{W}_N$ found in Lemma \ref{lem:5terms}. Recall that we focus on the case of $\widetilde{W}_{4}$ in \eqref{eq:W4_W4_tilde}.  
	
	\begin{proof}[\textbf{Proof of $\widetilde{W}_N$ bounds in Lemma \ref{lem:5terms}}]
		Observe that, by Proposition \ref{prop:twoTerms} and a trivial bound on the pre-factor $\frac{1}{(1+\xi_5\eta_5)^2}\le \frac{1}{(2c_{\ref{eqn:pHatBounds} }\delta)^2} = O(\delta^{-2})$, we get 
  \begin{align}
		\widetilde{W}_4 
        &=O(\delta^{-2}) \sup_{|\xi_1| \le 1-2c_{\ref{eqn:pHatBounds}}\delta} \prod_{j=1}^4 \frac{(1-\eta_j^2)}{(1+\xi_j\eta_j)^2} 
  \label{eqn:LinftyBoundfor5} = O(\delta^{-2})W_4 =O(\delta^{-4}).
		\end{align}
		and so we just need to analyze the events that have probability at most $O(\delta^4)$. 
    Moreover, for any random variables $X$ it holds that
	\begin{equation}\label{eqn:genHolder}
		\E[X^2]\le \|X\|_{L^\infty} \, \E[|X|].
	\end{equation}  Therefore, the second moment bound in Lemma \ref{lem:5terms} for $\widetilde{W}_4$ follows from the first moment bound and \eqref{eqn:LinftyBoundfor5}.
		
    In the rest of the proof, we will show the first moment bound on $\widetilde{W}_{4}$ as stated in \eqref{eqn:4termsDiag}. We use the same definitions for the events $A_{i,j,k}$ as in \eqref{eqn:AijkDef} and that for `flips', `moderate failures', and `severe failures' in the paragraph preceding \eqref{eqn:AijkDef}. 
    Observe that the events $A_{i,j,k}$ do not involve the random variable $\eta_5 = Z_{y_5} = Z_x$ or the signal $\sigma_{y_5} = \sigma_x$. To account for these variables at the fifth node $y_{5}$ contributing to $\widetilde{W}_{4}$, define events 
  \begin{align*}
  B_1 := \{\sigma_{y_5}= \sigma_{y_4}\} \quad \textup{and} \quad  B_2 := \{\sigma_{y_5}\eta_5\ge 1-C_{\ref{eqn:etareconstruct}}\delta\}.
  \end{align*}
  We will frequently write the shorthand
  \begin{equation*}
      B_2 = \{\sigma_{y_5} \eta_5 \ge 1-O(\delta)\} 
  \end{equation*}
  however we will never write $\sigma_{y_5}\eta_5 = 1-\Theta(\delta)$ as this is a.s. \textit{false} when $y_5 = x$ is a leaf.

  By Claim \ref{claim:unsigned_mag}, $B_1, B_2$ are independent of each other 
  and of any of the $A_{i,j,k}$'s, and by Assumption \ref{assumption1} and Thm. \ref{thm:Robust} \textbf{(i)}, 
  \begin{equation*}
      \PR(B_1) \ge 1-2c_{\ref{eqn:pHatBounds}}\delta \qquad \textup{and}\qquad \PR(B_2)\ge 1-c_{\ref{eqn:Reconstruct}}\delta.
  \end{equation*}
  Let us write
  \begin{equation*}
  \mathcal{A}:= \left(B_1 \cap B_2 \cap \bigcup_{i+j+2k\le 3} A_{i,j,k}\right)  \cup B_1^c \cup B_2^c.
  \end{equation*}
  Observe that by using \eqref{eqn:AijkBounds} 
  and the above bounds, 
		\begin{align}\nonumber 
			\PR\left( \mathcal{A} \right)= 1-  O(\delta^4).
		\end{align}
    Together with \eqref{eqn:LinftyBoundfor5}, this yields $\E[\widetilde{W}_4 \mathbf{1}_{\mathcal{A}^{c}}]= O(1)$. Hence in order to show the first moment bound on $\widetilde{W}_{4}$ in \eqref{eqn:4termsDiag}, it suffices to show that
  \begin{equation}\label{eqn:tildeW4L1bound1}
      \E[\widetilde{W}_4 \mathbf{1}_{\mathcal{A}}] \le C_{\ref{eqn:tildeW4L1bound1}}.
  \end{equation}
    for some  constant $C_{\ref{eqn:tildeW4L1bound1}}<\infty$. 
  
To establish the above, first note that by Claim \ref{claim:unsigned_mag}, $B_1$ and $B_2$ are independent of $\prod_{j=1}^4 \frac{(1-\eta_j^2)}{(1+\xi_j\eta_j)^2}$, and hence $W_{4}$. It follows from \eqref{eqn:LinftyBoundfor5} that
\begin{align*}
    \E[\widetilde{W}_4\mathbf{1}_{B_1^c\cup B_2^c}]\le\PR(B_1^c\cup B_2^c)\times O(\delta^{-2}) \E[W_4] = O(1).
\end{align*}

  Then we can use the ``approximate distribution'' in Lemma \ref{lem:4termsAlt}, equation \eqref{eqn:boundForFiniteSample_pf}, along with the bounds \eqref{eqn:AijkBounds} for the probability of the events $A_{i,j,k}$. First note that 
 	\begin{equation*}\label{eqn:simplify1}
			\E\left[W_4%\sup_{|\xi_1|\le 1-2c_{\ref{eqn:pHatBounds}}\delta} \prod_{j=1}^4 \frac{(1-\eta_j^2)}{(1+\xi_j\eta_j)^2} 
   \mathbf{1}_{[A_{0,0,0} \cup A_{2,0,0} \cup A_{3,0,0} \cup A_{0,1,0}\cup A_{0,2,0}\cup A_{2,1,0}\cup A_{1,0,1}]}\right] = O(\delta^2).
		\end{equation*} 
  We conclude using the first bound in \eqref{eqn:LinftyBoundfor5} that 
  \begin{align*}
      \E&\left[\widetilde{W}_4 \mathbf{1}_{[A_{0,0,0} \cup A_{2,0,0} \cup A_{3,0,0} \cup A_{0,1,0}\cup A_{0,2,0}\cup A_{2,1,0}\cup A_{1,0,1}]} \right]\\
     \nonumber &\le  O(\delta^{-2}) \E\left[{W}_4 \mathbf{1}_{[A_{0,0,0} \cup A_{2,0,0} \cup A_{3,0,0} \cup A_{0,1,0}\cup A_{0,2,0}\cup A_{2,1,0}\cup A_{1,0,1}]} \right] = O(1).
  \end{align*}
    Examining the definition of the event $\mathcal{A}$, %just above \eqref{eqn:tildeW4L1bound1}, 
    it remains  to analyze $\widetilde{W}_4$ on the event 
    \begin{align*}
        \mathcal{A}&\setminus (B_1^c\cup B_2^c \cup A_{0,0,0} \cup A_{2,0,0} \cup A_{3,0,0} \cup A_{0,1,0}\cup A_{0,2,0}\cup A_{2,1,0}\cup A_{1,0,1}) \\
        &=  B_1\cap B_2\cap \left(A_{1,0,0} \cup A_{0,0,1} \cup A_{1,1,0}\cup A_{1,2,0}  \cup A_{0,1,1} \right).
    \end{align*}
	 We now proceed to bound $\widetilde{W}_4$ systematically on this remaining event, splitting it into three cases depending on which $A_{i,j,k}$ occurs.

        \vspace{0.1cm}
		\textbf{On the event $B_1\cap B_2 \cap (A_{1,0,0}\cup A_{0,0,1})$:} 
        For this case, $A_{1,0,0}\cup A_{0,0,1}$ means that there is one flip or one severe failure (see Sec. \ref{sec:Aijk}). On this event, we show that the denominator $1+\xi_{5}\eta_{5}$ in the prefactor in $\widetilde{W}_{4}$ is bounded away from 0. Then we can combine it with the bound on $W_{4}$ in Lemma \ref{lem:4termsAlt} to conclude.

        Without loss of generality, we assume that $\sigma_{y_5} = +1$. As we are on the event $B_1$ we also see that $\sigma_{y_4} = +1$. 
    We claim that on this event, 
\begin{equation}\label{eqn:b1b2a100a001etas}
    \eta_{1},\eta_2,\eta_3\ge 1-C_{\ref{eqn:etareconstruct}}\delta\qquad\textup{ or }\qquad \eta_{4}\ge 1-C_{\ref{eqn:etareconstruct}}\delta .
    \end{equation}
    Indeed, on $A_{1,0,0}\cap \{\sigma_{y_4} = +1\}$, there is only a single flip, which can either be on the edge $\{y_4,w_4\}$ or not. If it is on that edge, then $\sigma_{w_4} = -1$ is the only spin of $-1$, and otherwise $\sigma_{w_4} = +1$. So either 
    \begin{equation*}
        \sigma_{w_1} = \sigma_{w_2} = \sigma_{w_3} = +1 \qquad\textup{ or }\qquad \sigma_{w_4} = +1
    \end{equation*} and \eqref{eqn:b1b2a100a001etas} 
    holds as there are no failures of reconstruction. On $A_{0,0,1}\cap \{\sigma_{y_{4}}=+1\}$, we have $\sigma_{w_{1}}=\dots=\sigma_{w_{4}}=+1$ and there is a single severe failure and no moderate failure. If that severe failure happens at $w_{4}$, then we have successful reconstruction for $w_{1},w_{2},w_{3}$. But since all spins at these nodes are $+1$, we have the first inequality in \eqref{eqn:b1b2a100a001etas}. Otherwise, we have successful reconstruction for $w_{4}$ and get the second inequality in \eqref{eqn:b1b2a100a001etas}. This shows the claim.

    Now recall that $\xi_1\ge -1+2c_{\ref{eqn:pHatBounds}}\delta$ from Assumption \ref{assumption1}. Recalling the definition of $\xi_4$
    from \eqref{eq:def_eta_xi} we see
        \begin{equation*}
            \frac{\xi_4}{\htheta_{3}} =  q\left(\eta_3, \htheta_2 q(\eta_2, \htheta_1q(\eta_1, \xi_1)\right)
        \end{equation*}
        is precisely of the form in \eqref{eqn:poly1} in Claim \ref{claim:quickBound2}. Hence by using Claim \ref{claim:quickBound2} and Assumption \ref{assumption1}, we see that if  $\eta_1,\eta_2,\eta_3\ge 1-C_{\ref{eqn:etareconstruct}}\delta$ then
        \begin{equation}\nonumber
            \xi_4\ge 1-O(\delta)\qquad\textup{and}\qquad \eta_4\ge -1+\Omega(\delta). 
        \end{equation} 
        Otherwise, we have the generic bound $\xi_4 \ge -1+\Omega(\delta)$ and $\eta_4\ge 1-O(\delta)$.
        Therefore, we see that in either case, 
\begin{equation}\label{eqn:helper1forB1B2A100}
        \max\left(\eta_4,\xi_4 \right)\ge 1-O(\delta)\qquad \textup{and}\qquad \min\left(\eta_4,\xi_4 \right)\ge -1+\Omega(\delta).
        \end{equation}        
        Applying Claim \ref{claim:strongopposites.}, we can find a constant $\kappa\in(0,1)$ such that 
        \begin{equation}\label{eqn:antieqn1}
			\xi_5  = \htheta_4 q(\xi_4,\eta_4)\ge (1-O(\delta)) q(\xi_4,\eta_4)\ge -\kappa.
		\end{equation}

		Hence, note that on the event $B_1\cap B_2\cap (A_{1,0,0}\cup A_{0,0,1})\cap \{\sigma_{y_5} = +1\}$, we have $\eta_5 = \sigma_{y_5} \eta_5 \ge 1-O(\delta)$ and $\xi_5\ge -\kappa$ and so $(1+\xi_5\eta_5) \ge 1-\kappa$ (by Claim \ref{claim:infxy}). Hence
		\begin{equation*}
	\widetilde{W}_4\mathbf{1}_{B_1\cap B_2\cap (A_{1,0,0}\cup A_{0,0,1})\cap\{\sigma_{y_5} = +1 \}}\le \frac{1}{(1-\kappa)^2} W_4\mathbf{1}_{ A_{1,0,0} \cup A_{0,0,1 }}.
		\end{equation*} 
  It follows from Lemma \ref{lem:4termsAlt} and \eqref{eqn:AijkBounds} that 
		\begin{equation}
			\nonumber
\E\left[\widetilde{W}_4\mathbf{1}_{B_1\cap B_2\cap (A_{1,0,0}\cup A_{0,0,1})\cap\{\sigma_{y_5} = +1 \}} \right] =O(\delta).
		\end{equation}

		\vspace{0.1cm}
		\textbf{On the event $B_1\cap B_2 \cap (A_{1,1,0}\cup A_{1,2,0})$:} Let us now turn to the case where there is a single flip and one or two  moderate failures (see Sec. \ref{sec:Aijk}). 
        Without loss of generality, we suppose that $\sigma_{y_5} = +1$. This implies that $\sigma_{y_4} = +1$ as we are on the event $B_1$. 
        Since there is a single flip and no severe failure, we can break it into three sub-cases: 
        \begin{description}
            \item{(1)} $\sigma_{w_{4}}=+1$ and good reconstruction at $w_{4}$; 

            \item{(2)} $\sigma_{w_{4}}=+1$ and moderate failure at  $w_{4}$; 

            \item{(3)} $\sigma_{w_{4}}=-1$.
        \end{description}

        For the sub-case (1), we have  $\sigma_{w_4} =+1$ and $|\eta_4|\ge 1-O(\delta)$. Then since there is no severe failure in this case, we have $\eta_4\ge 1-O(\delta)$. So with Assumption \ref{assumption1}, the bounds in \eqref{eqn:helper1forB1B2A100} hold and therefore so does the lower bound on $\xi_{5}$ in \eqref{eqn:antieqn1}, for some $\kappa\in(0,1)$. 
        Hence 
		$(1+\xi_5\eta_5)\ge 1-\kappa$ (by Claim \ref{claim:infxy}). Thus by  Lemma \ref{lem:eqn:bacwardRecurse}, 
		\begin{align}
			       \widetilde{W}_4\mathbf{1}_{B_1\cap B_2 \cap (A_{1,1,0}\cup A_{1,2,0}) \cap \{\sigma_{w_4} = \sigma_{y_4} = 1, |\eta_4|\ge 1-O(\delta)\}}			&\le \frac{1}{(1-\kappa)^2} W_4\mathbf{1}_{A_{1,1,0}\cup A_{1,2,0}}  = O(\delta^{-1})\mathbf{1}_{A_{1,1,0}\cup A_{1,2,0}} 
          \label{eqn:a110part1}
		\end{align} 
        The expectation of the last expression is at most $O(1)$ by \eqref{eqn:AijkBounds}.
        This deals with the case where neither the failure of reconstruction nor the flip occurs at $w_4$. 

        Next, assume the sub-case (2): $\sigma_{w_4} = 1$ and there is a moderate failure at $w_4$. Then the definition of $B_{2}$, the monotonicity of $q$, and the definition of $q$ yields that, for $\delta$ small enough that $C_{\ref{eqn:etareconstruct}}\delta+c_{\ref{eqn:antiReconstruction}}\le 1$, 
        \begin{align*}
            \tilde{\eta}_4 = q(\htheta_4 \eta_5, \eta_4)\ge 
            q(1-C_{\ref{eqn:etareconstruct}}\delta, -c_{\ref{eqn:antiReconstruction}})
            \in[0,1]. 
        \end{align*}    
        Then  by using Lemma \ref{lem:swap}\textbf{(i)} and \textbf{(ii)}, letting $\widetilde{\eta}_5 = \eta_5$,
  \begin{align}
    \frac{1}{(1+\xi_5\eta_5)^2} \frac{(1-\eta_4^2)}{(1+\xi_4\eta_4)^2} &= \frac{1}{(1+\htheta_4\eta_4\tilde{\eta}_5 )^2} \frac{(1-\eta_4^2)}{(1+\xi_4\tilde{\eta}_4)^2} \qquad\qquad\textup{by \textbf{(i)}}\nonumber\\
    &=\frac{1}{(1+\htheta_4\eta_4\tilde{\eta}_5 )^2} \frac{(1-\eta_4^2)}{(1-\tilde{\eta}_4^2)} \frac{(1-\tilde{\eta}_4^2)}{(1+\xi_4\tilde{\eta}_4)^2} \qquad \nonumber\\
  \nonumber  &=\frac{1}{(1-(\htheta_4\tilde{\eta}_5)^2 )} \frac{(1-\tilde{\eta}_4^2)}{(1+\xi_4\tilde{\eta}_4)^2}\qquad\qquad\textup{by \textbf{(ii)}}%\label{eqn:with5xis}. 
  \\
  \nonumber &\le \frac{1}{4c_{\ref{eqn:pHatBounds}}\delta} \frac{1-\tilde{\eta}_4^2}{(1+\xi_4\tilde{\eta}_4)^2} \qquad \qquad \textup{by }1-C_{\ref{eqn:etareconstruct}}\delta \le \htheta_4\tilde{\eta}_5 \le 1-2c_{\ref{eqn:pHatBounds}}\delta\textup{ from }B_{2}\textup{ and Assumption \ref{assumption1}.}
\end{align}
This yields 
\begin{align*}
    \widetilde{W}_{4} &\le \sup_{|\xi_3|\le 1-2c_{\ref{eqn:pHatBounds}}\delta} \frac{1}{(1+\xi_5\eta_5)^2} \frac{(1-\eta_4^2)}{(1+\xi_4\eta_4)^2}  \frac{(1-\eta_3^2)}{(1+\xi_3\eta_3)^2} 
            \sup_{|\xi_1|\le 1-2c_{\ref{eqn:pHatBounds}}\delta}\, F(\hat{\theta}_{1},\eta_{1},\eta_{2},\xi_{1}) \\
            & \le \frac{1}{4c_{\ref{eqn:pHatBounds}}\delta}  \sup_{|\xi_3|\le 1-2c_{\ref{eqn:pHatBounds}}\delta} \frac{1-\tilde{\eta}_4^2}{(1+\xi_4\tilde{\eta}_4)^2} 
           \frac{(1-\eta_3^2)}{(1+\xi_3\eta_3)^2} 
            \sup_{|\xi_1|\le 1-2c_{\ref{eqn:pHatBounds}}\delta}\, F(\hat{\theta}_{1},\eta_{1},\eta_{2},\xi_{1}) \\
            &=\frac{1}{4c_{\ref{eqn:pHatBounds}}\delta}  \sup_{|\xi_3|\le 1-2c_{\ref{eqn:pHatBounds}}\delta} F(\htheta_3,\eta_3,\tilde{\eta}_4,\xi_3)  \sup_{|\xi_1|\le 1-2c_{\ref{eqn:pHatBounds}}\delta}\, F(\htheta_1,\eta_1,\eta_2,\xi_1).
\end{align*}
Now if we further assume $A_{1,1,0}$, then all of $w_{1},w_{2},w_{3}$ have successful reconstruction and there is one flip except on the edge $\{y_{4},w_{4}\}$. If we further assume $A_{1,2,0}$, then two of $w_{1},w_{2},w_{3}$ have successful reconstruction and the other has a moderate failure, while there is one flip except on the edge $\{y_{4},w_{4}\}$. Thus by using Proposition \ref{prop:twoTerms}, we deduce from the bound above that 
		\begin{align}
            &\widetilde{W}_{4} \mathbf{1}_{B_1\cap B_2\cap (A_{1,1,0}\cup A_{1,2,0}) \cap \{\sigma_{y_4} = \sigma_{w_4} = 1, \textup{ moderate failure at }w_4\}} 
			 \le \frac{1}{4c_{\ref{eqn:pHatBounds}}\delta}\left( O(\delta^{-1})\mathbf{1}_{A_{1,1,0}}  + O({\delta^{-2}}) \mathbf{1}_{A_{1,2,0}}\right), \label{eqn:a110part3}
		\end{align}

		Lastly, assume the sub-case (3) where $\sigma_{w_4} = -1$, meaning that the single flip occurred between $y_4$ and $w_4$, so $\sigma_{w_{1}}=\sigma_{w_{2}}=\sigma_{w_{3}}=+1$. Then, $\eta_4\ge -1+2c_{\ref{eqn:pHatBounds}}\delta$ as that always holds by Assumption \ref{assumption1}, and $\eta_i = \sigma_{w_i}\eta_i \ge -c_{\ref{eqn:antiReconstruction}}$ for $i\in[3]$ since there are no severe failures.  Let us define $\tilde{\eta}_5 = \eta_5$ and $\tilde{\eta}_i = q(\htheta_i\tilde{\eta}_{i+1}, \eta_i)$ as in Lemma \ref{lem:eqn:bacwardRecurse}. Note that by Claim \ref{claim:strongopposites.}
		\begin{align*}
			\tilde{\eta}_4 &= q(\htheta_4\eta_5, \eta_4) \ge q(1-O(\delta), -1+\Omega(\delta)) \ge -\kappa\quad  \textup{for some constant }\kappa\in(0,1).
		\end{align*} 
        Now, by an application of Claim \ref{claim:qRecurseClaim}, we can find some other constant $\kappa'\in(0,1)$ such that \begin{equation*}\label{eqn:qRecurseAntiReconstruct3}
			\eta_1,\eta_2,\eta_3, \tilde{\eta}_1,\tilde{\eta}_2,\tilde{\eta}_3,\tilde{\eta}_4 \ge -\kappa'.
		\end{equation*} Combining this with Claim \ref{claim:infxy}, we see 
		\begin{equation*}
(1+\htheta_j\eta_j\tilde{\eta}_{j+1}) \ge 1-\kappa',\qquad \textup{for }j = 1,2,3.
		\end{equation*}
		By Assumption \ref{assumption1}, we have  $1+\htheta_4\eta_4\tilde{\eta}_5 , 1+\xi_1\tilde{\eta}_1 = \Omega(\delta)$. Note that there are at least two successful reconstructions among $\eta_{1},\dots,\eta_{4}$ on the event $A_{1,1,0}\cup A_{1,2,0}$. 
  Therefore by Lemma \ref{lem:eqn:bacwardRecurse},
		\begin{align}
		&\widetilde{W}_{4}	\mathbf{1}_{B_1\cap B_2 \cap (A_{1,1,0}\cup A_{1,2,0}) \cap \{\sigma_{w_4} = -1, \sigma_{y_4} = 1\}} \nonumber \\
		&\qquad = \sup_{|\xi_1|\le 1-2c_{\ref{eqn:pHatBounds}}\delta}  \frac{1}{(1+\xi_1\tilde{\eta}_1)^2} \prod_{j=1}^4 \frac{(1-\eta_j^2)}{(1+\htheta_j\eta_j\tilde{\eta}_{j+1})^2} \mathbf{1}_{B_1\cap B_2 \cap (A_{1,1,0}\cup A_{1,2,0}) \cap \{\sigma_{w_4} = -1, \sigma_{y_4} = 1\}} \nonumber \\
			&\qquad \le \frac{1}{\Omega(\delta)^2} \frac{O(\delta)^2 O(1)^2}{\Omega(\delta)^2 (1-\kappa')^6} \mathbf{1}_{A_{1,1,0}\cup A_{1,2,0}} = O(\delta^{-2})\mathbf{1}_{A_{1,1,0}\cup A_{1,2,0}}.\label{eqn:a110part2}
		\end{align}
        Since $\PR_{\param^{*}}(A_{1,1,0}\cup A_{1,2,0})=O(\delta^{2})$ from \eqref{eqn:AijkBounds}, the expectation of the last expression is $O(1)$. We conclude this sub-case.

     Examining equations \eqref{eqn:a110part1}, \eqref{eqn:a110part3}, and \eqref{eqn:a110part2},  we see that
		\begin{equation}\nonumber %\label{eqn:a110part4}
            \widetilde{W}_{4} 
            \le O(\delta^{-2}) \mathbf{1}_{A_{1,1,0}} + O(\delta^{-3}) \mathbf{1}_{A_{1,2,0}}.
		\end{equation} 
        Using the bounds in \eqref{eqn:AijkBounds} for $\PR(A_{i,j,k})$, we have $\PR(A_{1,2,0})= O(\delta^3)$ and $\PR(A_{1,1,0}) = O(\delta^2)$,         and so we conclude 
		\begin{align}
			\nonumber 
          \E_{\param^{*}}\left[\widetilde{W}_{4} \mathbf{1}_{B_1\cap B_2\cap (A_{1,1,0}\cup A_{1,2,0})}\right]
        =O(1).
		\end{align}

        \vspace{0.1cm}
		\textbf{On the event $B_1\cap B_2 \cap A_{0,1,1}$:} On this event, there is no flip (so WLOG, we  assume $\sigma_{z}=+1$ for all $z\in \{y_{1},\dots,y_{4},w_{1},\dots,w_{4}\}$) and there is one moderate and one severe failure among $w_{1},\dots,w_{4}$. 
        Note that since there are two failures of reconstruction, we have at least two $j\in[4]$ such that $\eta_j\ge 1-O(\delta)$. On the event 
        $B_{3} := \{\eta_1,\eta_2\ge 1-O(\delta)\}\cup \{\eta_3,\eta_4\ge 1-O(\delta)\}$, we have $W_{4}=O(\delta^{-1})$ by \eqref{eq:W4_2_2} and Proposition \ref{prop:twoTerms}. Hence using the second equality in \eqref{eqn:LinftyBoundfor5} 
        and \eqref{eqn:AijkBounds}, we deduce 
		\begin{align*}
           \E\left[\widetilde{W}_{4} \mathbf{1}_{[B_1\cap B_2 \cap A_{0,1,1} \cap \{\sigma_{y}=1\}\cap B_{3}]} \right]
           &= O(\delta^{-2}) \,  \E[W_{4} \mathbf{1}_{[B_1\cap B_2 \cap A_{0,1,1} \cap \{\sigma_{y}=1\}\cap B_{3}]} ]\\
           &=
         O(\delta^{-3})\PR(A_{0,1,1})  = O(1).
		\end{align*} 
		
		It remains to consider the case where $\eta_{j}\ge 1-O(\delta)$ for exactly one $j\in\{1,2\}$ and exactly one $j\in\{3,4\}$. We consider three sub-cases depending on whether we have successful reconstruction, moderate failure, or severe failure at $\eta_{4}$. 
  
        First, note that if $\eta_4\ge 1-O(\delta)$, then by Assumption \ref{assumption1} and Claim \ref{claim:strongopposites.}, we get 
\begin{align}\label{eq:xi_5_lower_bd_011}
        \xi_5 = \htheta_4 q(\eta_4,\xi_4) = (1-O(\delta)) \, q(1-O(\delta), -1+\Omega(\delta)) \ge -\kappa \textup{ for some }\kappa\in(0,1).
		\end{align} Applying Claim \ref{claim:infxy} gives $(1+\xi_5\eta_5)^2 \ge (1-\kappa)^2$. 
  
  Similarly, if there is a moderate failure at $\eta_4$ (i.e., in particular, $\eta_4 \ge -c_{\ref{eqn:antiReconstruction}}$), then by our assumption 
        $\eta_{3} \ge 1-O(\delta)$
        and so by the same argument as in \eqref{eq:xi_5_lower_bd_011}, 
        $\xi_4 = \hat{\theta}_{3} q(\eta_{3},\xi_{3}) \ge -\kappa$ for some (perhaps different) $\kappa\in(0,1)$. Therefore, by Claim \ref{claim:qRecurseClaim} there exists a $\kappa'\in(0,1)$ such that 
\begin{equation}\nonumber\label{eqn:moderateEta4}
			\xi_5 = \htheta_4 q(\eta_4,\xi_4)\ge (1-O(\delta)) \, q(-c_{\ref{eqn:antiReconstruction}}, -\kappa)  \ge -\kappa'.
		\end{equation}  
  By Claim \ref{claim:infxy}, we then obtain that $(1+\xi_5\eta_5)^2 \ge (1-\kappa')^2$ on the subcase where there is a moderate failure. But this implies
  \begin{equation*}
     \frac{1}{(1+\xi_5\eta_5)^2} \mathbf{1}_{[B_1\cap B_2\cap A_{0,1,1} \cap \{\sigma_{y}=1\} \cap \{\eta_4\ge -c_{\ref{eqn:antiReconstruction}}\}]} \le  \max\left(\frac{1}{(1-\kappa)^2}, \frac{1}{(1-\kappa')^2}\right) = O(1).
  \end{equation*} Hence by Proposition \ref{prop:twoTerms}, we see that
		\begin{align}
	\nonumber		& \widetilde{W}_{4} \mathbf{1}_{[B_1\cap B_2\cap A_{0,1,1} \cap \{\sigma_{y}=1\} \cap \{\eta_4\ge -c_{\ref{eqn:antiReconstruction}}\}]} \\
			\nonumber &\qquad \le O(1) 
    \sup_{|\xi_1|\le 1-2c_{\ref{eqn:pHatBounds}}\delta} 
   F(\htheta_3,\eta_3,\eta_4,\xi_3)
    F(\htheta_1,\eta_1,\eta_2,\xi_1) \mathbf{1}_{A_{0,1,1}} = O(\delta^{-2})\mathbf{1}_{A_{0,1,1}}.\label{eqn:moderateEta4.1}
		\end{align}

		The last sub-case to consider is when $\eta_4$ has a severe failure of reconstruction. In this case, our argument proceeds the same as on the event $A_{1,1,0}\cap \{\textup{the flip occurred at }w_4\}$, which we analyzed above. Indeed, in that analysis, we only used the fact that $\eta_4\ge -1+2c_{\ref{eqn:pHatBounds}}\delta$ (which is always true under Assumption \ref{assumption1}) and 
        $\eta_1,\eta_2,\eta_3\ge -c_{\ref{eqn:antiReconstruction}}$ (since there is no more severe failure),  which also holds in this situation. In particular, the analysis leading up to \eqref{eqn:a110part2} remains valid and we obtain
		\begin{align*}
			&
   \widetilde{W}_4\mathbf{1}_{B_1\cap B_2 \cap A_{0,1,1}\cap \{\sigma_{y} = 1\}\cap \{\eta_4\le -c_{\ref{eqn:antiReconstruction}}\}}\\
			&\qquad = \sup_{|\xi_1|\le 1-2c_{\ref{eqn:pHatBounds}}\delta}\frac{1}{(1+\xi_1\tilde{\eta}_1)^2} \prod_{j=1}^4 \frac{(1-\eta_j^2)}{(1+\htheta_j\eta_j\tilde{\eta}_{j+1})^2} \mathbf{1}_{B_1\cap B_2 \cap A_{0,1,1}\cap \{\sigma_{y} = 1\}\cap \{\eta_4\le -c_{\ref{eqn:antiReconstruction}}\}}\\
			&\qquad =  \frac{1}{\Omega(\delta^4)} \frac{O(\delta^2)}{(1-\kappa)^6} \mathbf{1}_{A_{0,1,1}}.
		\end{align*}
		As $A_{0,1,1}$ has probability at most $O(\delta^3)$ by \eqref{eqn:AijkBounds}, the last expression has expectation of order $O(\delta)$. We have established the desired bound \eqref{eqn:tildeW4L1bound1} for all cases and hence completed the proof.
	\end{proof}

	\subsection{Proof of Lemma \ref{lem:5terms}: The $R_r$ term}

     In this section, we prove the bounds on the first two moments of $R_r$ stated in Lemma \ref{lem:5terms}. This will complete the proof of Lemma \ref{lem:5terms}.

	\begin{proof}[\textbf{Proof of Lemma \ref{lem:5terms}: the $R_r$ term}]

        We break the proof into three separate cases $r = 1, 2, 3$. We begin with the expectation of $R_{r}$. 
		
	\medskip\noindent\textbf{Case $r=1$:}	If $r = 1$, we have $R_1 = \frac{1-\eta_0^2}{(1+\xi_0\eta_0)^2}$. Hence,
  \begin{equation}
      R_1 \mathbf{1}_{\{\xi_0\eta_0\ge 0\}} \le 1.\label{eqn:R1boundonexi_0eta_0ge0}
  \end{equation}
  This yields 
    \begin{align}\label{eq:E_R1_first_reduction}
            \nonumber\E_{\param^{*}} \left[R_{1}\right] & 
            \le 1 + \E_{\param^{*}}[R_{1} \mathbf{1}_{[\xi_{0}\eta_{0}<0]}]. 
    \end{align}
    Hence we now focus on analyzing $R_{1}$ on the event $\{\xi_0\eta_0<0\}$. We will decompose this event according to the value of $\sigma_{w_{0}}\eta_{0}$ as
    \begin{align}
       \{ \xi_{0}\eta_{0}<0 \} = E_{1}\cup E_{2} \cup E_{3} \cup E_{4},
    \end{align}
    where 
    \begin{align*}
        E_{1} &:= \{ \xi_{0}\eta_{0}<0,\,  \sigma_{w_{0}}\eta_{0}\ge 1-C_{\ref{eqn:etareconstruct}}\delta \}, \qquad E_{2} :=\{ \xi_{0}\eta_{0}<0,\,  \sigma_{w_{0}}\eta_{0}\in [0,  1-C_{\ref{eqn:etareconstruct}}\delta) \}, \\
        E_{3}&:= \{ \xi_{0}\eta_{0}<0,\,  \sigma_{w_{0}}\eta_{0}\in [-c_{\ref{eqn:antiReconstruction}},  0) \}, \qquad E_{4} := \{ \xi_{0}\eta_{0}<0,\,  \sigma_{w_{0}}\eta_{0}< -c_{\ref{eqn:antiReconstruction}} \}.
    \end{align*}
    Below we will bound $\PR_{\param^{*}}(E_{i})$ and  the value of $R_{1}$ on each $E_{i}$. 
    
    First we claim that 
  \begin{equation}\label{eqn:xieta0}
      \PR(\xi_0\eta_0 < 0) = O(\delta). 
  \end{equation} 
  Indeed note the inclusion 
  \begin{equation*}
      \{\xi_0\eta_0 \ge 0\} \supset \{\sigma_{y_{-1}}\xi_0 \ge 1-C_{\ref{eqn:etareconstruct}}\delta\} \cap \{\sigma_{w_0}\eta_0 \ge 1-C_{\ref{eqn:etareconstruct}} \delta\} \cap \{ \sigma_{w_0} = \sigma_{y_0} = \sigma_{y_{-1}}\}. 
  \end{equation*}
  The three events in the right-hand side above have probability $1-O(\delta)$ by Corollary \ref{cor:eta} and Assumption \ref{assumption1}, and they are all independent by Claim \ref{claim:unsigned_mag}. So \eqref{eqn:xieta0} follows.

  Second, we claim that 
  \begin{equation} \label{eqn:modetaandnofip}
      \PR_{\param^{*}}\left( E_{2}\right) = O(\delta^2).
  \end{equation}    
  To see this, note if $\xi_0\eta_0<0$ but $\sigma_{w_0} \eta_{0}\ge 0$ then $\sigma_{w_0} \xi_0  <0$. Since
  \begin{equation*}
  \sigma_{w_0} \xi_0 = (\sigma_{w_0} \sigma_{y_{-1}}) (\sigma_{y_{-1}} \xi_0) < 0
  \end{equation*} either $\sigma_{w_0} \neq \sigma_{y_{-1}}$ and $\sigma_{y_{-1}} \xi_0>0$  or $\sigma_{w_0} = \sigma_{y_{-1}}$ and $\sigma_{y_{-1}} \xi_{0}<0$. Therefore
  \begin{align*} 
      \PR&\left( \xi_0\eta_0  < 0 \textup{ and } \sigma_{w_0}\eta_0 \in[0,1-C_{\ref{eqn:etareconstruct}}\delta)\right) 
    \\&\le \PR(\sigma_{w_0}\eta_0 \in[0,1-C_{\ref{eqn:etareconstruct}}\delta)\textup{ and }\sigma_{w_0}\neq \sigma_{y_{-1}}) + \PR(\sigma_{w_0}\eta_0 \in[0,1-C_{\ref{eqn:etareconstruct}}\delta )\textup{ and } \sigma_{y_{-1}} \xi_0  <0)\\
      &\le \PR(\sigma_{w_0}\eta_0 \in[0,1-C_{\ref{eqn:etareconstruct}}\delta)) \PR(\sigma_{w_0}\neq \sigma_{y_{-1}})+\PR(\sigma_{w_0}\eta_0 \in[0,1-C_{\ref{eqn:etareconstruct}}\delta)) \PR( \sigma_{y_{-1}}\xi_0 < 0)
  \end{align*}
  where we used the independence in Claim \ref{claim:unsigned_mag}. Equation \eqref{eqn:modetaandnofip} now follows by Corollary \ref{cor:eta}, Thm. \ref{thm:Robust} and Assumption \ref{assumption1}.
 
  Third, we note the following upper bounds on $R_{1}$: 
  \begin{align}
    \label{eqn:R1boundonosgiamwo1}  &R_1 \mathbf{1}_{[\{\sigma_{w_0}\eta_0\ge 1-C_{\ref{eqn:etareconstruct}}\delta\}]}\le \frac{2C_{\ref{eqn:etareconstruct}}\delta}{(2c_{\ref{eqn:pHatBounds}}\delta)^2}  = O(\delta^{-1})\qquad\textup{using a generic bound on the denominator}\\
   \label{eqn:R1boundonosgiamwo2}   &R_1 \le \frac{1}{(2c_{\ref{eqn:pHatBounds}}\delta)^2}  = O(\delta^{-2})\qquad\textup{using a generic bound on the denominator (again)}
  \end{align} and
  \begin{equation}
    \label{eqn:R1boundonosgiamwo3} R_1 \le \frac{1}{(1-c_{\ref{eqn:antiReconstruction}})^2} = O(1) \qquad\textup{ whenever }\sigma_{w_0} \eta_0 \in[-c_{\ref{eqn:antiReconstruction}},0] \textup{ and }\xi_0\eta_0<0.
  \end{equation}
  
  Using the above, we can  further bound the second term in  \eqref{eq:E_R1_first_reduction} as 
    \begin{align} 
            \nonumber\E_{\param^{*}} \left[R_{1}\right] & 
            =  1 
            + \underbrace{O(\delta^{-1})}_{\textup{by \eqref{eqn:R1boundonosgiamwo1}}} \, \PR(E_{1}) + \underbrace{O(\delta^{-2}) \, \PR(E_{4}) + O(\delta^{-2})\, \PR(E_{2})}_{\textup{both by \eqref{eqn:R1boundonosgiamwo2}}}+ \underbrace{O(1)}_{\textup{by \eqref{eqn:R1boundonosgiamwo3}}} \,\PR(E_{3}). 
	\end{align} 
Now note that $\PR(E_{1}\cap E_{3})=O(\delta)$ by \eqref{eqn:xieta0}, $\PR(E_{2})=O(\delta^{2})$ by \eqref{eqn:modetaandnofip}, and $\PR(E_{4})=O(\delta^{2})$ by the severe failure in Cor. \ref{cor:eta}. Hence the right-hand side above is of $O(1)$. This establishes the case $r = 1$.

	\medskip\noindent\textbf{Case $r=2$:}	We now turn to the case $r = 2$, which is relatively simple. Observe that by Proposition \ref{prop:twoTerms},
  \begin{align*}
      R_2 = \prod_{j=0}^1 \frac{1-\eta_j^2}{(1+\xi_j\eta_j)^2} & \le \sup_{|\xi_0|\le 1-2c_{\ref{eqn:pHatBounds}}\delta} \prod_{j=0}^1 \frac{1-\eta_j^2}{(1+\xi_j\eta_j)^2} \\
      &\le  O(1)\mathbf{1}_{[|\eta_1|,|\eta_0|\ge 1-O(\delta)]}+O(\delta^{-1}) \mathbf{1}_{[|\eta_1|< 1- O(\delta) \textup{ or } |\eta_0| <  1-O(\delta)]}.
  \end{align*}
  Hence by taking expectation and using Corollary \ref{cor:eta}, 
  \begin{equation}\label{eqn:supr2}
       \E_{\param^{*}}\left[\sup_{|\xi_0|\le 1-2c_{\ref{eqn:pHatBounds}}\delta} \prod_{j=0}^1 \frac{1-\eta_j^2}{(1+\xi_j\eta_j)^2} \right] = O(1).
  \end{equation}

	\medskip\noindent\textbf{Case $r=3$:}	The case $r =3$ follows from 
  \begin{equation*}
      R_3 = \prod_{j=0}^2 \frac{1-\eta_j^2}{(1+\xi_j\eta_j)^2} \le \left(\sup_{|\xi_1|\le 1-2c_{\ref{eqn:pHatBounds}}\delta} \prod_{j=1}^2 \frac{1-\eta_j^2}{(1+\xi_j\eta_j)^2}  \right) R_1.
  \end{equation*} 
  Indeed, these two terms are independent of each other by the same argument leading to Lemma \ref{lem:independentWj}. Hence taking expectation and combining with \eqref{eqn:supr2} (with an index shift) and $\E_{\param^{*}}[R_{1}]=O(1)$ (the $r=1$ case above), we get $\E_{\param^{*}}[R_{3}]= O(1)$. 

  Finally, it follows from the above analysis that, for all $r=1,2,3$, 
\begin{align*}\label{eqn:suprbound2}
R_r = O(\delta^{-3}). 
\end{align*} 
Therefore the second moment bounds follow from the first moment bounds using \eqref{eqn:genHolder}.
		\end{proof}

		\subsection{Proof of Lemma \ref{lem:k123}}

        Finally in this section, we prove Lemma \ref{lem:k123}, which bounds the first two moments of the entries in the Hessian of the population log-likelihood function for two edges that are at most three edges apart in the tree.

		\begin{proof}[\textbf{Proof of Lemma }\ref{lem:k123}]

        For this proof we will use the notations introduced in Fig. \ref{fig:TREEhess}. For instance, recall that we are analyzing the Hessian of the log-likelihood corresponding to the edges $e=\{x,y\}$ and $f=\{u,v\}$, and we denoted a shortest path between them by the sequence of adjacent nodes $y_{N+1},y_{N},\dots,y_{0},y_{-1}$, where $y_{N+1}=x$, $y_{N}=y$, $y_{0}=u$, and $y_{-1}=v$. Here $N$ is the shortest path distance between $e$ and $f$.

        Recall, using \eqref{eq:def_eta_xi} and \eqref{eq:Hessian_offdiag_bd}, that
        \begin{equation}\label{eqn:tildeRdefiBoundforhess}
            \left|\frac{\partial^2}{\partial \htheta_e \partial \htheta_f} \ell(\hparam,\sigma|_{L})\right| \le \frac{1}{(1+\xi_{N+1}\eta_{N+1})^2} \prod_{j=0}^N \frac{1-\eta_j^2}{(1+\eta_j\xi_j)^2}=:\widetilde{R}_N,
        \end{equation} 
        where $N=0,1,2$. We will prove the desired moment bounds for the random variable in the right-hand side above. 
        We begin by establishing the expectation bounds for three cases depending on $N=0,1,2$.

        \medskip\noindent\textbf{Case $N=0$:} This is the case when the edges $e$ and $f$ share a vertex as $e=\{y_{1},y_{0}\}$ and $f=\{y_{0},y_{-1}\}$. 
        Let us write $g = \{w_0,y_0\}$ as the edge between $y_0 = y = u$ and $w_0$. As there is just a single ``$w$'' vertex, we will write $w = w_0$ (see Fig. \ref{fig:TREEhess} for illustration).
        We begin with some simplifications. First observe that by Lemma \ref{lem:derivative}, we have
        \begin{align}\label{eqn:N=0swap}
            \left|\frac{\partial^2}{\partial \htheta_e \partial \htheta_f} \ell(\hparam,\sigma)\right| \le \begin{cases}  \displaystyle\frac{1- (\htheta_{g} Z_{w})^2}{(1+\htheta_e Z_x q(\htheta_g Z_{w} ,\htheta_{f} Z_v))^2(1+\htheta_f \htheta_g Z_w Z_v)^2}\\
            \displaystyle 
\frac{1- (\htheta_{g} Z_{w})^2}{(1+\htheta_f Z_v q(\htheta_g Z_{w}, \htheta_{g} Z_x))^2(1+\htheta_g \htheta_e Z_w Z_x)^2}
            \end{cases}.
        \end{align} 
        The second inequality follows from the first by the symmetry of the mixed partial derivatives $\frac{\partial^2}{\partial \htheta_e \partial \htheta_f} \ell(\hparam,\sigma)= \frac{\partial^2}{\partial \htheta_f \partial \htheta_e} \ell(\hparam,\sigma|_{L})$. By using the formula for the Hessian in \eqref{eqn:hessianTerms} in Lemma \ref{lem:derivative} and since magnetizations change sign when the spins change sign, 
        we can see that 
\begin{equation}\label{eqn:zposequaldist}
        \left(\left|\frac{\partial^2}{\partial \htheta_e \partial \htheta_f} \ell(\hparam,\sigma|_{L})\right|,(Z_a;a\in T) \right) \overset{d}{=} \left(\left|\frac{\partial^2}{\partial \htheta_e \partial \htheta_f} \ell(\hparam,-\sigma)\right|,(-Z_a;a\in T)\right)
        \end{equation} where $\sigma\sim \PR_{\param^{*}}$. It is therefore sufficient to show 
        \begin{equation*}\label{eqn:Zxpos}
            \E_{\param^{*}}\left[\left|\frac{\partial^2}{\partial \htheta_e \partial \htheta_f} \ell(\hparam,\sigma)\right|\mathbf{1}_{[Z_x\ge 0]}\right] =O(1).
        \end{equation*}

        We decompose the event $\{Z_x\ge 0\}$ as
        \begin{align*}
            \{Z_x\ge 0\} = \{Z_x,Z_w,Z_v\ge 0\} \cup \{Z_x\ge 0, Z_w,Z_v<0\} \cup \{Z_x, Z_w \ge 0 , Z_v< 0\} \cup \{Z_x, Z_v\ge 0, Z_w<0\}.
        \end{align*}
        It is clear that if $a,b\ge 0$ then $q(a,b)\ge 0$. Hence from the first inequality in \eqref{eqn:N=0swap}, 
        \begin{equation*}
            \textup{if } Z_{x},Z_w,Z_v \ge 0 \qquad\textup{ then }\left|\frac{\partial^2}{\partial \htheta_e \partial \htheta_f} \ell(\hparam,\sigma)\right| \le 1.
        \end{equation*}
        Also, it is easy to see that if $Z_{x}\ge 0$ and $Z_w,Z_v<0$ then $\htheta_f\htheta_g Z_wZ_y>0$ and $1+Z_x\htheta_e q(\htheta_g Z_w, \htheta_f Z_v) \ge 1-\htheta_e Z_x$. This and the  first inequality in \eqref{eqn:N=0swap} give
        \begin{equation}\label{eqn:Zx+Zy-Zw-}
            \textup{if } Z_x\ge 0\textup{ and }Z_w<0, Z_v<0\textup{ then } \left|\frac{\partial^2}{\partial \htheta_e \partial \htheta_f} \ell(\hparam,\sigma)\right| \le \frac{1-(\htheta_g Z_w)^2}{(1-\htheta_e Z_x)^2}.
        \end{equation}
        Note that the right-hand side 
        satisfies 
        \begin{equation*}
            \frac{1-(\htheta_g Z_w)^2}{(1-\htheta_e Z_x)^2} \le \mathbf{1}_{[\sigma_x = 1]} \frac{1-(\htheta_g \sigma_w Z_w)^2}{(1-\htheta_e\sigma_x  Z_x)^2} + \mathbf{1}_{[\sigma_x = -1]} \frac{1
            }{(1+\htheta_e\sigma_x  Z_x)^2}.
        \end{equation*}

        Now observe that by Thm. \ref{thm:Robust}, and the independence Claim
        \ref{claim:unsigned_mag}, we have
        \begin{equation*}\label{eqn:sigmax=-1forz}
            \E_{\param^{*}}\left[\mathbf{1}_{[\sigma_x = -1]} \frac{1}{(1+\htheta_e \sigma_x Z_x)^2}\right] \le \frac{1}{2} \left(\frac{1}{(1-c_{\ref{eqn:antiReconstruction}})^2}\PR(\sigma_xZ_x\ge -c_{\ref{eqn:antiReconstruction}}) 
 + \frac{1}{(2c_{\ref{eqn:pHatBounds}}\delta)^2} \PR(\sigma_x Z_x< -c_{\ref{eqn:antiReconstruction}})\right) = O(1) .
        \end{equation*} The factor of $\frac12$ comes from $\PR(\sigma_x = -1) = \frac12$.

        Next, we claim
       \begin{equation}\label{eqn:sigmaxandZv}
           \E\left[\mathbf{1}_{[\sigma_x = 1, Z_w<0, Z_v<0]} (1-(\htheta_g \sigma_w Z_w)^2)\right] = O(\delta^2). 
       \end{equation} 
       If we can show this, then using the generic bound $(1-\htheta_e\sigma_x  Z_x)^{2} \ge (2c_{\ref{eqn:pHatBounds}}\delta)^2$, we get
        \begin{equation*}
           \E\left[\mathbf{1}_{[\sigma_x = 1, Z_w<0, Z_v<0]} \frac{1-(\htheta_g \sigma_w Z_w)^2}{(1-\htheta_e\sigma_x  Z_x)^2}\right] = O(1). 
       \end{equation*} 
     Then by \eqref{eqn:Zx+Zy-Zw-}, we can conclude
       \begin{align*}\label{eqn:Zx+Zy-Zw-expectation}
            \E_{\param^{*}}&\left[ \left|\frac{\partial^2}{\partial \htheta_e \partial \htheta_f} \ell(\hparam,\sigma)\right| \mathbf{1}_{[Z_x\ge 0, Z_w<0, Z_v<0]}\right] = O(1).
        \end{align*}
        
        Let us now establish \eqref{eqn:sigmaxandZv} using independence provided by  Claim \ref{claim:unsigned_mag}. Recall that either $\sigma_y = +1$ or $\sigma_y = -1$. Hence
        \begin{align*}
        \mathbf{1}_{[\sigma_x = 1, Z_w<0, Z_v<0]} (1-(\htheta_g \sigma_w Z_w)^2) \le \mathbf{1}_{[\sigma_y=1, Z_w<0, Z_v<0]}  + \mathbf{1}_{[\sigma_x\sigma_y = -1]} (1-(\htheta_g \sigma_w Z_w)^2).
        \end{align*}
        Using the second bound in Cor. \ref{cor:Robust} to bound the expectation of the first term and using the independence in Claim \ref{claim:unsigned_mag} and Assumption \ref{assumption1} to bound the expectation of the second term, we see
        \begin{align*}
            \E_{\param^{*}} &\left[\mathbf{1}_{[\sigma_x = 1, Z_w<0, Z_v<0]} (1-(\htheta_g \sigma_w Z_w)^2)\right] \le O(\delta^2)+ O(\delta) \E_{\param^{*}} [1-(\htheta_g \sigma_w Z_w)^2] .
        \end{align*}
        The rightmost expectation is easily seen to be $O(\delta)$ by Thm. \ref{thm:Robust}. This establishes \eqref{eqn:sigmaxandZv}.

        One can redo the previous analysis for $\{Z_x\ge 0, Z_w, Z_v<0\}$ on the event $\{Z_x,Z_w\le 0, Z_v>0\}$ using the second representation in \eqref{eqn:N=0swap}, to see that
        \begin{equation*}
            \E_{\param^{*}} \left[ \left|\frac{\partial^2}{\partial \htheta_e \partial \htheta_f} \ell(\hparam,-\sigma|_{L})\right| \mathbf{1}_{[Z_x\le 0, Z_w\le 0, Z_v> 0]}\right]= O(1)
        \end{equation*}
        as well. (The weak and strict inequalities play no major role in the prior analysis.)
        Then by using the equality in distribution in \eqref{eqn:zposequaldist}, it follows that
        \begin{align*}
            \E_{\param^{*}}\left[ \left|\frac{\partial^2}{\partial \htheta_e \partial \htheta_f} \ell(\hparam,\sigma|_{L})\right| \mathbf{1}_{[Z_x\ge 0, Z_w\ge 0, Z_v<0]}\right] &=  
            \E_{\param^{*}} \left[ \left|\frac{\partial^2}{\partial \htheta_e \partial \htheta_f} \ell(\hparam,-\sigma|_{L})\right| \mathbf{1}_{[Z_x\le 0, Z_w\le 0, Z_v> 0]}\right]= O(1)
        \end{align*}

        It remains to show (for the case $N = 0$) that
        \begin{align}\label{eqn:Zx+Zv+Zw-expectation}
            \E_{\param^{*}}&\left[ \left|\frac{\partial^2}{\partial \htheta_e \partial \htheta_f} \ell(\hparam,\sigma|_{L})\right| \mathbf{1}_{[Z_x, Z_v\ge 0, Z_w<0]}\right] = O(1).
        \end{align}
        To do this, we aim to use Proposition \ref{prop:twoTerms}; however,  a straightforward application is a bit difficult because $\eta_1 = \eta_{N+1} = Z_x$ in \eqref{eq:def_eta_xi} and whenever $x$ is a leaf the magnetization $Z_x = \sigma_x\in\{\pm1\}$ a.s. Recall that
        \begin{equation*}
            F(a, b, c,d ) = \frac{(1-c^2)(1-b^2)}{(1+a c \, q(b,d))^2 (1+bd)^2}.
        \end{equation*}
        Therefore,
        \begin{align}\label{eq:F_sqrt_eta0_xi0}
          F(\htheta_e^{1/2}, \eta_0, \htheta_e^{1/2}\eta_1, \xi_0) =   \frac{1-({\htheta_e^{1/2}}\eta_1)^2}{(1+\eta_1\xi_1)^2} \frac{1-\eta_0^2}{(1+\xi_0\eta_0)^2} = \left(1-(\htheta_e^{1/2} \eta_1)^2\right)\widetilde{R}_0,
        \end{align}
        where we have used the identity 
        \begin{equation*}
            1+\eta_1\xi_1 = 1+ \eta_1 \htheta_e q(\xi_0,\eta_0) = 1+({\htheta_e^{1/2}})({\htheta_e^{1/2}} \eta_1 ) q(\eta_0,\xi_0).
        \end{equation*}
        Now since  $\htheta_e^{1/2}\in [1-2C_{\ref{eqn:pHatBounds}}\delta, 1-c_{\ref{eqn:pHatBounds}}\delta]$ (i.e. $\htheta^{1/2}_e = 1-\Theta(\delta)$), $\htheta_e^{1/2}\eta_1,\eta_0\in[-1+c_{\ref{eqn:pHatBounds}}\delta, 1-c_{\ref{eqn:pHatBounds}}\delta]$, and $|\xi_0| = |\htheta_f Z_v |\le 1-2c_{\ref{eqn:pHatBounds}}\delta$ we can apply Proposition \ref{prop:twoTerms}  along with \eqref{eq:F_sqrt_eta0_xi0} as follows:
        \begin{align}\label{eqn:N=0.1}
            \widetilde{R}_0 = \frac{1}{1-({\htheta_e^{1/2}}\eta_1)^2} \left( {\frac{1-({\htheta_e^{1/2}}\eta_1)^2}{(1+\eta_1\xi_1)^2} \frac{1-(\eta_0)^2}{(1+\xi_0\eta_0)^2}} \right)& \le O(\delta^{-1})F(\htheta_e^{1/2}, \eta_0, \htheta_e^{1/2}\eta_1, \xi_0) \\
           \nonumber  &=  \begin{cases}
               O(\delta^{-2}) &: \textup{always}\\
               O(\delta^{-1})&: \textup{if $|\eta_0|,|\eta_1|\ge 1-O(\delta)$}.
            \end{cases}. 
        \end{align}
        In particular, by  \eqref{eqn:tildeRdefiBoundforhess}, almost surely, 
        \begin{align*}
             \left|\frac{\partial^2}{\partial \htheta_e \partial \htheta_f} \ell(\hparam,\sigma|_{L})\right| = O(\delta^{-2}). 
        \end{align*}

        As $\PR(\sigma_x Z_x, \sigma_v Z_v \le 1-O(\delta^2))= O(\delta^2)$ by Claim \ref{claim:unsigned_mag} and Thm. \ref{thm:Robust}, it suffices to suppose that either $\sigma_x Z_x\ge 1-O(\delta^2)$ or $\sigma_vZ_v\ge 1-O(\delta^2)$. However, by using \eqref{eqn:N=0swap}, we can assume without loss of generality that $\sigma_x Z_x\ge 1-O(\delta^2)$.
        Also note that by the second inequality in  Cor. \ref{cor:Robust}, $\PR(\sigma_y = -1, Z_v, Z_x \ge 0) =O(\delta^2)$ and so it suffices to suppose that $\sigma_y = +1$. The last simplification is that 
        by Thm. \ref{thm:Robust}, we can suppose that each $\sigma_x Z_x, \sigma_v Z_v, \sigma_w Z_w \ge -c_{\ref{eqn:antiReconstruction}}$. Therefore \eqref{eqn:Zx+Zv+Zw-expectation} follows once we show
        \begin{equation*}\label{eqn:caseN=0.5}
            \E_{\param^{*}} \left[\left| \frac{\partial^2}{\partial\theta_e\partial \theta_f} \ell(\hparam,\sigma|_{L})\right|\mathbf{1}_{\mathcal{A}} \right] = O(1),
        \end{equation*}
        where
        \begin{equation*}
            \mathcal{A} := \{ \sigma_y = +1, \sigma_xZ_x\ge 1- O(\delta^2), \sigma_w Z_w \ge -c_{\ref{eqn:antiReconstruction}}, \sigma_v Z_v \ge -c_{\ref{eqn:antiReconstruction}}\} \cap \{ Z_x, Z_v \ge 0, Z_w<0\}.
        \end{equation*}

        In fact, we want to simplify a little more. More precisely, we claim that it suffices to 
        show
        \begin{equation}\label{eqn:caseN=0.6}
            \E_{\param^{*}} \left[\left|\frac{\partial^2}{\partial\theta_e\partial \theta_f} \ell(\hparam,\sigma|_{L}) \right| \mathbf{1}_{\mathcal{A}'}\right] = O(1)
        \end{equation} on the smaller event \begin{equation*}
            \mathcal{A}' := \mathcal{A} \cap \{ \sigma_w Z_w\in[-c_{\ref{eqn:antiReconstruction}}, 1-O(\delta^2)),\, \sigma_vZ_v\ge 1-O(\delta^2)\}.
        \end{equation*}
        Indeed, if $\sigma_w Z_w \ge 1-O(\delta^2)$, then $|\eta_0| = \htheta_g |Z_w|\ge 1-O(\delta)$ and so we can use the $O(\delta^{-1})$ bound in \eqref{eqn:N=0.1} to get 
        \begin{equation*}
            \widetilde{R}_0 \mathbf{1}_{\mathcal{A} \cap \{\sigma_w Z_w\ge 1-C_{\ref{eqn:antiReconstruction}}\delta^2\}} \le O(\delta^{-1}) \mathbf{1}_{[\sigma_y=1, Z_w< 0]}.
        \end{equation*} By Cor. \ref{cor:Robust}, the expectation of the right-hand side is of $O(1)$ and so we do not need to consider $\mathcal{A}\cap \{\sigma_wZ_w\ge 1-C_{\ref{eqn:Reconstruct}}\delta^2\}$.
        By Claim \ref{claim:unsigned_mag} and Thm. \ref{thm:Robust}
        \begin{equation*}
            \PR(\mathcal{A}\cap \{\sigma_w Z_w \in[-c_{\ref{eqn:antiReconstruction}}, 1-O(\delta^2)],  \, \sigma_v Z_v< 1-O(\delta^2)\}) = O(\delta^2)
        \end{equation*} 
        and so, using \eqref{eqn:N=0.1}, the expectation of $\widetilde{R}_{0}$ on $\mathcal{A}\cap \{\sigma_v Z_v< 1-O(\delta^2)\}$ is also $O(1)$.
        
         Note that on $\mathcal{A}'$ we have $\sigma_y = +1$ and $Z_w<0$. If $\sigma_w Z_w\in [0,1-O(\delta^2))$ then $\sigma_w = -1$ and there was a flip between $w$ and $y$. By Claim \ref{claim:unsigned_mag}, this flip is independent of the failed reconstruction $\sigma_w Z_w< 1-O(\delta^2)$, and hence $\mathcal{A}'\cap \{\sigma_w Z_w\in[0,1-C_{\ref{eqn:Reconstruct}}\delta^2)\}$ has probability $O(\delta^2)$ by Thm. \ref{thm:Robust}. Finally, on $\mathcal{A}'\cap \{\sigma_wZ_w\in[-c_{\ref{eqn:antiReconstruction}},0]\}$, noting that $Z_{w}Z_{v}\le 0$ and  $Z_{x}\ge 0$ on $\mathcal{A}'$, we have
        \begin{align*}
           &1+\htheta_g Z_w \htheta_f Z_v \ge 1 - |Z_{v}Z_{w}| \ge 1 - |Z_{w}| 
           \ge 1-c_{\ref{eqn:antiReconstruction}}.
    \end{align*} 
    Then by the monotonicity of $q$,  Assumption \ref{assumption1}, and Claim \ref{claim:strongopposites.}, 
    \begin{align*}\label{eq:q_Zw_Zv_lowerbd}
      \nonumber  q(\htheta_g Z_w, \htheta_f Z_v) &\ge q\Big(-\htheta_gc_{\ref{eqn:antiReconstruction}} , \htheta_f (1-C_{\ref{eqn:Reconstruct}}\delta^2)\Big)  \\
        &\ge q\Big(-  c_{\ref{eqn:antiReconstruction}} , (1-O(\delta)) (1-O(\delta^2))\Big) \ge -\kappa
    \end{align*}
    for some constant $\kappa\in(0,1)$. 
    Hence using that $\htheta_{e}Z_{x}\in [0,1]$ on $\mathcal{A}'$, 
    \begin{align*}
       &1+\htheta_e Z_x q(\htheta_g Z_w, \htheta_f Z_v) \ge 1-\kappa
        \end{align*}
    Thus the denominator of the right-hand side of the first bound in \eqref{eqn:N=0swap} is positive and bounded away from 0. 
    Hence,  
        \begin{equation*}
            \left|\frac{\partial^2}{\partial\theta_e\partial \theta_f} \ell(\hparam,\sigma|_{L}) \right| \mathbf{1}_{\mathcal{A}'\cap \{\sigma_wZ_w\in[-c_{\ref{eqn:antiReconstruction}},0]\}} \le \frac{1}{(1-\kappa)^2 (1-c_{\ref{eqn:antiReconstruction}})^2} = O(1)
        \end{equation*}
        is bounded. This establishes \eqref{eqn:caseN=0.6}. This finishes the proof of \eqref{eqn:nearDiag} for $N = 0$.

        \medskip\noindent\textbf{Case $N = 1$:} This case  starts with a similar upper bound as is found in \eqref{eqn:N=0.1}. Recall that $\xi_j$ and $\eta_i$ are defined by (see Figure \ref{fig:tree_N1})
        \begin{align*}
            \eta_i = 
            \begin{cases}
                \htheta_{\{y_i,w_i\}}Z_{w_i} &:i = 0,1\\
                Z_{x} &: i=2
            \end{cases} 
            ,\qquad 
            \xi_i 
            = \begin{cases}
                \htheta_{f} Z_{v} &: i = 0\\
                \htheta_{i-1} \, q(\xi_{i-1} , \eta_{i-1}) &: i = 1,2
            \end{cases}
            ,\qquad 
            \htheta_{i}=  \htheta_{\{y_i, y_{i-1}\}}. 
        \end{align*}
        \begin{figure}[h!]
		\centering
		\includegraphics[width=0.7 \linewidth]{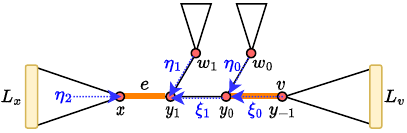}
		\caption{Depiction of the random variables $\xi_{0},\xi_{1},\eta_{0},\eta_{1},\eta_{2}$ associated with the $N=1$ case.}
		\label{fig:tree_N1}
	\end{figure}

    \noindent Also, let $\tilde{\eta}_1,\tilde{\eta}_2$ be as in Lemma \ref{lem:swap}:
        \begin{align*}
            \tilde{\eta}_2 = \eta_2= Z_{x}, \qquad  \tilde{\eta}_1 = q(\htheta_1 \tilde{\eta}_2, \eta_1).
        \end{align*}
        Then 
        \begin{align*}
            \displaystyle \left|\frac{\partial^2}{\partial \htheta_e\partial\htheta_f} \ell(\hparam,\sigma|_{L})\right|&\le \frac{(1-\eta_1^2)(1-\eta_0^2)}{(1+\xi_2\eta_2)^2(1+\xi_1\eta_1)^2(1+\xi_0\eta_0)^2} \qquad \textup{by \eqref{eqn:tildeRdefiBoundforhess}}\\
            &= \frac{(1-\eta_1^2) (1-\eta_0^2)}{(1+\htheta_1 \eta_1\widetilde{\eta_2})^2 (1+\xi_1 \tilde{\eta}_1)^2 (1+\xi_0\eta_0)^2}\qquad \textup{by Lemma \ref{lem:swap} \textbf{(i)}}\\
            &=\frac{(1-\tilde{\eta}_1^2) (1-\eta_0^2)}{(1-(\htheta_1 \widetilde{{\eta}}_2)^2)(1+\xi_1 \tilde{\eta}_1)^2 (1+\xi_0\eta_0)^2} \qquad \textup{ by Lemma \ref{lem:swap} \textbf{(ii)}}\\
            &\le \frac{(1-\tilde{\eta}_1^2) (1-\eta_0^2)}{(2c_{\ref{eqn:pHatBounds}}\delta)(1+\xi_1 \tilde{\eta}_1)^2 (1+\xi_0\eta_0)^2},
        \end{align*} 
        where the last inequality follows from $\htheta_1\in [1-2C_{\ref{eqn:pHatBounds}}\delta, 1-2c_{\ref{eqn:pHatBounds}} \delta]$  (by Assumption \ref{assumption1}) and $\tilde{\eta}_2 = \eta_2\in[-1,1]$. 
        Note that
        \begin{equation*}
            \frac{(1-\tilde{\eta}_1^2) (1-\eta_0^2)}{(1+\xi_1 \tilde{\eta}_1)^2 (1+\xi_0\eta_0)^2} \le  \frac{\left(1-(\htheta_1^{1/2} \tilde{\eta}_1)^2\right) (1-\eta_0^2)}{{(1+\xi_1 \tilde{\eta}_1)^2 (1+\xi_0\eta_0)^2}} = F\left(\htheta_1^{1/2}, \eta_0, \htheta_1^{1/2} \tilde{\eta}_1, \xi_0\right),
        \end{equation*}
        where the equality follows from the same observation as in \eqref{eq:F_sqrt_eta0_xi0}. This shows that 
        \begin{align}\label{eq:hessian_off_bd_N1_pf1}
            \displaystyle \left|\frac{\partial^2}{\partial \htheta_e\partial\htheta_f} \ell(\hparam,\sigma|_{L})\right| \le \frac{1}{2c_{\ref{eqn:pHatBounds}}\delta} F\left(\htheta_1^{1/2}, \eta_0, \htheta_1^{1/2} \tilde{\eta}_1, \xi_0\right).
        \end{align}
        In order to bound the right-hand side above, we will use Proposition \ref{prop:twoTerms} in two cases.

        Define 
        \begin{equation*}
            \mathcal{B} := \{\sigma_{w_0}\eta_0,\sigma_{w_1}\eta_1,\sigma_x\eta_2, \sigma_v \xi_0 \ge 1-O(\delta), \sigma_x = \sigma_{w_1}\}.
        \end{equation*}
        This is the event that the signals at nodes $w_{0},w_{1},x,v$ have successful reconstructions and $\sigma_{x}=\sigma_{w_{1}}$. %(We have covered the case when $\sigma_{x}\ne \sigma_{w_{1}}$)
        Note that $\htheta_1^{1/2} = 1-\Theta(\delta)$ and $|\htheta_1^{1/2}\tilde{\eta}_1|, |\eta_0|, |\xi_0|\le 1-O(\delta)$ under Assumption \ref{assumption1}. (This follows from $\htheta_i=1-\Theta(\delta)$). Note that on $\mathcal{B}$, either $\eta_{1},\eta_{2}\ge 1-O(\delta)$ or $-\eta_{1},-\eta_{2}\ge 1-O(\delta)$, depending on $\sigma_{x}=+1$ or $-1$. 
        Hence by Claim \ref{claim:quickBound1}, we have 
        \begin{equation*}
           \left| \htheta_1^{1/2}\tilde{\eta}_1\right| \mathbf{1}_{\mathcal{B}} \ge (1-O(\delta)) \, q(1-O(\delta), 1-O(\delta)) \ge (1-O(\delta))(1-O(\delta^2)) \ge 1- O(\delta).
        \end{equation*}
         Thus the right-hand side of \eqref{eq:hessian_off_bd_N1_pf1} is $O(\delta^{-1})$ on $\mathcal{B}$ by Proposition \ref{prop:twoTerms} (2); on the complement, it is of $O(\delta^{-2})$ by Proposition \ref{prop:twoTerms} (1). That is, 
        \begin{align}\label{eqn:N=1}
            &\left|\frac{\partial^2}{\partial \htheta_e \partial \htheta_f}\ell(\hparam,\sigma|_{L})\right|\le O(1)\left(\delta^{-1} \mathbf{1}_{\mathcal{B}} + \delta^{-2} \mathbf{1}_{\mathcal{B}^c}\right).
        \end{align}

        We now claim that 
       \begin{align}\label{eqn:N=1.1}
           &\E_{\param^{*}} \left[\left|\frac{\partial^2}{\partial \htheta_e \partial \htheta_f} \ell(\hparam,\sigma|_{L}) \right|\mathbf{1}_{\mathcal{B}}\right] = O(1) \qquad \textup{and} \qquad \E_{\param^{*}} \left[\left|\frac{\partial^2}{\partial \htheta_e \partial \htheta_f} \ell(\hparam,\sigma|_{L}) \right|\mathbf{1}_{\mathcal{B}^c}\right] = O(1).
       \end{align}
       We first consider the %first 
       bound on the event $\mathcal{B}$. Let $\mathcal{B}'$ be the event that $\mathcal{B}$ holds and all involved nodes have the same spin: 
       \begin{equation*}
           \mathcal{B}' := \mathcal{B}\cap \{\sigma_{y_j}= \sigma_x, \sigma_{w_i} = \sigma_x \, \textup{ for }i=0,1\textup{ and }j = -1,0,1\}.
       \end{equation*}
       Then by a union bound and recalling that a flip has probability at most $C_{\ref{eqn:pBounds}}\delta$  by Assumption \ref{assumption1}, 
       \begin{equation*}
           \PR_{\param^{*}}(\mathcal{B}\setminus \mathcal{B}') \le 5 C_{\ref{eqn:pBounds}} \delta.
       \end{equation*} 
       %by the union bound, the probability of a signal flip being at most $C_{\ref{eqn:pBounds}}\delta$ under $\PR_{\param^{*}}$, and Cor. \ref{cor:Robust}. 
       Therefore by \eqref{eqn:N=1}, 
       \begin{align*}
           \E_{\param^{*}} \left[\left|\frac{\partial^2}{\partial \htheta_e \partial \htheta_f} \ell(\hparam,\sigma|_{L}) \right|\mathbf{1}_{\mathcal{B}}\right] \le \E_{\param^{*}} \left[\left|\frac{\partial^2}{\partial \htheta_e \partial \htheta_f} \ell(\hparam,\sigma|_{L}) \right|\mathbf{1}_{\mathcal{B}'}\right]  + O(1).
       \end{align*}
        Now on $\mathcal{B}'$ all signals have the same sign (as all the spins are the same and magnetizations are good approximates for the spins), so %\snote{The event is defined in terms of the $\sigma$s, not $\xi, \eta$. Does this follow from a property of $q$? DC: The spins are all the same from def of $\mathcal{B}'$ and $\sigma \eta = 1-O(\delta)$ by def of $\mathcal{B}$ (sim for $\sigma\xi$.}
        \begin{equation*}
            1+\xi_j\eta_j \ge 1\qquad\textup{ for all }j =0,1,2.
        \end{equation*}
        It follows that all the denominators in $\widetilde{R}_{1}$ (see 
        \eqref{eqn:tildeRdefiBoundforhess}) are at least one (the numerators are always at most one), so we get $\widetilde{R}_1\mathbf{1}_{\mathcal{B}'}\le 1$. It follows that 
        \begin{equation*}
            \E_{\param^{*}} \left[\left|\frac{\partial^2}{\partial \htheta_e \partial \htheta_f} \ell(\hparam,\sigma|_{L}) \right|\mathbf{1}_{\mathcal{B}'}\right] \le 1.
        \end{equation*} 
        This establishes the first bound in  \eqref{eqn:N=1.1}.

        For the second bound in  \eqref{eqn:N=1.1}, recall that $\left|\frac{\partial^2}{\partial \htheta_e \partial \htheta_f}\ell(\hparam,\sigma|_{L})\right|\mathbf{1}_{\mathcal{B}^c}=O(\delta^{-2})$ by \eqref{eqn:N=1}. So we can safely disregard sub-events of $\mathcal{B}^{c}$ of probabilities at most $O(\delta^{2})$. Below are such events: %\snote{What are the $\eta_v$s below? Are they supposed to be $\xi_v$s? Comes up in several displays on this page and the next one. DC: Should be $\xi_0$. Corrected them.}
        \begin{align*}
            &E_1 := \{\textup{at least one of }\sigma_x\eta_2,\sigma_{w_1}\eta_1, \sigma_{w_{0}}\eta_0 , \sigma_v \xi_0 < -c_{\ref{eqn:antiReconstruction}}\}\\
            &E_2 := \{\textup{at least two of }\sigma_x\eta_2,\sigma_{w_1}\eta_1, \sigma_{w_{0}}\eta_0 , \sigma_v \xi_0 \textup{ lie in } [-c_{\ref{eqn:antiReconstruction}}, 1-O(\delta)] \}\\
            &E_3 := \Big\{\textup{there are at least two flips across the edges in }\{e,f, \{y_1,y_0\},\{y_1,w_1\}, \{y_0,w_0\} \}\Big\}\\
            & E_4 := \{\textup{there is at least one flip across the edges in }\{e,f, \{y_1,y_0\},\{y_1,w_1\}, \{y_0,w_0\} \}\\
            &\qquad\qquad \qquad\textup{ and }\textup{at least one of }\sigma_x\eta_2,\sigma_{w_1}\eta_1, \sigma_{w_{0}}\eta_0 , \sigma_v \xi_0 \textup{ lies in } [-c_{\ref{eqn:antiReconstruction}}, 1-O(\delta)] \Big\}
        \end{align*}
        Indeed, by Thm. \ref{thm:Robust}, the independence in Claim \ref{claim:unsigned_mag}, and Assumption \ref{assumption1},
        \begin{align*}
            &\PR_{\param^{*}}(E_1) \le 4 C_{\ref{eqn:antiReconstruction}}\delta^2, &
            &\PR_{\param^{*}}(E_2) \le \binom{4}{2} c_{\ref{eqn:Reconstruct}}^2\delta^2, &
            &\PR_{\param^{*}}(E_3) \le \binom{5}{2} C_{\ref{eqn:pBounds}}^2\delta^2, & &\textup{and}&
            &\PR_{\param^{*}}(E_4) \le 20 c_{\ref{eqn:Reconstruct}} C_{\ref{eqn:pBounds}}\delta^2.
        \end{align*}
        Now define
        \begin{align*}
            \mathcal{B}'' = \mathcal{B}^c \cap (E_1\cup E_2\cup E_3\cup E_4)^c. 
        \end{align*}
        Then by \eqref{eqn:N=1}
        \begin{equation}\label{eqn:N=1.2}
            \E_{\param^{*}} \left[\left|\frac{\partial^2}{\partial \htheta_e \partial \htheta_f} \ell(\hparam,\sigma|_{L}) \right|\mathbf{1}_{\mathcal{B}^c}\right] \le \E_{\param^{*}} \left[\left|\frac{\partial^2}{\partial \htheta_e \partial \htheta_f} \ell(\hparam,\sigma|_{L}) \right|\mathbf{1}_{\mathcal{B}''}\right] + O(1).
        \end{equation}

        Note that $\mathcal{B}''$ requires either a single moderate failure but no severe failures and no flips \textit{or} a single flip and no failures of reconstruction. Therefore, we define the events %\snote{Did you mean to have $\mathcal{B}^c \cap $ on the first line too? DC: Yes}
        \begin{align*}
            \mathcal{B}_0'' &:=\mathcal{B}^c\cap \Big\{\textup{each of }\sigma_x\eta_2,\sigma_{w_1}\eta_1, \sigma_{w_{0}}\eta_0 , \sigma_v \xi_0 \ge -c_{\ref{eqn:antiReconstruction}},\\
        & \qquad\qquad\textup{ and }\textup{there are no flips across the edges in }\{e,f, \{y_1,y_0\},\{y_1,w_1\}, \{y_0,w_0\}\}\Big\},\\
        \mathcal{B}_1''&:= \mathcal{B}^c \cap \Big\{\textup{each of }\sigma_x\eta_2,\sigma_{w_1}\eta_1, \sigma_{w_{0}}\eta_0 , \sigma_v \xi_0 \ge1 -O(\delta),\\
        & \qquad\qquad\textup{ and }\textup{there is exactly one flip across the edges in }\{e,f, \{y_1,y_0\},\{y_1,w_1\}, \{y_0,w_0\}\}\Big\}.
        \end{align*}
        Note $\mathcal{B}''   \subset \mathcal{B}_0''\cup \mathcal{B}_1''$. 
        Observe that on $\mathcal{B}_0''$ it holds by Claim \ref{claim:qRecurseClaim} that 
        \begin{equation*}
            \eta_j,\xi_j\ge -\kappa
        \end{equation*} for some universal constant $\kappa\in(0,1)$. Hence, by Claim \ref{claim:infxy} $(1+\xi_j\eta_j)\ge 1-\kappa$. It follows that all the denominators in $\widetilde{R}_{1}$ (see 
        \eqref{eqn:tildeRdefiBoundforhess}) are at least $(1-\kappa)^2$ (the numerators are always at most one), so we get
        \begin{equation*}
            \widetilde{R}_1\mathbf{1}_{\mathcal{B}_0''} \le \frac{1}{(1-\kappa)^6} = O(1).
        \end{equation*}
        Next, note that 
        \begin{equation*}
            \mathcal{B}^c \cap \{\textup{each of }\sigma_x\eta_2,\sigma_{w_1}\eta_1, \sigma_{w_{0}}\eta_0 , \sigma_v \xi_0 \ge1 -O(\delta)\} \subset \{\sigma_x\neq \sigma_{w_1}\}
        \end{equation*}
        and so, since there is one flip on $\mathcal{B}_1''$, 
        \begin{align*}
            \mathcal{B}_1'' = \Big\{\textup{each of }&\sigma_x\eta_2,\sigma_{w_1}\eta_1, \sigma_{w_{0}}\eta_0 , \sigma_v \xi_0 \ge1 -O(\delta),\, \textup{the only flip is across }e=\{x, y_{1} \} \textup{ or }\{y_1,w_1\}\Big\}.
        \end{align*}
        On the one hand, if the flip is across the edge $e$, then $\xi_0,\eta_0,\eta_1$ are of the same sign and so $(1+\xi_j\eta_j)\ge 1$ for $j=0,1$. Hence
        \begin{align*}
            \widetilde{R}_1 \mathbf{1}_{\mathcal{B}''_1\cap\{\textup{flip across }e\}} &\le \frac{(O(\delta))^2}{(1+\xi_2\eta_2)^2} = \frac{O(\delta^2)}{\Omega(\delta^2)} = O(1).
        \end{align*} 
        On the other hand, if the flip is across $\{y_1,w_1\}$, then $\eta_0,\xi_0$ are the same sign and $\eta_1$ is the opposite sign and so $(1+\xi_0\eta_0)\ge 1$. Suppose (without loss of generality) that $\sigma_{w_1} = -1$ and the remaining spins $\sigma_x = \sigma_y = \sigma_u =\sigma_v = \sigma_{w_0}= +1$. On the event $\mathcal{B}''_{1}$, it follows that $\eta_{2},\eta_{0},\xi_{0}\ge 1-O(\delta)$ and $\eta_{1}\le -1+O(\delta)$. 
        Hence by Claim \ref{claim:quickBound1}, 
        \begin{align*}
        \xi_1 = \htheta_0 q(\xi_0,\eta_0) \ge (1-O(\delta))(1-O(\delta^2)) \ge 1-O(\delta).
        \end{align*}
         By Claim \ref{claim:strongopposites.} there is a constant $\kappa\in(0,1)$ such that
        \begin{equation*}\label{eqn:N=1.3}
            \xi_2 = \htheta_1 q(\eta_1,\xi_1) \in[-\kappa,\kappa].
        \end{equation*} 
        Hence, by Claim \ref{claim:infxy} and our generic lower bound, \begin{equation*}
            (1+\eta_2\xi_2)\ge 1-\kappa\qquad \textup{and}\qquad (1+\xi_1\eta_1) = \Omega(\delta).
        \end{equation*}
        Thus we get 
         \begin{align*}
            \widetilde{R}_1 \mathbf{1}_{\mathcal{B}''_1\cap\{\textup{flip across $\{ y_{1}, w_{1} \}$}\}} &\le \frac{O(\delta^2)}{\Omega(\delta^2) (1-\kappa)^4} = O(1). 
        \end{align*}
        Combining the above bounds, we deduce 
        \begin{equation*}
    \widetilde{R}_1\mathbf{1}_{\mathcal{B}''}  \le \widetilde{R}_1\mathbf{1}_{\mathcal{B}''_{0}}  +  \widetilde{R}_1 \mathbf{1}_{\mathcal{B}''_1\cap\{\textup{flip across }e\}} +  \widetilde{R}_1 \mathbf{1}_{\mathcal{B}''_1\cap\{\textup{flip across $\{ y_{1}, w_{1} \}$}\}}  = O(1).
        \end{equation*}
        Using \eqref{eqn:N=1.2}, the above bounds establish the second bound in \eqref{eqn:N=1} and finishes the case $N = 1$.

        \medskip\noindent\textbf{Case $N = 2$:} Similar to the bounds obtained in \eqref{eqn:N=0.1} and \eqref{eqn:N=1} we can use Lemma \ref{lem:swap} and Proposition \ref{prop:twoTerms} to get the generic upper bound 
        \begin{equation}\label{eqn:N=2}
            \left|\frac{\partial^2}{\partial \htheta_e \partial \htheta_f} \ell(\hparam; \sigma|_{L})\right| = O(\delta^{-3}).
            %\frac{C_{\ref{eqn:N=2}}}{\delta^3}
        \end{equation} 
        More precisely, using $\sqrt{\htheta_3} \eta_3, \eta_2,\eta_1,\eta_0, \xi_0\in[-1+O(\delta), 1-O(\delta)]$ a.s. by Cor. \ref{cor:Robust} and the fact that $\sqrt{\htheta_3}= 1-\Theta(\delta)$ under Assumption \ref{assumption1}, we get
        \begin{align}
           \nonumber  \left|\frac{\partial^2}{\partial \htheta_e \partial \htheta_f} \ell(\hparam; \sigma|_{L})\right| &= \frac{1}{1-(\htheta_3^{1/2}\eta_3)^2} \frac{(1-(\htheta_3^{1/2}\eta_3)^2)(1-\eta_2^2)(1-\eta_1^2)(1-\eta_0^2)}{(1+\xi_3\eta_3)^2 (1+\xi_2\eta_2)^2 (1+\xi_1\eta_1)^2(1+\xi_0\eta_0)^2}\\
            \label{eqn:N=2.2}  &=O(\delta^{-1}) F(\htheta_3^{1/2},\eta_2,\htheta_3^{1/2}\eta_3,\xi_2)\cdot F(\htheta_1,\eta_0,\eta_1,\xi_0).
        \end{align} The bound \eqref{eqn:N=2} follows from Proposition \ref{prop:twoTerms} (1). 
        
       We will use similar notions of failures and flips as we introduced in Section \ref{sec:Aijk}. Here we need to modify them slightly to incorporate the first and the last nodes, $x$ and $v$. Namely, 
       the definitions of moderate failures and severe failures at vertices $w_2,w_1,w_0$ (at the beginning of Sec. \ref{sec:Aijk})
       remain the same. In addition, we introduce the similar failure events for the first and the last nodes $x$ and $v$: 
       \begin{enumerate}
           \item There is a moderate (resp. severe) failure of reconstruction at $y_{-1} = v$ if $\sigma_{v}\xi_0\in[-c_{\ref{eqn:antiReconstruction}}, 1-C_{\ref{eqn:etareconstruct}}\delta]$ (resp. $\sigma_v \xi_0 < -c_{\ref{eqn:antiReconstruction}}$). 
           \item There is a moderate (resp. severe) failure of reconstruction at $x$ if $\sigma_x\htheta_3^{1/2} \eta_3 \in[-c_{\ref{eqn:antiReconstruction}}, 1-C_{\ref{eqn:etareconstruct}}\delta]$ (resp. $\sigma_x \htheta_3^{1/2} \eta_3 < -c_{\ref{eqn:antiReconstruction}}$).
       \end{enumerate}       
       Note that $\eta_3 = Z_x$ so a moderate failure at $x$ has probability $O(\delta)$ by Thm. \ref{thm:Robust} and a severe failure at $x$ has probability $O(\delta^2)$. We will also say that there is a \textit{flip} if for two neighboring vertices $a,b\in \{x = y_3, y = y_2,y_1, u=y_0,v =  y_{-1}, w_0, w_1, w_2\}$ it holds $\sigma_a \neq \sigma_b$. Similar to \eqref{eqn:AijkDef}, define 
        \begin{equation*}\label{eqn:AijkDef2}
            \tilde{A}_{i,j,k} =\left\{\begin{array}{c}
            \textup{there are }i\textup{ many flips and }\\
            j\textup{ many moderate failures and }k\textup{ many severe failures}\\
            \textup{ at } y_{-1}, y_3, w_2,w_1,w_0
            \end{array}\right\}.
        \end{equation*}
        Similarly to \eqref{eqn:AijkBounds}, we can use Claim \ref{claim:unsigned_mag} and Thm. \ref{thm:Robust} to see that
        \begin{equation}\label{eqn:aijkbounds2}
            \PR_{\param^{*}}\left(\tilde{A}_{i,j,k}\right) = O(\delta^{i+j+2k}).
        \end{equation}
        By \eqref{eqn:N=2}, it suffices to show that on $\mathcal{C} = \bigcup_{i+j+2k\le 2} \tilde{A}_{i,j,k}$, %that there is some constant $C_{\ref{eqn:N=2.1}}$
        \begin{equation*}\label{eqn:N=2.1}
            \E_{\param^{*}}\left[ \left|\frac{\partial^2}{\partial \htheta_e \partial \htheta_f} \ell(\hparam; \sigma|_{L})\right| \mathbf{1}_{\mathcal{C}}\right] = O(1).
        \end{equation*}
       Below we show this by decomposing the event $\mathcal{C}$ into three cases. 

        \textbf{On the event $\bigcup_{j=0}^2 \tilde{A}_{0,j,0}$:} %As in thLemma \ref{lem:4terms}, 
        This is the event that there are no flips and at most two moderate failures of reconstruction. Without loss of generality, we can suppose that $\sigma_{x} = 1$. Note that on this event we have a.s.
        \begin{equation*}
            \xi_0,\eta_j \ge -c_{\ref{eqn:antiReconstruction}}\textup{ for all }j = 0,1,2,3
        \end{equation*}
        and therefore by Claim \ref{claim:qRecurseClaim}
        \begin{equation*}
            \xi_1,\xi_2,\xi_3 \ge -\kappa
        \end{equation*} for some universal constant $\kappa\in (0,1)$.
        Hence $1+\xi_j\eta_j \ge 1-\kappa$ for each $j = 0,1,2,3$ by an application of Claim \ref{claim:infxy}. As $1-\eta_j^2\le 1$ we see using \eqref{eqn:tildeRdefiBoundforhess}
        \begin{equation*}
        \left|\frac{\partial^2}{\partial \htheta_e \partial \htheta_f} \ell(\hparam, \sigma|_{L})\right| \mathbf{1}_{\bigcup_{j=0}^2 \tilde{A}_{0,j,0}} \le \frac{1}{(1-\kappa)^8} = O(1).
        \end{equation*} 
        The expectation bound follows.

        \textbf{On the event $\bigcup_{i=1}^2 \tilde{A}_{i,0,0}$:} This is the event that there are one or two flips and there is no failure of reconstruction. 
        %On this event we have no failures of reconstruction 
        On this event, it holds that $|\htheta_3^{1/2}\eta_3|,|\eta_i|\ge 1-O(\delta)$ 
        for each $i=0,1,2.$ Then using \eqref{eqn:N=2.2} and Proposition \ref{prop:twoTerms}(2) with $\xi_0,\xi_2 \in[-1+\Theta(\delta),1-\Theta(\delta)]$ a.s.,  we see
        \begin{equation*}
            \left|\frac{\partial^2}{\partial \htheta_e \partial \htheta_f} \ell(\hparam, \sigma|_{L})\right| \mathbf{1}_{\bigcup_{i=1}^2 \tilde{A}_{i,0,0}} = O(\delta^{-1}) \mathbf{1}_{\bigcup_{i=1}^2 \tilde{A}_{i,0,0}}
        \end{equation*} 
        Using equation \eqref{eqn:aijkbounds2}, we see that
        \begin{equation*}\label{eqn:N=2.3}
            \E_{\param^{*}}\left[\left|\frac{\partial^2}{\partial \htheta_e \partial \htheta_f} \ell(\hparam, \sigma|_{L})\right| \mathbf{1}_{\cup_{i=1}^2 \tilde{A}_{i,0,0}} \right] = O(1).
        \end{equation*}

        \textbf{On the event $\tilde{A}_{0,0,1} \cup \tilde{A}_{1,1,0}$:} This is the event that there is either a single severe failure of reconstruction or a moderate failure paired with a flip. Note that this event has probability $O(\delta^2)$ by \eqref{eqn:aijkbounds2}. Moreover, as there is only one failure of reconstruction, we know that either 
        \begin{equation*}
            |\htheta_3^{1/2} \eta_3|, |\eta_2|\ge 1-O(\delta)\qquad\textup{ or }\qquad |\eta_1|,|\eta_0|\ge 1-O(\delta)
        \end{equation*} by the pigeonhole principle. 
        Hence, using both parts of Proposition \ref{prop:twoTerms} and \eqref{eqn:N=2.2},
        \begin{equation*}
            \left|\frac{\partial^2}{\partial \htheta_e \partial \htheta_f} \ell(\hparam, \sigma|_{L})\right| \mathbf{1}_{\tilde{A}_{0,0,1} \cup \tilde{A}_{1,1,0}} \le O(\delta^{-1}) O(1)O(\delta^{-1}) \mathbf{1}_{\tilde{A}_{0,0,1} \cup \tilde{A}_{1,1,0}} = O(\delta^{-1})\mathbf{1}_{\tilde{A}_{0,0,1} \cup \tilde{A}_{1,1,0}}
        \end{equation*}
        and so 
        \begin{equation*}            \E_{\param^{*}}\left[\left|\frac{\partial^2}{\partial \htheta_e \partial \htheta_f} \ell(\hparam, \sigma|_{L})\right| \mathbf{1}_{\tilde{A}_{0,0,1} \cup \tilde{A}_{1,1,0}} \right] \le O(\delta^{-2}) \PR_{\param^{*}}({\tilde{A}_{0,0,1} \cup \tilde{A}_{1,1,0}})  = O(1). 
        \end{equation*}
        This exhausts all of $\mathcal{C}$ and hence proves the result.
    \end{proof}

	\section*{Acknowledgements} 
	HL was partially supported by NSF grant DMS-2206296. DC and SR were partially supported by the Institute for Foundations of Data Science (IFDS) through NSF grant DMS-2023239 (TRIPODS Phase II). The paper is based upon work supported by the NSF under grant DMS-1929284 while one of the authors (SR) was in residence at the Institute for Computational and Experimental Research in Mathematics (ICERM) in Providence, RI, during the Theory, Methods, and Applications of Quantitative Phylogenomics semester program. SR was also supported by NSF grant DMS-2308495, as well as a Van Vleck Research Professor Award and a Vilas Distinguished Achievement Professorship.		

\providecommand{\bysame}{\leavevmode\hbox to3em{\hrulefill}\thinspace}
\providecommand{\etalchar}[1]{$^{#1}$}

% BibTex style file: imsart-number.bst, 2017-11-03
% Default style options (sort=1,type=number).
% Used options (sort=1,type=number).

\end{document}